\documentclass[12pt,a4paper]{amsart}
\usepackage[T1]{fontenc}
\usepackage[czech,english]{babel}
\usepackage{amssymb,eucal}
\usepackage{mathrsfs}
\usepackage[all,cmtip]{xy}

\calclayout

\newcommand{\+}{\protect\nobreakdash-}
\renewcommand{\:}{\colon}

\newcommand{\rarrow}{\longrightarrow}
\newcommand{\larrow}{\longleftarrow}
\newcommand{\ot}{\otimes}
\newcommand{\ocn}{\odot}
\newcommand{\tim}{\rightthreetimes}
\newcommand{\st}{\star}

\newcommand{\lrarrow}{\mskip.5\thinmuskip\relbar\joinrel\relbar\joinrel
 \rightarrow\mskip.5\thinmuskip\relax}
\newcommand{\llarrow}{\mskip.5\thinmuskip\leftarrow\joinrel\relbar
 \joinrel\relbar\mskip.5\thinmuskip\relax}
\newcommand{\bu}{{\text{\smaller\smaller$\scriptstyle\bullet$}}}

\newcommand{\plim}
 {\mathop{\text{\normalfont``$\varprojlim$''\!\!}}\nolimits}
\newcommand{\proocn}{\mathbin{\widehat\odot}}

\DeclareMathOperator{\Hom}{Hom}
\DeclareMathOperator{\Ext}{Ext}
\DeclareMathOperator{\Tot}{Tot}
\DeclareMathOperator{\coker}{coker}
\DeclareMathOperator{\im}{im}
\DeclareMathOperator{\Prod}{\mathsf{Prod}}
\DeclareMathOperator{\hocolim}{hocolim}
\DeclareMathOperator{\holim}{holim}

\newcommand{\Vect}{{\operatorname{\mathsf{--Vect}}}}
\newcommand{\Modl}{{\operatorname{\mathsf{--Mod}}}}
\newcommand{\Modr}{{\operatorname{\mathsf{Mod--}}}}
\newcommand{\Modrcs}{{\operatorname{\mathsf{Mod_{cs}--}}}}
\newcommand{\Modrcsom}{{\operatorname{\mathsf{Mod_{cs}^\omega--}}}}
\newcommand{\coh}{{\operatorname{\mathsf{coh--}}}}
\newcommand{\Cohpro}{{\operatorname{\mathsf{Coh_{pro}--}}}}
\newcommand{\Discr}{{\operatorname{\mathsf{Discr--}}}}
\newcommand{\Discrinj}{{\operatorname{\mathsf{Discr^\inj--}}}}
\newcommand{\Contra}{{\operatorname{\mathsf{--Contra}}}}

\newcommand{\Sets}{\mathsf{Sets}}

\newcommand{\Hot}{\mathsf{Hot}}
\newcommand{\Ac}{\mathsf{Ac}}
\newcommand{\Ab}{\mathsf{Ab}}
\newcommand{\Pro}{\mathsf{Pro}}
\newcommand{\SPro}{\mathsf{SPro}}

\newcommand{\inj}{\mathsf{inj}}
\newcommand{\proj}{\mathsf{proj}}
\renewcommand{\flat}{\mathsf{flat}}

\newcommand{\bctr}{\mathsf{bctr}}
\newcommand{\bco}{\mathsf{bco}}
\newcommand{\bb}{\mathsf{b}}
\newcommand{\fp}{\mathsf{fp}}
\newcommand{\cont}{\mathrm{cont}}
\newcommand{\pro}{\mathsf{pro}}

\newcommand{\sop}{{\mathsf{op}}}
\newcommand{\id}{{\mathrm{id}}}
\newcommand{\ev}{{\mathrm{ev}}}

\newcommand{\M}{\mathcal M}
\newcommand{\N}{\mathcal N}
\newcommand{\cA}{\mathcal A}
\newcommand{\cB}{\mathcal B}
\newcommand{\cU}{\mathcal U}
\newcommand{\cS}{\mathcal S}

\newcommand{\R}{\mathfrak R}
\newcommand{\fP}{\mathfrak P}
\newcommand{\fQ}{\mathfrak Q}
\newcommand{\fF}{\mathfrak F}
\newcommand{\fG}{\mathfrak G}
\newcommand{\fH}{\mathfrak H}
\newcommand{\fI}{\mathfrak I}
\newcommand{\fJ}{\mathfrak J}
\newcommand{\fM}{\mathfrak M}

\newcommand{\fA}{\mathfrak A}
\newcommand{\fB}{\mathfrak B}

\newcommand{\fU}{\mathfrak U}
\newcommand{\fV}{\mathfrak V}

\newcommand{\fZ}{\mathfrak Z}

\newcommand{\sA}{\mathsf A}
\newcommand{\sB}{\mathsf B} 
\newcommand{\sC}{\mathsf C}
\newcommand{\sD}{\mathsf D}
\newcommand{\sE}{\mathsf E}

\newcommand{\sK}{\mathsf K}

\newcommand{\sP}{\mathsf P}
\newcommand{\sS}{\mathsf S}
\newcommand{\sT}{\mathsf T}
\newcommand{\sU}{\mathsf U}

\newcommand{\rD}{\mathscr D}
\newcommand{\rE}{\mathscr E}
\newcommand{\rL}{\mathscr L}
\newcommand{\rM}{\mathscr M}
\newcommand{\rN}{\mathscr N}
\newcommand{\rP}{\mathscr P}
\newcommand{\rU}{\mathscr U}

\newcommand{\bB}{\boldsymbol B}
\newcommand{\bC}{\boldsymbol C}
\newcommand{\bD}{\boldsymbol D}
\newcommand{\bE}{\boldsymbol E}
\newcommand{\bF}{\boldsymbol F}
\newcommand{\bK}{\boldsymbol K}
\newcommand{\bL}{\boldsymbol L}

\newcommand{\bP}{\boldsymbol P}

\newcommand{\bbe}{\boldsymbol e}
\newcommand{\bbf}{\boldsymbol f}
\newcommand{\bbg}{\boldsymbol g}
\newcommand{\bbh}{\boldsymbol h}
\newcommand{\bbp}{\boldsymbol p}

\newcommand{\bcM}{\boldsymbol{\mathcal M}}
\newcommand{\bcN}{\boldsymbol{\mathcal N}}

\newcommand{\boZ}{\mathbb Z}
\newcommand{\boQ}{\mathbb Q}

\newcommand{\Section}[1]{\bigskip\section{#1}\medskip}
\theoremstyle{plain}
\newtheorem{thm}{Theorem}[section]

\newtheorem{lem}[thm]{Lemma}
\newtheorem{prop}[thm]{Proposition}
\newtheorem{cor}[thm]{Corollary}
\theoremstyle{definition}
\newtheorem{ex}[thm]{Example}

\newtheorem{rem}[thm]{Remark}

\begin{document}

\title{Compact generators of the contraderived category
of contramodules}

\author{Leonid Positselski}

\address{Leonid Positselski, Institute of Mathematics, Czech Academy
of Sciences \\ \v Zitn\'a~25, 115~67 Praha~1 \\ Czech Republic}

\email{positselski@math.cas.cz}

\author{Jan \v S\v tov\'\i\v cek}

\address{Jan {\v S}{\v{t}}ov{\'{\i}}{\v{c}}ek, Charles University,
Faculty of Mathematics and Physics, Department of Algebra,
Sokolovsk\'a 83, 186 75 Praha, Czech Republic}

\email{stovicek@karlin.mff.cuni.cz}

\begin{abstract}
 We consider the contraderived category of left contramodules over
a right linear topological ring $\R$ with a countable base of
neighborhoods of zero.
 Equivalently, this is the homotopy category of unbounded complexes of
projective left $\R$\+contramodules.
 Assuming that the abelian category of discrete right $\R$\+modules is
locally coherent, we show that the contraderived category of left
$\R$\+contramodules is compactly generated, and describe its full
subcategory of compact objects as the opposite category to the bounded
derived category of finitely presentable discrete right $\R$\+modules.
 Under the same assumptions, we also prove the flat and projective
periodicity theorem for $\R$\+contramodules.
\end{abstract}

\maketitle

\tableofcontents

\section*{Introduction}
\medskip

 The study of the homotopy category $\Hot(R\Modl_\proj)$ of unbounded
complexes of projective modules over an associative ring $R$ goes back
to the paper of J\o rgensen~\cite{Jor}.
 Under certain assumptions including right coherence of a ring $R$,
it was shown in~\cite[Theorems~2.4 and~3.2]{Jor} that the triangulated
category $\Hot(R\Modl_\proj)$ is compactly generated, and its full
subcategory of compact objects is anti-equivalent to the bounded
derived category of finitely presentable right $R$\+modules.
 A full generality in this context was achieved in Neeman's
paper~\cite[Proposition~7.4]{Neem}, where the same assertion was proved
for all right coherent rings~$R$.
 This result was further generalized to the homotopy category of
graded-projective left CDG-modules over a curved DG\+ring $(B^*,d,h)$
with a graded right coherent underlying graded ring $B^*$ in the present
authors' preprint~\cite[Theorems~6.14 and~6.16]{PS7}.

 The contemporary point of view emphasizes the utility of considering
the triangulated category $\Hot(R\Modl_\proj)$ also as a quotient
category (rather than only a full subcategory) of the homotopy category
of arbitrary complexes of $R$\+modules $\Hot(R\Modl)$.
 A complex of left $R$\+modules $B^\bu$ is said to be 
\emph{contraacyclic} (\emph{in the sense of
Becker}~\cite[Propositions~1.3.8(1)]{Bec}) if, for any complex of
projective left $R$\+modules $P^\bu$, the complex of abelian groups
$\Hom_R(P^\bu,B^\bu)$ is acyclic.
 The triangulated Verdier quotient category $\sD^\bctr(R\Modl)=
\Hot(R\Modl)/\Ac^\bctr(R\Modl)$ of the homotopy category
$\Hot(R\Modl)$ by the thick subcategory $\Ac^\bctr(R\Modl)$ of all
contraacyclic complexes of left $R$\+modules is called
the \emph{Becker contraderived category} (of the abelian
category $R\Modl$).
 For any associative ring $R$, the composition of functors
$\Hot(R\Modl_\proj)\rarrow\Hot(R\Modl)\rarrow\sD^\bctr(R\Modl)$
is a triangulated equivalence
$$
 \Hot(R\Modl_\proj)\simeq\sD^\bctr(R\Modl)
$$
\cite[Proposition~8.1]{Neem}, \cite[Proposition~1.3.6(1)]{Bec}.

 In J\o rgensen's paper, the proofs of~\cite[Theorems~2.4 and~3.2]{Jor}
depended on the assumption that all flat left $R$\+modules have finite
projective dimensions.
 One of the advances in Neeman's paper~\cite{Neem}, which 
(essentially) made it possible to get rid of this unnecessary
assumption, was a proof of the \emph{flat and projective periodicity
theorem}.
 The latter can be interpreted by saying that, in certain respect,
flat modules always behave as if they had finite projective dimension
(even when they don't).
 Specifically, the periodicity theorem claims that in any acyclic
complex of projective modules with flat modules of cocycles (over
any associative ring~$R$), the modules of cocycles are actually
projective~\cite[Remark~2.15 and Theorem~8.6]{Neem}.
 This result is actually an equivalent reformulation of an earlier
theorem of Benson and Goodearl~\cite[Theorem~2.5]{BG}, as explained
in the paper~\cite[Proposition~7.6]{CH}.

 The flat and projective periodicity theorem provides yet another
equivalent definition of the contraderived category of $R$\+modules. 
 Specifically, the triangulated category $\Hot(R\Modl_\proj)\simeq
\sD^\bctr(R\Modl)$ is also equivalent to the \emph{derived category
of the exact category of flat $R$\+modules},
$$
 \Hot(R\Modl_\proj)\simeq\sD(R\Modl_\flat)\simeq\sD^\bctr(R\Modl).
$$
 More precisely, the claim here is that a complex of flat
$R$\+modules is contraacyclic if and only if it is acyclic with
flat modules of cocycles.
 This result, which can be called ``the full form of the flat
and projective periodicity theorem'', is also a part
of~\cite[Theorem~8.6]{Neem}.
 For a generalization of the flat and projective periodicity theorem
to curved DG\+modules over curved DG\+rings, see~\cite[Corollary~4.12,
Theorem~4.13, and Theorem~5.5]{PS7}.

 The aim of the present paper is to generalize the results mentioned
above, and first of all the description of compact generators, from
discrete rings $R$ to certain topological rings~$\R$.
 For any complete, separated topological associative ring $\R$ with
a base of neighborhoods of zero formed by open right ideals,
there is an abelian category of \emph{left\/ $\R$\+contramodules}
$\R\Contra$.
 The abelian category $\R\Contra$ is always locally presentable and
has enough projective objects.
 So one can consider the homotopy category $\Hot(\R\Contra_\proj)$ and
the Becker contraderived category $\sD^\bctr(\R\Contra)=
\Hot(\R\Contra)/\Ac^\bctr(\R\Contra)$.
 The result of~\cite[Corollary~7.4]{PS4}, which holds for any
locally presentable abelian category with enough projective objects,
claims that the composition of the obvious functors
$\Hot(\R\Contra_\proj)\rarrow\Hot(\R\Contra)\rarrow\sD^\bctr(\R\Contra)$
is a triangulated equivalence, just as in the module case,
$$
 \Hot(\R\Contra_\proj)\simeq\sD^\bctr(\R\Contra).
$$

 Periodicity theorems for contramodule categories were proved in
the first-named author's papers~\cite{Pflcc,Pbc} under the assumption
that $\R$ has a \emph{countable} base of neighborhoods of zero
consisting of open \emph{two-sided} ideals.
 The assumptions in the present paper are less restrictive: we only 
assume a countable base of neighborhoods of zero consisting of open
\emph{right} ideals.
 Just as in Neeman's paper~\cite{Neem}, in order to prove the compact
generation we need to assume a suitable version of right coherence
condition on~$\R$.
 (See~\cite[Example~7.16]{Neem} for a counterexample showing that
this condition cannot be simply dropped in the module case.)
 Since our proof of the flat/projective periodicity for contramodules
in this paper is based on the construction of compact generators, it
also depends on a right coherence assumption.

 The appropriate version of right coherence condition on a topological
ring $\R$ with a base of neighborhoods of zero formed by open right
ideals is called the \emph{topological right coherence}.
 This concept goes back to the paper of Roos~\cite[Section~2 and
Definition~4.3]{Roo}.
 Specifically, a topological ring $\R$ is called topologically right
coherent if the Grothendieck abelian category of discrete right
$\R$\+modules $\Discr\R$ is locally coherent (i.~e., it has a set of
coherent generators, or equivalently, it is locally finitely
presentable with an abelian full subcategory of finitely presentable
objects).
 Let us \emph{warn} the reader that an object of $\Discr\R$ is
finitely generated if and only if it is finitely generated
in $\Modr\R$, but a finitely presentable/coherent object of
$\Discr\R$ \emph{need not} be finitely presentable in $\Modr\R$.
 A further discussion can be found in~\cite[Section~13]{PS3}
and~\cite[Section~8.2]{PS5}.

 Dually to the discussion of the homotopy categories of unbounded
complexes of projective modules/objects above, for any Grothendieck
abelian category $\sA$ one can consider the homotopy category of
unbounded complexes of injective objects $\Hot(\sA^\inj)$.
 A complex $A^\bu$ in $\sA$ is said to be \emph{coacyclic}
(in the sense of Becker~\cite[Proposition~1.3.8(2)]{Bec}) if, for any
complex of injective objects $J^\bu$ in $\sA$, the complex of abelian
groups $\Hom_\sA(A^\bu,J^\bu)$ is acyclic.
 The triangulated Verdier quotient category $\sD^\bco(\sA)=
\Hot(\sA)/\Ac^\bco(\sA)$ of the homotopy category $\Hot(\sA)$ by
the thick subcategory $\Ac^\bco(\sA)\subset\Hot(\sA)$ of all coacyclic
complexes is called the \emph{Becker coderived category} of~$\sA$.
 For any Grothendieck category $\sA$, the composition of functors
$\Hot(\sA^\inj)\rarrow\Hot(\sA)\rarrow\sD^\bco(\sA)$ is a triangulated
equivalence
$$
 \Hot(\sA^\inj)\simeq\sD^\bco(\sA)
$$
\cite[Theorem~2.13]{Neem2}, \cite[Corollary~5.13]{Kra3},
\cite[Theorem~4.2]{Gil4}, \cite[Corollary~9.5]{PS4}.

 For a locally Noetherian Grothendieck category $\sA$, it was shown
in the paper of Krause~\cite[Proposition~2.3]{Kra} that the homotopy
category $\Hot(\sA^\inj)$ is compactly generated, and its full
subcategory of compact objects is equivalent to the bounded derived
category of finitely generated objects in~$\sA$.
 This result was generalized to locally coherent Grothendieck
categories $\sA$ in the second-named author's
preprint~\cite[Corollary~6.13]{Sto}.
 The assertion was similar: the homotopy category $\Hot(\sA^\inj)$ is
compactly generated, and its full subcategory of compact objects is
equivalent to the bounded derived category of the abelian category
of finitely presentable objects in~$\sA$.
 A further generalization to locally coherent abelian DG\+categories,
such as the DG\+category of CDG\+modules over a graded left coherent
curved DG\+ring or the DG\+category of quasi-coherent matrix
factorizations over a coherent scheme, was obtained in the present
authors' paper~\cite[Theorem~8.19, Corollary~8.20, and
Corollary~9.4]{PS5}.

 In particular, for a topologically right coherent topological ring
$\R$, the Becker coderived category $\Hot(\Discrinj\R)\simeq
\sD^\bco(\Discr\R)$ is compactly generated, and its full subcategory
of compact objects is equivalent to the bounded derived category
$\sD^\bb(\coh\R)$ of the abelian category $\coh\R$ of finitely
presentable/coherent objects in the locally coherent Grothendieck
category $\Discr\R$.
 The aim of this paper is to prove, assuming additionally that $\R$
is complete, separated, and has a \emph{countable} base of
neighborhoods of zero consisting of open right ideals, that
the Becker contraderived category $\Hot(\R\Contra_\proj)\simeq
\sD^\bctr(\R\Contra)$ is compactly generated, and its full subcategory
of compact objects is anti-equivalent to $\sD^\bb(\coh\R)$.

 As a corollary of the description of compact generators, under
the same assumptions we also deduce the full form of the flat and
projective periodicity theorem for $\R$\+contramodules.
 So we obtain triangulated equivalences
$$
 \Hot(\R\Contra_\proj)\simeq\sD(\R\Contra_\flat)
 \simeq\sD^\bctr(\R\Contra),
$$
based on the theorem that a complex of flat left $\R$\+contramodules is
contraacyclic if and only if it is acyclic with flat contramodules
of cocycles.

 Let us emphasize once again that the latter result, continuing
the discussion of the flat and projective periodicity for contramodules
started in the first-named author's previous
papers~\cite[Proposition~12.1]{Pflcc}, \cite[Theorems~5.1 and~6.1]{Pbc},
is proved in the present paper under assumptions \emph{both more and
less general} than in~\cite{Pflcc,Pbc}.
 The setting in this paper is more general in the important aspect
that we do not assume a topology base of open two-sided ideals in
the ring $\R$, but only a topology base of open right ideals.
 But on the other hand, the proof of the flat and projective
periodicity for $\R$\+contramodules in the present paper uses
the assumption of topological right coherence of the ring~$\R$,
which was not needed in~\cite{Pflcc,Pbc}.

 To end, let us say a few words about our methods.
 The present paper uses a variety of techniques.
 Strict pro-objects in abelian categories are the main basic language
for our constructions and proofs, and we discuss them at length,
working out the foundations.
 Right linear topological modules over right linear topological rings
and the construction of Pontryagin duality for topological rings
developed in our previous paper~\cite[Section~3]{PSsp} play
an important role as well.
 The crucial construction of the fully faithful contravariant
triangulated functor $\Xi\:\sD^\bb(\coh\R)^\sop\rarrow
\Hot(\R\Contra_\proj)$ describing the compact generators in
the contraderived category goes through the derived category of
the exact category $\Cohpro\R$ of countably indexed strict pro-objects
in $\coh\R$ and the Pontryagin duality.

 As a complement to the Pontryagin duality, we also define and use
the functor of \emph{pro-contratensor product} of right linear
topological right modules and left contramodules over a right linear
topological ring.
 As a final touch, to finish the proof of compact generatedness and
make the application to flat/projective periodicity possible, we
use a result from the theory of well-generated triangulated categories
obtained in the joint paper of Saor\'\i n and the second-named
author~\cite[Proposition~4.9]{SaoSt}.

\subsection*{Acknowledgement}
 The authors are grateful to Michal Hrbek for helpful discussions.
 This research is supported by the GA\v CR project 23-05148S.
 The first-named author is also supported by
the Czech Academy of Sciences (RVO~67985840).

\Section{Preliminaries on Pro-Objects}

 Throughout this paper, the category of sets is denoted by $\Sets$,
and the category of abelian groups by~$\Ab$.
 Given two categories $\sC$ and $\sD$, the (possibly not locally
small) category of functors $\sD\rarrow\sC$ is denoted by $\sC^\sD$.

 Let $\sC$ be a category.
 A \emph{pro-object} in $\sC$ is covariant functor $\sC\rarrow\Sets$
from $\sC$ to the category of sets that is isomorphic to a directed
colimit of corepresentable functors.
 The category of pro-objects $\Pro(\sC)$ is the opposite category to
the full subcategory in the category of functors $\Sets^\sC$ formed
by all the functors that can be constructed as directed colimits
of corepresentables.
 When the category $\sC$ is additive, one can just as well view
pro-objects in $\sC$ as covariant functors $\sC\rarrow\Ab$ from $\sC$
to the category of abelian groups; the resulting category $\Pro(\sC)$
is the same.

 Equivalently, a pro-object $\bD$ in $\sC$ is a downwards directed
commutative diagram $(D_\delta\to D_\gamma)_{\gamma<\delta\in\Gamma}$
of objects $D_\gamma\in\sC$ and morphisms $D_\delta\rarrow D_\gamma$
in $\sC$ indexed by a directed poset~$\Gamma$.
 One denotes the pro-object $\bD\in\Pro(\sC)$ corresponding to
a diagram $(D_\gamma)_{\gamma\in\Gamma}$ by
$\bD=\plim_{\gamma\in\Gamma}D_\gamma$.

 For any two pro-objects $\bC=\plim_{\gamma\in\Gamma}C_\gamma$ and
$\bD=\plim_{\delta\in\Delta}D_\delta$, the set of morphisms
$\Hom_{\Pro(\sC)}(\bC,\bD)$ is computed as the directed limit of
directed colimits
\begin{equation} \label{hom-in-pro-objects-formula}
 \Hom_{\Pro(\sC)}(\plim_{\gamma\in\Gamma}C_\gamma,
 \plim_{\delta\in\Delta}D_\delta)
 =\varprojlim\nolimits_{\delta\in\Delta}
 \varinjlim\nolimits_{\gamma\in\Gamma}\Hom_\sC(C_\gamma,D_\delta)
\end{equation}
taken in the category of sets.

 The dual concept to the pro-objects is known as the \emph{ind-objects}.
 We refer to~\cite[Chapter~6]{KS} or~\cite[Section~9]{Pextop} for
an additional discussion.

 We will denote by $\Pro_\omega(\sC)\subset\Pro(\sC)$ the full
subcategory consisting of all the \emph{countable} directed colimits
of corepresentable functors $\sC\rarrow\Sets$, or equivalently,
of all the pro-objects corresponding to downwards directed diagrams
in $\sC$ indexed by \emph{countable} directed posets.

\begin{lem} \label{pro-objects-full-subcategory}
 Let\/ $\sD\subset\sC$ be a full subcategory.
 Then\/ $\Pro(\sD)$ is naturally a full subcategory in\/ $\Pro(\sC)$.
 In particular, $\Pro_\omega(\sD)$ is a full subcategory in\/
$\Pro_\omega(\sC)$.
\end{lem}

\begin{proof}
 Given a functor $F\:\sD\rarrow\sC$, the induced functor $\Pro(F)\:
\Pro(\sD)\rarrow\Pro(\sC)$ assigns to a pro-object $\bD=
\plim_{\gamma\in\Gamma}D_\gamma$ in $\sD$ represented by
a $\Gamma^\sop$\+indexed diagram $(D_\gamma)_{\gamma\in\Gamma}$
the pro-object $\Pro(F)(\bD)=\plim_{\gamma\in\Gamma}F(D_\gamma)$
in $\sC$ represented by the diagram $(F(D_\gamma))_{\gamma\in\Gamma}$.
 It is clear from the formula~\eqref{hom-in-pro-objects-formula}
that the functor $\Pro(F)\:\Pro(\sD)\rarrow\Pro(\sC)$ is fully faithful
whenever a functor $F\:\sD\rarrow\sC$ is.
\end{proof}

 In the next two lemmas, given a category $\sC$ and a directed poset
$\Gamma$, we consider the category $\sC^{\Gamma^\sop}$ of
functors/diagrams $\Gamma^\sop\rarrow\sC$. 

\begin{lem} \label{lim-in-quotes-preserves-finite-co-limits}
\textup{(a)} Let\/ $\sC$ be a category with finite colimits.
 Then the functor\/ $\plim\,\:\sC^{\Gamma^\sop}\rarrow\Pro(\sC)$
preserves finite colimits.
 When\/ $\Gamma$ is countable, the same applies to the functor\/
$\plim\,\:\sC^{\Gamma^\sop}\rarrow\Pro_\omega(\sC)$. \par
\textup{(b)} Let\/ $\sC$ be a category with finite limits.
 Then the functor\/ $\plim\,\:\sC^{\Gamma^\sop}\rarrow\Pro(\sC)$
preserves finite limits.
 When\/ $\Gamma$ is countable, the same applies to the functor\/
$\plim\,\:\sC^{\Gamma^\sop}\rarrow\Pro_\omega(\sC)$.
\end{lem}

\begin{proof}
 Both parts~(a) and~(b) follow from
the formula~\eqref{hom-in-pro-objects-formula} in view of the fact
that all limits and all directed colimits preserve finite limits
in the category of sets.
 The point is that the assertion about a particular object being
a limit or a colimit of a given finite diagram in some category means
a description of a certain finite limit in the category of sets.
\end{proof}

\begin{lem} \label{morphism-in-pro-objects-arise-from}
\textup{(a)} Any morphism in\/ $\Pro(\sC)$ is isomorphic to a morphism
coming from a morphism in\/ $\sC^{\Gamma^\sop}$ via the functor\/
$\plim\,\:\sC^{\Gamma^\sop}\rarrow\Pro(\sC)$ for some directed
poset\/~$\Gamma$. \par
\textup{(b)} Any morphism in\/ $\Pro_\omega(\sC)$ is isomorphic to
a morphism coming from a morphism in\/ $\sC^{\Gamma^\sop}$ via
the functor\/ $\plim\,\:\sC^{\Gamma^\sop}\rarrow\Pro_\omega(\sC)$
for some countable directed poset\/~$\Gamma$.
\end{lem}

\begin{proof}
 Part~(a): more generally, for any finite category $\sK$ having no
nonidentity endomorphisms of objects, the category $\Pro(\sC)^\sK$
of $\sK$\+indexed diagrams in $\Pro(\sC)$ is equivalent to
the category $\Pro(\sC^\sK)$ of pro-objects in the category of
$\sK$\+indexed diagrams in~$\sC$.
 This result goes back to~\cite[Expos\'e~I, Proposition~8.8.5]{SGA4};
see also~\cite[Theorem~6.4.3]{KS} or~\cite[Theorem~1.3]{Hen}.
 So any $\sK$\+indexed diagram of pro-objects in $\sC$ arises from
a $\sK\times\Gamma^\sop$\+indexed diagram in $\sC$ for some directed
poset~$\Gamma$.

 An explicit argument for the case when $\sK$ is the category
$\bullet\to\bullet$ (so $\sK$\+indexed diagrams are just morphisms) can
be found in~\cite[Proposition~6.1.3]{KS} or~\cite[Section~9]{Pextop}.
 The point is that for any morphism $\bbf\:\plim_{\gamma\in\Gamma}
C_\gamma\rarrow\plim_{\delta\in\Delta}D_\delta$ in $\Pro(\sC)$ there
exists a directed poset $\Upsilon$ together with two cofinal maps of
directed posets $\gamma\:\Upsilon\rarrow\Gamma$ and $\delta\:\Upsilon
\rarrow\Delta$ such that $\bbf=\plim_{\upsilon\in\Upsilon}f_\upsilon$,
where $(f_\upsilon\:C_{\gamma(\upsilon)}\to
D_{\delta(\upsilon)})_{\upsilon\in\Upsilon}$ is a morphism of
$\Upsilon^\sop$\+indexed diagrams in~$\sC$.
 Part~(b) is provable similarly (see also the proof of
Lemma~\ref{filtered-limits-in-pro-objects}(b) below).
\end{proof}

\begin{cor} \label{pro-objects-abelian}
 Let\/ $\sA$ be an abelian category.  In this context: \par
\textup{(a)} The category\/ $\Pro(\sA)$ is abelian.
 Moreover, for any directed poset\/ $\Gamma$, the functor\/
$\plim\,\:\sA^{\Gamma^\sop}\rarrow\Pro(\sA)$ is exact, and every short
exact sequence in $\Pro(\sA)$ is isomorphic to a short exact sequence
coming via this functor from a short exact sequence in the abelian
category\/ $\sA^{\Gamma^\sop}$ for some directed poset\/~$\Gamma$. \par
\textup{(b)} The category\/ $\Pro_\omega(\sA)$ is abelian.
 Moreover, for any countable directed poset\/ $\Gamma$, the functor\/
$\plim\,\:\sA^{\Gamma^\sop}\rarrow\Pro_\omega(\sA)$ is exact, and every
short exact sequence in $\Pro_\omega(\sA)$ is isomorphic to a short
exact sequence coming via this functor from a short exact sequence in
the abelian category\/ $\sA^{\Gamma^\sop}$ for a countable directed
poset\/~$\Gamma$.
\end{cor}

\begin{proof}
 Both parts~(a) and~(b) follow from the respective assertions of
Lemmas~\ref{lim-in-quotes-preserves-finite-co-limits}
and~\ref{morphism-in-pro-objects-arise-from}.
 See~\cite[Theorem~8.6.5]{KS} for some further details.
\end{proof}

 A pro-object $\bD$ in an abelian category $\sA$ is said to be
\emph{strict} if it can be represented by a downwards directed diagram
$(D_\delta\to D_\gamma)_{\gamma<\delta\in\Gamma}$ of epimorphisms
in~$\sA$.
 So, $\bD\in\Pro(\sA)$ is strict if one can choose $\Gamma$ and
the diagram $(D_\gamma)_{\gamma\in\Gamma}$ so that
$\bD\simeq\plim_{\gamma\in\Gamma}D_\gamma$ in $\Pro(\sA)$ and all
the transition morphisms $D_\delta\rarrow D_\gamma$ are epimorphisms.

\begin{lem} \label{strict-omega-pro-unambigous}
 A pro-object in\/ $\sA$ belonging to\/ $\Pro_\omega(\sA)\subset
\Pro(\sA)$ is strict if and only if it can be represented by
a countable downwards directed diagram of epimorphisms in\/~$\sA$.
 So the notion of a strict object in\/ $\Pro_\omega(\sA)$ is unambigous.
\end{lem}

\begin{proof}
 More generally, let $\bbf\:\bC\rarrow\bD$ be a morphism of pro-objects
in $\sA$, where the pro-object $\bC=\plim_{\gamma\in\Gamma}C_\gamma$
is represented by a downwards directed diagram of epimorphisms
$(C_{\gamma''}\to C_{\gamma'})_{\gamma'<\gamma''\in\Gamma}$, while
the pro-object $\bD=\plim_{\delta\in\Delta}D_\delta$ is represented by
some downwards directed diagram
$(D_{\delta''}\to D_{\delta'})_{\delta'<\delta''\in\Delta}$.
 Here the transition morphisms $D_{\delta''}\rarrow D_{\delta'}$
need not be epimorphisms.
 For every index $\delta\in\Delta$, consider the composition
$\bC\rarrow\bD\rarrow D_\delta$.
 Then there exists an index $\gamma\in\Gamma$ and a morphism
$f_{\gamma\delta}\:C_\gamma\rarrow D_\delta$ in $\sA$ such that
the square diagram $\bC\rarrow\bD\rarrow D_\delta$, \ $\bC\rarrow
C_\gamma\rarrow D_\delta$ is commutative in $\Pro(\sA)$.

 Firstly, suppose that there exists another morphism
$f'_{\gamma\delta}\:C_\gamma\rarrow D_\delta$ in $\sA$ making
the square diagram $\bC\rarrow\bD\rarrow D_\delta$, \ $\bC\rarrow
C_\gamma\rarrow D_\delta$ commutative.
 Then, by the definition of morphisms in $\Pro(\sA)$, there exists
an index $\gamma'\in\Gamma$, $\gamma'\ge\gamma$, such that
the morphisms $f_{\gamma\delta}$ and $f'_{\gamma\delta}\:C_\gamma
\rarrow D_\delta$ have equal compositions with the morphism
$C_{\gamma'}\rarrow C_\gamma$.
 Since the latter morphism is an epimorphism, it follows that
$f_{\gamma\delta}=f'_{\gamma\delta}$.

 Now let us define the subobject $E_\delta\subset D_\delta$ as
the image of the morphism $f_{\gamma\delta}\:C_\gamma\rarrow D_\delta$
in the abelian category~$\sA$.
 Since the morphism $C_{\gamma'}\rarrow C_\gamma$ is an epimorphism
for every $\gamma'>\gamma\in\Gamma$, the subobject $E_\delta\subset
D_\delta$ does not depend on the choice of an index $\gamma\in\Gamma$
as above.
 For all $\delta'<\delta''\in\Delta$, the morphisms $D_{\delta''}
\rarrow D_{\delta'}$ restrict to epimorphisms between the subobjects
$E_{\delta''}\rarrow E_{\delta'}$.
 We have constructed a downwards directed diagram of epimorphisms
$(E_{\delta''}\to E_{\delta'})_{\delta'<\delta''\in\Delta}$ in $\sA$
together with a morphism of diagrams $(E_\delta)_{\delta\in\Delta}
\rarrow(D_\delta)_{\delta\in\Delta}$ that is a termwise monomorphism.

 Put $\bE=\plim_{\delta\in\Delta}E_\delta$.
 Now the original morphism $\bbf\:\bC\rarrow\bD$ factorizes naturally
as $\bC\overset\bbg\rarrow\bE\overset\bbh\rarrow\bD$, where
$\bbh\:\bE\rarrow\bD$ is a monomorphism in $\Pro(\sA)$
(by Lemma~\ref{lim-in-quotes-preserves-finite-co-limits}(b)).
 In particular, if the morphism~$\bbf$ is an epimorphism in $\Pro(\sA)$,
then the morphism~$\bbh$ is an isomorphism.
 In the special case when $\bbf$~is an isomorphism and $\Delta$ is
countable, we obtain the desired assertion of the lemma.
\end{proof}

 Let us denote the full subcategories of strict pro-objects by
$\SPro(\sA)\subset\Pro(\sA)$ and $\SPro_\omega(\sA)=
\Pro_\omega(\sA)\cap\SPro(\sA)\subset\Pro_\omega(\sA)$.

\begin{prop} \label{strict-pro-objects-closed-under-extensions}
 Let\/ $\sA$ be an abelian category.  In this context: \par
\textup{(a)} The full subcategory\/ $\SPro(\sA)$ is closed under
extensions and quotients in the abelian category\/ $\Pro(\sA)$.
 Any short exact sequence in\/ $\Pro(\sA)$ with the terms belonging to\/
$\SPro(\sA)$ comes from a short exact sequence in\/ $\sA^{\Gamma^\sop}$
whose terms are three $\Gamma^\sop$\+indexed diagrams of epimorphisms
in\/ $\sA$, for some directed poset\/~$\Gamma$. \par
\textup{(b)} The full subcategory\/ $\SPro_\omega(\sA)$ is closed
under extensions and quotients in the abelian category\/
$\Pro_\omega(\sA)$.
 Any short exact sequence in\/ $\Pro_\omega(\sA)$ with the terms
belonging to\/ $\SPro_\omega(\sA)$ comes from a short exact sequence
in\/ $\sA^{\Gamma^\sop}$ whose terms are three $\Gamma^\sop$\+indexed
diagrams of epimorphisms in\/ $\sA$, for a countable directed
poset\/~$\Gamma$. \par
\end{prop}

\begin{proof}
 Part~(a): the closedness under quotients follows from the explicit
proof of Lemma~\ref{morphism-in-pro-objects-arise-from}, or 
alternatively, from the proof of
Lemma~\ref{strict-omega-pro-unambigous}.
 Both the closedness under quotients and extensions are the results
of~\cite[Proposition~9.2]{Pextop}.
 The last assertion of part~(a) also follows from~\cite[proof of
Proposition~9.2]{Pextop}.
 The argument is based on the explicit proof of
Lemma~\ref{morphism-in-pro-objects-arise-from} and the preservation
of pullbacks and pushouts by the functors $\plim\,\:\sA^{\Gamma^\sop}
\rarrow\Pro(\sA)$, as per
Lemma~\ref{lim-in-quotes-preserves-finite-co-limits}.
 The proof of part~(b) is similar.
\end{proof}

\begin{lem} \label{pro-objects-in-full-subcategory-extension-closed}
 Let\/ $\sA$ be an abelian category and\/ $\sB\subset\sA$ be a full
subcategory closed under extensions.  In this context: \par
\textup{(a)} The full subcategory\/ $\Pro(\sB)$ is closed under
extensions in\/ $\Pro(\sA)$.
 In particular, the full subcategory\/ $\SPro(\sB)$ is closed under
extensions in\/ $\SPro(\sA)$. \par
\textup{(b)} The full subcategory\/ $\Pro_\omega(\sB)$ is closed
under extensions in\/ $\Pro_\omega(\sA)$.
 In particular, the full subcategory\/ $\SPro_\omega(\sB)$ is closed
under extensions in\/ $\SPro_\omega(\sA)$.
\end{lem}

\begin{proof}
 Part~(a): first of all, $\Pro(\sB)$ is a full subcategory in
$\Pro(\sA)$ by Lemma~\ref{pro-objects-full-subcategory}.
 The assertion that $\Pro(\sB)$ is closed under extensions in
$\Pro(\sA)$ is provable by the argument from~\cite[proof of
Proposition~9.2]{Pextop}.
 Part~(b) is similar.
\end{proof}

\Section{Exact Categories of Strict Pro-Objects in
Abelian Categories}

 We suggest the survey paper~\cite{Bueh} as a background reference
source on exact categories in the sense of Quillen.
 In particular, a discussion of quasi-abelian categories can be
found in~\cite[Section~4]{Bueh}.

 Let $\sA$ be an abelian category.
 Then the category of pro-objects $\Pro(\sA)$ is abelian by
Corollary~\ref{pro-objects-abelian}(a), and the full subcategory
$\SPro(\sA)$ is closed under extensions (and quotients) in
$\Pro(\sA)$ by
Proposition~\ref{strict-pro-objects-closed-under-extensions}(a).
 We will consider $\SPro(\sA)$ as an exact category with
the exact structure inherited from the abelian exact structure of
$\Pro(\sA)$, as per~\cite[Lemma~10.20]{Bueh}.

 Similarly, the category $\Pro_\omega(\sA)$ is abelian
Corollary~\ref{pro-objects-abelian}(b), and the full subcategory
$\SPro_\omega(\sA)$ is closed under extensions (and quotients) in
$\Pro_\omega(\sA)$ by
Proposition~\ref{strict-pro-objects-closed-under-extensions}(b).
 We will consider $\SPro_\omega(\sA)$ as an exact category with
the exact structure inherited from the abelian exact structure of
$\Pro_\omega(\sA)$.

 A category $\sA$ is said to be
\emph{well-powered}~\cite[Section~0.6]{AR} if there is only a set
(rather than a proper class) of subobjects in any given object
of~$\sA$.

\begin{prop} \label{spro-quasi-abelian-characterization}
 For any well-powered abelian category\/ $\sA$, the following five
conditions are equivalent:
\begin{enumerate}
\item any set of subobjects of an object of\/ $\sA$ has an intersection;
\item any well-ordered descending chain of subobjects of an object of\/
$\sA$ has an intersection;
\item the additive category\/ $\SPro(\sA)$ is quasi-abelian
(and the exact structure inherited from the abelian exact structure
of\/ $\Pro(\sA)$ coincides with the quasi-abelian exact structure
on\/ $\SPro(\sA)$);
\item a right adjoint functor to the fully faithful inclusion functor\/
$\SPro(\sA)\rarrow\Pro(\sA)$ (a coreflector) exists;
\item the kernels of all morphisms exist in the additive category\/
$\SPro(\sA)$.
\end{enumerate}
\end{prop}

\begin{proof}
 The implications (1)~$\Longrightarrow$~(2) and
(3)~$\Longrightarrow$~(5) are obvious.

 (2)~$\Longrightarrow$~(1) Follows from the fact that any finite
set of subobjects of an object of any abelian category has
an intersection.
 Given a set $S$ of subobjects of an object $A\in\sA$, finite 
intersections of subobjects belonging to $S$ form a downwards directed
family of subobjects in $A$, and it remains, e.~g., to refer
to~\cite[Corollary~1.7]{AR}.

 (2)~$\Longrightarrow$~(4) Let $\bB=\plim_{\gamma\in\Gamma}B_\gamma$
be a pro-object in~$\sA$ represented by a $\Gamma^\sop$\+indexed
diagram $(B_\gamma)_{\gamma\in\Gamma}$ for a directed poset~$\Gamma$.
 Let $\lambda$~be a cardinal greater than the cardinality of
the disjoint union over $\gamma\in\Gamma$ of the sets of all subobjects
in~$B_\gamma$.
 Proceeding by transfinite induction, we construct for every ordinal
$i\le\lambda$ and every index $\gamma\in\Gamma$ a subobject
$B_\gamma(i)\subset B_\gamma$ such that
$(B_\gamma(i))_{\gamma\in\Gamma}$ is a subdiagram in
$(B_\gamma)_{\gamma\in\Gamma}$.
 
 For $i=0$, we put $B_\gamma(0)=B_\gamma$ for all $\gamma\in\Gamma$.
 For a limit ordinal~$i$, put $B_\gamma(i)=\bigcap_{j<i}B_\gamma(j)$.
 For a successor ordinal $i=j+1$, look for a pair of indices
$\alpha<\beta\in\Gamma$ such that the transition morphism
$B_\beta(j)\rarrow B_\alpha(j)$ is not an epimorphism.
 If such a pair of indices does not exist, put $B_\gamma(i)=
B_\gamma(j)$ for all $\gamma\in\Gamma$.
 If a pair if indices $\alpha<\beta$ as above was found, let
$B_\alpha(i)\subset B_\alpha(j)$ be the image of the morphism
$B_\beta(j)\rarrow B_\alpha(j)$.
 For all indices $\gamma\in\Gamma$, \ $\gamma>\alpha$, let
$B_\gamma(i)\subset B_\gamma(j)$ be the pullback of the subobject
$B_\alpha(i)\subset B_\alpha(j)$ with respect to the transition
morphism $B_\gamma(j)\rarrow B_\alpha(j)$.
 For all the other indices $\gamma\in\Gamma$, put
$B_\gamma(i)=B_\gamma(j)$.

 After the construction has been finished, the resulting diagram 
$(C_\gamma=B_\gamma(\lambda))_{\gamma\in\Gamma}$ is the unique maximal
subdiagram in $(B_\gamma)_{\gamma\in\Gamma}$ that is a diagram
of epimorphisms in~$\sA$.
 The pro-object $\bC=\plim_{\gamma\in\Gamma}C_\gamma$ is the desired
coreflection of the object $\bB=\plim_{\gamma\in\Gamma}B_\gamma$ onto
the full subcategory $\SPro(\sA)\subset\Pro(\sA)$.

 (4)~$\Longrightarrow$~(5) The kernel of any morphism $\bbg\:\bD
\rarrow\bE$ in $\SPro(\sA)$ can be constructed as the coreflection
of the kernel of~$\bbg$ in the abelian category $\Pro(\sA)$ onto
the full subcategory $\SPro(\sA)\subset\Pro(\sA)$.

 (2)~$\Longrightarrow$~(3) Let $\bC\overset\bbf\rarrow\bD\overset\bbg
\rarrow\bE$ be a pair of morphisms in $\SPro(\sA)$ such that $\bbf$~is
the kernel of~$\bbg$ and $\bbg$~is the cokernel of~$\bbf$
in $\SPro(\sA)$.
 It suffices to show that $0\rarrow\bC\rarrow\bD\rarrow\bE\rarrow0$ is
an admissible short exact sequence in the exact structure on $\SPro(\sA)$
inherited from the abelian exact structure on $\Pro(\sA)$ (i.~e.,
an exact sequence in the abelian category $\Pro(\sA)$).
 Then, in view of
Proposition~\ref{strict-pro-objects-closed-under-extensions}(a),
it will follow that the collection of all such short sequences forms
an exact category structure on $\SPro(\sA)$, which precisely means that
the additive category $\SPro(\sA)$ is quasi-abelian.
 So, we need to show that every kernel-cokernel pair of morphisms in
$\SPro(\sA)$ is also a kernel-cokernel pair in $\Pro(\sA)$.

 By Lemma~\ref{morphism-in-pro-objects-arise-from}(a), there exists
a directed poset $\Gamma$ such that the morphism~$\bbg$ can be
represented by a morphism of $\Gamma^\sop$\+indexed diagrams
$(g_\gamma)_{\gamma\in\Gamma}\:(D_\gamma)_{\gamma\in\Gamma}\rarrow
(E_\gamma)_{\gamma\in\Gamma}$ in~$\sA$.
 Moreover, the explicit proof of
Lemma~\ref{morphism-in-pro-objects-arise-from}(a) above that can be
found in~\cite[Section~9]{Pextop} shows that, for a morphism~$\bbg$
in $\SPro(\sA)$, one can choose $(D_\gamma)_{\gamma\in\Gamma}$ and
$(E_\gamma)_{\gamma\in\Gamma}$ to be diagrams of epimorphisms in~$\sA$.

 Put $B_\gamma=\ker(g_\gamma)$.
 By Lemma~\ref{lim-in-quotes-preserves-finite-co-limits}(b),
the object $\bB=\plim_{\gamma\in\Gamma}B_\gamma$ is the kernel of
the morphism~$\bbg$ in the abelian category $\Pro(\sA)$.
 Following the arguments for the implications (2)~$\Longrightarrow$
(4)~$\Longrightarrow$ (5) above, the object $\bC\in\SPro(\sA)$ can be
represented by a subdiagram $(C_\gamma)_{\gamma\in\Gamma}$ of
the diagram $(B_\gamma)_{\gamma\in\Gamma}$ such that
$(C_\gamma)_{\gamma\in\Gamma}$ is a diagram of epimorphisms in~$\sA$.
 So we have monomorphisms $C_\gamma\rarrow B_\gamma$ and
$B_\gamma\rarrow D_\gamma$; denote their compositions by
$f_\gamma\:C_\gamma\rarrow D_\gamma$.
 Put $E'_\gamma=\coker(f_\gamma)\in\sA$.
 
 By Lemma~\ref{lim-in-quotes-preserves-finite-co-limits}(a),
the object $\plim_{\gamma\in\Gamma}E'_\gamma$ is the cokernel of
the morphism~$\bbf$ in $\Pro(\sA)$.
 Since $(E'_\gamma)_{\gamma\in\Gamma}$ is a diagram of epimorphisms
in $\sA$, we have $\plim_{\gamma\in\Gamma}E'_\gamma\in\SPro(\sA)$;
so $\plim_{\gamma\in\Gamma}E'_\gamma$ is also the cokernel of~$f$
in $\SPro(\sA)$.
 Thus we have $\bE=\plim_{\gamma\in\Gamma}E'_\gamma$.
 We have shown that the pair of morphisms $(\bbf,\bbg)$ can be obtained
by applying the functor $\plim_{\gamma\in\Gamma}$ to a short
exact sequence of $\Gamma^\sop$\+indexed diagrams $0\rarrow C_\gamma
\rarrow D_\gamma\rarrow E'_\gamma\rarrow0$.
 Therefore, $0\rarrow\bC\rarrow\bD\rarrow\bE\rarrow0$ is an admissible
short exact sequence in the exact structure on $\SPro(\sA)$ inherited
from the abelian exact structure of $\Pro(\sA)$, as desired.

 (5)~$\Longrightarrow$~(1) Let $A\in\sA$ be an object and
$A_\gamma\subset A$ be a downwards directed family of subobjects in $A$
indexed by a directed poset~$\Gamma$ (so $A_\gamma\supset A_\delta$
whenever $\gamma<\delta\in\Gamma$).
 Consider $A$ as an object of $\SPro(\sA)$, and consider also
the object $\bB=\plim_{\gamma\in\Gamma}A/A_\gamma\in\SPro(\sA)$.
 The morphism of $\Gamma^\sop$\+indexed diagrams
$(A\to A/A_\gamma)_{\gamma\in\Gamma}$ induces a morphism of
pro-objects $\bbg\:A\rarrow\bB$ in~$\sA$.
 Let $\bbf\:\bK\rarrow A$ be the kernel of~$\bbg$ in $\SPro(\sA)$.
 Then we have $\bK=\plim_{\delta\in\Delta}K_\delta$, where
$(K_\delta)_{\delta\in\Delta}$ is a $\Delta^\sop$\+indexed diagram
of epimorphisms in $\sA$ for some directed poset~$\Delta$.

 Following the explicit proof of
Lemma~\ref{morphism-in-pro-objects-arise-from}(a) above that can be
found in~\cite[Section~9]{Pextop}, there exists a directed poset
$\Upsilon$ and two cofinal maps of directed posets $\delta\:\Upsilon
\rarrow\Delta$ and $\gamma\:\Upsilon\rarrow\Gamma$ such that
the morphism~$f$ can be represented by a morphism of
$\Upsilon^\sop$\+indexed diagrams
$(f_\upsilon)_{\upsilon\in\Upsilon}\:
(K_{\delta(\upsilon)})_{\upsilon\in\Upsilon}\rarrow
(A_{\gamma(\upsilon)}=A)_{\upsilon\in\Upsilon}$.
 Let $\alpha<\beta\in\Upsilon$ be two indices.
 We know that the morphism $K_{\delta(\beta)}\rarrow
K_{\delta(\alpha)}$ is an epimorphism in~$\sA$; let us show that it
is also a monomorphism in~$\sA$.

 Indeed, let $L_\beta\subset K_{\delta(\beta)}$ be the kernel of
the morphism $K_{\delta(\beta)}\rarrow K_{\delta(\alpha)}$.
 Denote by $\Upsilon'\subset\Upsilon$ the subset of all indices
$\upsilon'\in\Upsilon$ such that $\upsilon'\ge\beta$ in~$\Upsilon$.
 Endow $\Upsilon'$ with the restricted order; then $\Upsilon'$ is
a cofinal subset in~$\Upsilon$.
 For every $\upsilon'\in\Upsilon'$, denote by $L_{\upsilon'}\subset
K_{\delta(\upsilon')}$ the pullback of the subobject
$L_\beta\subset K_{\delta(\beta)}$ with respect to the transition
morphism $K_{\delta(\upsilon')}\rarrow K_{\delta(\beta)}$.
 Clearly, the transition morphisms $L_{\upsilon''}\rarrow L_{\upsilon'}$
are epimorphisms for all $\upsilon''>\upsilon'\in\Upsilon'$.
 Put $\bL=\plim_{\upsilon'\in\Upsilon'}L_{\upsilon'}\in\SPro(\sA)$.
 For every $\upsilon'\in\Upsilon'$, let $h_{\upsilon'}\:L_{\upsilon'}
\rarrow K_{\delta(\upsilon')}$ denote the natural monomorphism.
 Then $\bbh=\plim_{\upsilon'\in\Upsilon'}h_{\upsilon'}\:\bL\rarrow\bK$
is a morphism in $\SPro(\sA)$.

 For every index $\upsilon'\in\Upsilon'$, the composition of morphisms
$$
 L_{\upsilon'}\overset{h_{\upsilon'}}\lrarrow K_{\delta(\upsilon')}
 \overset{f_{\upsilon'}}\lrarrow A
$$
vanishes, as one can see from the commutative diagram
$$
 \xymatrix{
  L_{\upsilon'} \ar[r]^-{h_{\upsilon'}} \ar[d]
  & K_{\delta(\upsilon')} \ar[r]^-{f_{\upsilon'}} \ar[d]
  & A \ar@{=}[d] \\
  L_\beta \ar[r]^-{h_\beta} \ar[d]
  & K_{\delta(\beta)} \ar[r]^-{f_\beta} \ar[d]
  & A \ar@{=}[d] \\
  0 \ar[r] & K_{\delta(\alpha)} \ar[r]^-{f_\alpha} & A
 }
$$
 Applying the functor $\plim_{\upsilon'\in\Upsilon'}$, we conclude that
the composition of morphisms of pro-objects
$$
 \bL\overset\bbh\lrarrow\bK\overset\bbf\lrarrow A
$$
vanishes as well.
 Since the morphism $\bbf\:\bK\rarrow A$, being a kernel of
the morphism~$\bbg$ in $\SPro(\sA)$, is a monomorphism in $\SPro(\sA)$
by assumption, it follows that $\bbh=0$.

 At the same time, $\bbh$~is a monomorphism in $\Pro(\sA)$ (hence also
in $\SPro(\sA)$), since it is represented by a termwise monomorphism
of $\Upsilon'{}^\sop$\+indexed diagrams
$(h_{\upsilon'})_{\upsilon'\in\Upsilon'}$.
 Therefore, we have $\bL=0$.
 On the other hand, the natural morphism $\bL\rarrow L_\beta$ is
an epimorphism in $\Pro(\sA)$, and in fact even an admissible
epimorphism in $\SPro(\sA)$, since the morphisms $L_{\upsilon'}\rarrow
L_\beta$ are epimorphisms in $\sA$ and their kernels form
$\Upsilon'{}^\sop$\+indexed diagram of epimorphisms in~$\sA$.
 We can finally conclude that $L_\beta=0$ and $K_{\delta(\beta)}
\rarrow K_{\delta(\alpha)}$ is an isomorphism in~$\sA$.
 
 Thus $(K_{\delta(\upsilon)})_{\upsilon\in\Upsilon}$ is
an $\Upsilon^\sop$\+indexed diagram of isomorphisms in~$\sA$;
so $\bK=K\in\sA\subset\SPro(\sA)$.
 It remains to notice that, given a subobject $M\subset A$ of
the object $A$ in the category $\sA$, the composition $M\rarrow A
\overset\bbg\rarrow\bB$ vanishes if and only if $M$ is contained in
$A_\gamma$ for every $\gamma\in\Gamma$.
 Since $K$ is the kernel of~$\bbg$ in $\SPro(\sA)$, it follows that
$K=\bigcap_{\upsilon\in\Upsilon}A_{\gamma(\upsilon)}=
\bigcap_{\gamma\in\Gamma}A_\gamma$ is the desired intersection of
all the subobjects $A_\gamma$ in the object $A\in\sA$.
\end{proof}

\begin{prop} \label{spro-omega-quasi-abelian-characterization}
 Given a well-powered abelian category\/ $\sA$, consider the following
seven conditions:
\begin{enumerate}
\item any set of subobjects of an object of\/ $\sA$ has an intersection;
\item any well-ordered descending chain of subobjects of an object of\/
$\sA$ has an intersection;
\item the additive category\/ $\SPro_\omega(\sA)$ is quasi-abelian,
and the exact structure inherited from the abelian exact structure
of\/ $\Pro_\omega(\sA)$ coincides with the quasi-abelian exact
structure on\/ $\SPro_\omega(\sA)$;
\item a right adjoint functor to the fully faithful inclusion functor\/
$\SPro_\omega(\sA)\rarrow\Pro_\omega(\sA)$ (a coreflector) exists;
\item the kernels of all morphisms exist in the additive category\/
$\SPro_\omega(\sA)$;
\item any countable set of subobjects of an object of\/ $\sA$ has
an intersection;
\item any descending chain of subobjects of an object of\/ $\sA$
indexed by the well-ordered set of nonnegative integers has
an intersection.
\end{enumerate}
 Then the implications (1) $\Longleftrightarrow$ (2) $\Longrightarrow$
(3) $\Longrightarrow$ (4) $\Longleftrightarrow$ (5)
$\Longrightarrow$ (6) $\Longleftrightarrow$ (7) hold.
\end{prop}

\begin{proof}
 The equivalence (1)~$\Longleftrightarrow$~(2) is a part of
Proposition~\ref{spro-quasi-abelian-characterization}, and
the equivalence (6)~$\Longleftrightarrow$~(7) is similar.
 The implication (3)~$\Longrightarrow$~(5) is obvious.

 The implications (2) $\Longrightarrow$ (4) $\Longrightarrow$ (5) and
(2)~$\Longrightarrow$~(3) are similar to the respective implications
in Proposition~\ref{spro-quasi-abelian-characterization}.
 The implication (5)~$\Longrightarrow$~(6) is similar to
Proposition~\ref{spro-quasi-abelian-characterization}\,%
(5)\,$\Rightarrow$\,(1).

 (5)~$\Longrightarrow$~(4) Let $\bbg\:\bD\rarrow\bE$ be a morphism in
$\SPro_\omega(\sA)$, and let $\bB$ be the kernel of~$\bbg$ in
the abelian category $\Pro_\omega(\sA)$.
 Then a kernel of the morphism~$\bbg$ in the category
$\SPro_\omega(\sA)$ is the same thing as a coreflection of the object
$\bB\in\Pro_\omega(\sA)$ onto the full subcategory $\SPro_\omega(\sA)
\subset\Pro_\omega(\sA)$.
 If any one of these two objects exists, then the other one exists, too;
and they are naturally isomorphic.
 It remains to construct a morphism $\bbg\:\bD\rarrow\bE$ in
$\SPro_\omega(\sA)$ for a given object $\bB\in\Pro_\omega(\sA)$.

 Let $(B_n)_{n\in\omega}$ be
an $\omega^\sop$\+indexed diagram in $\sA$ representing an object
$\bB=\plim_{n\in\omega}B_n\in\Pro_\omega(\sA)$.
 Consider the $\omega^\sop$\+indexed diagram $(D_n)_{n\in\omega}$
in $\sA$ with the components $D_n=\bigoplus_{m=0}^n B_m$ and
the natural direct summand projections $D_{n+1}\rarrow D_n$ as
the transition maps.
 So the transition maps in the diagram $(D_n)_{n\in\omega}$ are
split epimorphisms in~$\sA$.
 Then there is a natural morphism of $\omega^\sop$\+indexed diagrams
$(f_n)_{n\in\omega}\:(B_n)_{n\in\omega}\rarrow(D_n)_{n\in\omega}$,
where the components of the map $f_n\:B_n\rarrow D_n$ are
the transition maps $B_n\rarrow B_m$, \ $0\le m\le n$.
 Clearly, the morphism~$f_n$ is a split monomorphism in $\sA$ for
every $n\in\omega$.

 Denote by $g_n\:D_n\rarrow E_n$ the cokernel of the split
monomorphism $f_n\:B_n\rarrow D_n$.
 Then we have a short exact sequence of $\omega^\sop$ indexed
diagrams $0\rarrow(B_n)_{n\in\omega}\rarrow(D_n)_{n\in\omega}
\rarrow(E_n)_{n\in\omega}\rarrow0$.
 Applying the functor $\plim_{n\in\omega}$, we obtain a short
exact sequence $0\rarrow\bB\overset\bbf\rarrow\bD\overset\bbg\rarrow\bE
\rarrow0$ in $\Pro_\omega(\sA)$ with the objects $\bD$ and $\bE$
belonging to $\SPro_\omega(\sA)$.
 Let $\bC$ be the kernel of the morphism $\bbg\:\bD\rarrow\bE$ in
the category $\SPro_\omega(\sA$).
 Then $\bC$ is the coreflection of the object $\bB\in\Pro_\omega(\sA)$
onto the full subcategory $\SPro_\omega(\sA)\subset\Pro_\omega(\sA)$.
\end{proof}

 Let $\sE$ be an exact category.
 A full subcategory $\sA\subset\sE$ is said to be
\emph{self-cogenerating} if for every morphism $A\rarrow E$ in $\sE$
with $A\in\sA$ there exists a morphism $E\rarrow B$ in $\sE$ with
$B\in\sA$ such that the composition $A\rarrow B$ is an admissible
monomorphism in~$\sE$.
 A self-cogenerating full subcategory $\sA\subset\sE$ is said to be
\emph{self-coresolving} if $\sA$ is closed under extensions and
cokernels of admissible monomorphisms in~$\sE$.

\begin{lem} \label{objects-coresolving-in-strict-pro-objects}
 Let $\sA$ be an abelian category.  In this context: \par
\textup{(a)} The essential image of the natural exact fully faithful
functor\/ $\sA\rarrow\SPro(\sA)$ is a self-coresolving full subcategory
in the exact category\/ $\SPro(\sA)$.
 The exact structure on\/ $\sA$ inherited from the exact structure of\/
$\SPro(\sA)$ coincides with the abelian exact structure on\/~$\sA$. \par
\textup{(b)} The essential image of the natural exact fully faithful
functor\/ $\sA\rarrow\SPro_\omega(\sA)$ is a self-coresolving full
subcategory in the exact category\/ $\SPro_\omega(\sA)$.
 The exact structure on\/ $\sA$ inherited from the exact structure
of\/ $\SPro_\omega(\sA)$ coincides with the abelian exact structure
on\/~$\sA$.
\end{lem}

\begin{proof}
 Part~(a): Let $A\in\sA\subset\SPro(\sA)$ be an object and
$0\rarrow A\rarrow\bE\rarrow\bD\rarrow0$ be a short exact sequence
in $\Pro(\sA)$.
 It is clear from the proof of~\cite[Proposition~9.2]{Pextop}
(cf.\ Proposition~\ref{strict-pro-objects-closed-under-extensions})
that the exists a directed poset $\Gamma$ and a short exact sequence
of $\Gamma^\sop$\+indexed diagrams $0\rarrow(A_\gamma)_{\gamma\in\Gamma}
\rarrow(E_\gamma)_{\gamma\in\Gamma}\rarrow(D_\gamma)_{\gamma\in\Gamma}
\rarrow0$ in $\sA$ representing the short exact sequence
$0\rarrow A\rarrow\bE\rarrow\bD\rarrow0$ in $\Pro(\sA)$ such that
$(A_\gamma=A)_{\gamma\in\Gamma}$ is the constant diagram~$A$.
 Pick an element $\gamma_0\in\Gamma$, and put $B=E_{\gamma_0}\in\sA$.
 Then the composition $A\rarrow\bE\rarrow B$ is a monomorphism in $\sA$,
hence also a monomorphism in $\Pro(\sA)$ and an admissible monomorphism
in $\SPro(\sA)$.
 This proves that the full subcategory $\sA$ is self-cogenerating in
both $\Pro(\sA)$ and $\SPro(\sA)$.

 To prove that $\sA$ is closed under extensions in $\Pro(\sA)$,
consider two objects $A$, $D\in\sA\subset\SPro(\sA)$ and a short exact
sequence $0\rarrow A\rarrow\bE\rarrow D\rarrow0$ in $\Pro(\sA)$.
 Similarly to the argument above, it is clear from the proof
of~\cite[Proposition~9.2]{Pextop} that there exists a directed poset
$\Gamma$ and a short exact sequence of $\Gamma$\+indexed diagrams
in $\sA$ representing the short exact sequence
$0\rarrow A\rarrow\bE\rarrow D\rarrow0$ in $\Pro(\sA)$ such that both
$(A_\gamma=A)_{\gamma\in\Gamma}$ and $(D_\gamma=D)_{\gamma\in\Gamma}$
are the constant diagrams $A$ and~$D$.
 Then it follows that $(E_\gamma)_{\gamma\in\Gamma}$ is also
a constant diagram in~$\sA$.
 Therefore, the full subcategory $\sA$ is also closed under extensions
in $\Pro(\sA)$ and $\SPro(\sA)$.

 The assertion that the exact category structure on $\sA$ inherited
from $\Pro(\sA)$ or $\SPro(\sA)$ coincides with the abelian exact
structure on $\sA$ follows from exactness of the fully faithful
functor $\sA\rarrow\Pro(\sA)$, as there cannot be any admissible
short exact sequences in an exact structure on $\sA$ that are not
exact in the abelian exact structure.
 Furthermore, any monomorphism in $\Pro(\sA)$ or $\SPro(\sA)$ between
two objects from $\sA$ is a monomorphism in~$\sA$.
 As the cokernels of all (mono)morphisms exist in $\sA$, we can conclude
that $\sA$ is closed under the cokernels of admissible monomorphisms
in both $\Pro(\sA)$ and $\SPro(\sA)$.
 The proof of part~(b) is similar.
\end{proof}

 A discussion of the derived categories of exact categories can be
found in~\cite[Section~10]{Bueh}.

\begin{cor} \label{fully-faithful-on-D-b-and-D-plus}
 Let $\sA$ be an abelian category.  In this context: \par
\textup{(a)} The natural exact fully faithful functor\/
$\sA\rarrow\SPro(\sA)$ induces an isomorphism on all the Yoneda\/
$\Ext$ groups.
 In other words, the induced triangulated functor between the bounded
derived categories\/ $\sD^\bb(\sA)\rarrow\sD^\bb(\SPro(\sA))$ is
fully faithful.
 Moreover, the similar functor between the bounded below derived
categories\/ $\sD^+(\sA)\rarrow\sD^+(\SPro(\sA))$ is
fully faithful, too. \par
\textup{(b)} The natural exact fully faithful functor\/
$\sA\rarrow\SPro_\omega(\sA)$ induces an isomorphism on all
the Yoneda\/ $\Ext$ groups.
 In other words, the induced triangulated functor between the bounded
derived categories\/ $\sD^\bb(\sA)\rarrow\sD^\bb(\SPro_\omega(\sA))$ is
fully faithful.
 Moreover, the similar functor between the bounded below derived
categories\/ $\sD^+(\sA)\rarrow\sD^+(\SPro_\omega(\sA))$ is
fully faithful, too.
\end{cor}

\begin{proof}
 The assertions follow from the respective assertions of
Lemma~\ref{objects-coresolving-in-strict-pro-objects} by virtue
of~\cite[Theorem~12.1(2)]{Kel} or~\cite[Proposition~A.2.1]{Pcosh}.
\end{proof}

\Section{Products and Projectives in Countably Indexed
Strict~Pro-Objects}

 We start with a discussion of filtered limits and products in
the category $\Pro_\omega(\sC)$ for a possibly nonadditive category
$\sC$ before passing to the abelian category case.

\begin{lem} \label{filtered-limits-in-pro-objects}
\textup{(a)} For any category\/ $\sC$, all filtered limits exits in
the category\/ $\Pro(\sC)$. \par
\textup{(b)} For any category\/ $\sC$, all countable filtered limits
exist in the category\/ $\Pro_\omega(\sC)$.
 Moreover, the full subcategory\/ $\Pro_\omega(\sC)$ is closed under
countable filtered limits in\/ $\Pro(\sC)$.
\end{lem}

\begin{proof}
 Part~(a) is~\cite[Theorem~6.1.8]{KS}.
 Assuming $\sC$ to be a small category, this argument can be also
phrased as follows.
 The category $\Sets^\sC$ is locally finitely presentable, and
the corepresentable functors $\sC\rarrow\Sets$ are finitely presentable
objects in $\Sets^\sC$.
 Consequently, the class of all directed colimits of corepresentable
functors is closed under directed colimits in $\Sets^\sC$ by
a nonadditive version of~\cite[Section~4.1]{CB}
or~\cite[Proposition~5.11]{Kra0}, or by the special case
of~\cite[Proposition~1.2]{Pacc} for $\kappa=\aleph_0$.

 Part~(b): without loss of generality, one can restrict oneself to
the countable directed poset of nonnegative integers $\Gamma=\omega$.
 Then one observes that any $\omega^\sop$\+indexed diagram in
$\Pro_\omega(\sC)$ comes from an $\omega^\sop$\+indexed diagram in
$\sC^{\omega^\sop}$ via the functor $\plim_\omega\,\:\sC^{\omega^\sop}
\rarrow\Pro_\omega(\sC)$.

 Indeed, let $\dotsb\rarrow\bD_2\rarrow\bD_1\rarrow\bD_0$ be
an $\omega^\sop$\+indexed diagram in $\Pro_\omega(\sC)$.
 For every $n\ge0$, the pro-object $\bD_n$ can be represented as
$\bD_n=\plim_{k\in\omega}D_{n,k}$, where $(D_{n,k})_{k\in\omega}$
is an $\omega^\sop$\+indexed diagram in~$\sC$.

 The composition $\bD_1\rarrow\bD_0\rarrow D_{0,0}$ factorizes through
the morphism $\bD_1\rarrow D_{1,k_{0,0}}$ for some $k_{0,0}\in\omega$.
 So we obtain a morphism $f_{k_{0,0},0}^0\:D_{1,k_{0,0}}\rarrow
D_{0,0}$ in $\sC$ making the square diagram
$$
 \xymatrix{
  \bD_1 \ar[r] \ar[d] & \bD_0 \ar[d] \\
  D_{1,k_{0,0}} \ar[r] & D_{0,0}
 }
$$
commutative in $\Pro_\omega(\sC)$.
 The composition $\bD_1\rarrow\bD_0\rarrow D_{0,1}$ factorizes through
the morphism $\bD_1\rarrow D_{1,k'_{0,1}}$ for some $k'_{0,1}\in\omega$,
$k'_{0,1}\ge k_{0,0}$.
 Then the two compositions $D_{1,k'_{0,1}}\rarrow D_{1,k_{0,0}}\rarrow
D_{0,0}$ and $D_{1,k'_{0,1}}\rarrow D_{0,1}\rarrow D_{0,0}$ need not be
equal to each other as morphisms in~$\sC$.
 However, there exists an integer $k_{0,1}\in\omega$ such that $k_{0,1}
\ge k'_{0,1}$, and the two compositions $D_{1,k_{0,1}}\rarrow
D_{1,k'_{0,1}}\rarrow D_{1,k_{0,0}}\rarrow D_{0,0}$ and
$D_{1,k_{0,1}}\rarrow D_{1,k'_{0,1}}\rarrow D_{0,1}\rarrow D_{0,0}$
are equal to each other in~$\sC$.
 So we obtain a morphism $f^0_{k_{0,1},1}\:D_{1,k_{0,1}}\rarrow D_{0,1}$
in $\sC$ making the rightmost square diagram commutative in $\sC$, while
the leftmost square diagram is commutative in $\Pro_\omega(\sC)$:
$$
 \xymatrix{
  \bD_1 \ar[r] \ar[d] & \bD_0 \ar[d] \\
  D_{1,k_{0,1}} \ar[r] & D_{0,1}
 }
 \qquad\qquad
 \xymatrix{
  D_{1,k_{0,1}} \ar[r] \ar[d] & D_{0,1} \ar[d] \\
  D_{1,k_{0,0}} \ar[r] & D_{0,0}
 }
$$
 Proceeding in this way, we construct a cofinal subdiagram
$(D_{1,k_{0,m}})_{m\in\omega}$ in the diagram $(D_{1,k})_{k\in\omega}$
with the property that the morphism $\bD_1\rarrow\bD_0$ in
$\Pro_\omega(\sC)$ can be obtained by applying the functor $\plim$
to a morphism of diagrams $f^0\:(D_{1,k_{0,m}})_{m\in\omega}\rarrow
(D_{0,m})_{m\in\omega}$.
 Then, applying the same construction to the morphism of pro-objects
$\bD_2\rarrow\bD_1$, we obtain a cofinal subdiagram
$(D_{2,k_{1,m}})_{m\in\omega}$ in the diagram $(D_{2,k})_{k\in\omega}$
such that that the morphism $\bD_2\rarrow\bD_1$ in $\Pro_\omega(\sC)$
can be obtained by applying the functor $\plim$ to a morphism of
diagrams $f^1\:(D_{2,k_{1,m}})_{m\in\omega}\rarrow
(D_{1,k_{0,m}})_{m\in\omega}$, etc.

 We have shown that our $\omega^\sop$\+indexed diagram
$(\bD_n)_{n\in\omega}$ in $\Pro_\omega(\sC)$ comes from
an $\omega^\sop$\+indexed diagram in $\sC^{\omega^\sop}$, which means
(after a suitable renumbering)
a $\omega^\sop\times\nobreak\omega^\sop$\+indexed diagram
$(D_{n,m}\in\sC)_{n,m\in\omega}$ with commutative squares of morphisms
$$
 \xymatrix{
  D_{n+1,m+1} \ar[r] \ar[d] & D_{n,m+1} \ar[d] \\
  D_{n+1,m} \ar[r] & D_{n,m}
 }
$$
 Now the desired countable filtered limit in $\Pro_\omega(\sC)$
can be constructed as
$$
 \varprojlim\nolimits_{n\in\omega}\bD_n=
 \plim_{(n,m)\in\omega\times\omega}D_{n,m}=\plim_{n\in\omega}D_{n,n}.
$$

 The moreover clause follows from the constructions above.
 In fact, the constructions show that the full subcategory $\Pro(\sC)$
is closed under filtered limits in $\bigl(\Sets^\sC\bigr)^\sop$, and
the full subcategory $\Pro_\omega(\sC)$ is closed under countable
filtered limits in $\bigl(\Sets^\sC\bigr)^\sop$.
\end{proof}

\begin{cor} \label{products-of-pro-objects}
 Let\/ $\sC$ be a category with finite products.  Then \par
\textup{(a)} all products exist in the category\/ $\Pro(\sC)$; \par
\textup{(b)} all countable products exist in the category\/
$\Pro_\omega(\sC)$.
 Moreover, the full subcategory\/ $\Pro_\omega(\sC)$ is closed
under countable products in\/ $\Pro(\sC)$.
\end{cor}

\begin{proof}
 Part~(a) is~\cite[Proposition~6.1.18]{KS}.
 Simply put, all products are filtered limits of finite products.
 By a version of
Lemma~\ref{lim-in-quotes-preserves-finite-co-limits}(b),
finite products exist in $\Pro(\sC)$; and filtered limits exist in
$\Pro(\sC)$ by Lemma~\ref{filtered-limits-in-pro-objects}(a).
 The proof of part~(b) is similar and based on
Lemmas~\ref{lim-in-quotes-preserves-finite-co-limits}(b)
and~\ref{filtered-limits-in-pro-objects}(b).
 Simply put, all countable products are countable filtered limits
of finite products.

 Let us describe an explicit construction of products in $\Pro(\sC)$.
 Suppose given a family $(\bD_\pi)_{\pi\in\Pi}$ of objects in
$\Pro(\sC)$, each represented by a $\Gamma_\pi^\sop$\+indexed diagram
$(D_\gamma)_{\gamma\in\Gamma_\pi}$ in $\sC$ for some directed
poset~$\Gamma_\pi$; so $\bD_\pi=\plim_{\gamma\in\Gamma_\pi}D_\gamma$.
 Without loss of generality we can assume that each directed poset
$\Gamma_\pi$ contains a minimal element $0_\pi\in\Gamma_\pi$ such
that $0_\pi\le\gamma$ for all $\gamma\in\Gamma_\pi$ and
$D_{0_\pi}={*}$ is the terminal object of~$\sC$.
 (Adjoin a minimal element to each directed poset $\Gamma_\pi$ if
necessary.)
 
 Now let $\Delta\subset\prod_{\pi\in\Pi}\Gamma_\pi$ be the set of
all families of indices
$\delta=(\gamma_\pi(\delta)\in\Gamma_\pi)_{\pi\in\Pi}$ such that
$\gamma_\pi(\delta)=0_\pi$ for all but a finite subset of
indices $\pi\in\Pi$.
 Introduce a partial order on the set $\Delta$ by the rule
$\delta'\le\delta''$ if $\gamma_\pi(\delta')\le\gamma_\pi(\delta'')$
for all $\pi\in\Pi$.
 For every $\delta\in\Delta$, put $C_\delta=\prod_{\pi\in\Pi}
D_{\gamma_\pi(\delta)}$, where the product over $\pi\in\Pi$ is
taken in the category~$\sC$.
 The product exists in $\sC$ because it is essentially finite:
all but a finite subset of the objects $D_{\gamma_\pi(\delta)}$
are the terminal objects.
 The transition morphisms $C_{\delta''}\rarrow C_{\delta'}$ for
$\delta'<\delta''$ are the products of the transition morphisms
in the diagrams $(D_\gamma)_{\gamma\in\Gamma_\pi}$.

 It is clear that the desired product $\prod_{\pi\in\Pi}\bD_\pi$
in $\Pro(\sC)$ can be computed as
$$
 \prod\nolimits_{\pi\in\Pi}\plim_{\gamma\in\Gamma_\pi}D_\gamma =
 \plim_{\delta\in\Delta}C_\delta =
 \plim_{\delta\in\Delta}
 \prod\nolimits_{\pi\in\Pi}D_{\gamma_\pi(\delta)}.
$$

 Let us present a separate explicit construction of countable products
in $\Pro_\omega(\sC)$.
 Suppose given a countable family $(\bD_n)_{n\in\omega}$ of
objects in $\Pro_\omega(\sC)$, each represented by
an $\omega^\sop$\+indexed diagram $\dotsb\rarrow D_{n,2}\rarrow
D_{n,1}\rarrow D_{n,0}$ in~$\sC$.
 Then the product $\prod_{n\in\omega}\bD_n$ in $\Pro_\omega(\sC)$ is
represented by the diagram
$$
 \dotsb\lrarrow D_{0,2}\times D_{1,1}\times D_{2,0}
 \lrarrow D_{0,1}\times D_{1,0}\lrarrow D_{0,0},
$$
that is
$$
 \prod\nolimits_{n\in\omega}\plim_{m\in\omega}D_{n,m}=
 \plim_{k\in\omega}\prod\nolimits_{n+m=k}D_{n,m}.
$$
 Here the transition morphism $D_{0,k+1}\times\dotsb\times D_{k+1,0}
\rarrow D_{0,k}\times\dotsb\times D_{k,0}$ is the product of
the transition morphisms $D_{n,m+1}\rarrow D_{n,m}$ taken over
all $n+m=k$, precomposed with the subproduct projection morphism
$D_{0,k+1}\times\dotsb\times D_{k,1}\times D_{k+1,0}
\rarrow D_{0,k+1}\times\dotsb\times D_{k,1}$.

 The moreover clause follows from the moreover clause of
Lemma~\ref{filtered-limits-in-pro-objects}(b).
\end{proof}

\begin{cor} \label{morphism-from-countable-product-factorizes}
 Let\/ $\sC$ be a category with finite products.
 Then, for any countable family of objects $\bD_n\in\Pro_\omega(\sC)$,
\ $n\in\omega$, and for any object $C\in\sC$, any morphism
$\prod_{n\in\omega}\bD_n\rarrow C$ in $\Pro_\omega(\sC)$ factorizes
through a finite subproduct,
$$
 \prod\nolimits_{n\in\omega}\bD_n\lrarrow
 \prod\nolimits_{n=0}^m \bD_n\rarrow C,
$$
for some $m\in\omega$.
 Here\/ $\prod_{n\in\omega}\bD_n\rarrow\prod_{n=0}^m\bD_n$ is
the canonical subproduct projection.
\end{cor}

\begin{proof}
 Follows from the explicit construction of the countable product
$\prod_{n\in\omega}\bD_n$ in the proof of
Corollary~\ref{products-of-pro-objects}(b).
 We do not state the similar assertion for arbitrary products in
$\Pro(\sC)$, which is also clear from the explicit proof
of Corollary~\ref{products-of-pro-objects}(a).
\end{proof}

\begin{cor} \label{products-preserve-strict-pro-objects}
 For any abelian category\/ $\sA$, the full subcategory of strict
pro-objects\/ $\SPro(\sA)$ is closed under infinite products in\/
$\Pro(\sA)$.
 The full subcategory of countably indexed strict pro-objects\/
$\SPro_\omega(\sA)$ is closed under countable products in\/
$\Pro_\omega(\sA)$.
\end{cor}

\begin{proof}
 Follows from the explicit constructions in the proof of
Corollary~\ref{products-of-pro-objects}.
\end{proof}

\begin{lem} \label{products-exact-in-pro-objects}
 For any abelian category\/ $\sA$, the functors of infinite product
are exact in the abelian category\/ $\Pro(\sA)$, while the functor
of countable product is exact in the abelian category\/
$\Pro_\omega(\sA)$.
\end{lem}

\begin{proof}
 The assertions follow from the explicit descriptions of infinite
products in $\Pro(\sA)$ and countable products in $\Pro_\omega(\sA)$
given in the proof of Corollary~\ref{products-of-pro-objects}.
 Alternatively, one can observe that all limits in abelian categories
are left exact; and filtered limits are right exact in $\Pro(\sA)$,
since filtered colimits are left exact in $\Ab^\sA$ and the inclusion
$\Pro(\sA)^\sop\rarrow\Ab^\sA$ takes cokernels to kernels.
 So filtered limits are exact in $\Pro(\sA)$.
 In particular, countable filtered limits are exact in
$\Pro_\omega(\sA)$.
\end{proof}

\begin{cor} \label{products-exact-in-strict-pro-objects}
 For any abelian category\/ $\sA$, all products of admissible short
exact sequences are admissible short exact sequences in the exact
category\/ $\SPro(\sA)$, while all countable products of admissible
short exact sequences are admissible short exact sequences
in the exact category\/ $\SPro_\omega(\sA)$.
 In other words, the functors of infinite product are exact in
the exact category\/ $\SPro(\sA)$, while the functor of countable
product is exact in the exact category\/ $\SPro_\omega(\sA)$.
\end{cor}

\begin{proof}
 Compare Lemma~\ref{products-exact-in-pro-objects}
with Corollary~\ref{products-preserve-strict-pro-objects}.
\end{proof}

 A discussion of projective objects in exact categories can be found
in~\cite[Section~11]{Bueh}.

\begin{lem} \label{projectivity-criterion-in-spro-omega}
 Let\/ $\sA$ be an abelian category.
 Then an object $\bP\in\SPro_\omega(\sA)$ is projective in the exact
category\/ $\SPro_\omega(\sA)$ if and only if, for any epimorphism
$B\rarrow C$ in the abelian category\/ $\sA$, any morphism
$\bP\rarrow C$ can be lifted to a morphism $\bP\rarrow B$ in\/
$\SPro_\omega(\sA)$.
\end{lem}

\begin{proof}
 The ``only if'' implication holds since any epimorphism in $\sA$ is
an admissible epimorphism in $\SPro_\omega(\sA)$.
 Conversely, let $0\rarrow\bD\rarrow\bE\rarrow\bF\rarrow0$ be
an admissible short exact sequence in $\SPro_\omega(\sA)$.
 By Proposition~\ref{strict-pro-objects-closed-under-extensions}(b),
the short sequence $0\rarrow\bD\rarrow\bE\rarrow\bF\rarrow0$ can be
obtained by applying the functor $\plim_\omega$ to a short exact
sequence $0\rarrow(D_n)_{n\in\omega}\rarrow(E_n)_{n\in\omega}\rarrow
(F_n)_{n\in\omega}\rarrow0$ in the abelian category $\sA^{\omega^\sop}$
(i.~e., a termwise short exact sequence of $\omega^\sop$\+indexed
diagrams in~$\sA$) such that, in all the three diagrams
$(D_n)_{n\in\omega}$, \ $(E_n)_{n\in\omega}$, and $(F_n)_{n\in\omega}$,
all the transition morphisms are epimorphisms.
 Then it follows that, for every $n\in\omega$, the natural map to
the fibered product $E_{n+1}\rarrow E_n\times_{F_n}F_{n+1}$ is
an epimorphism, too.

 Suppose given a morphism $\bbf\:\bP\rarrow\bF$ in $\SPro_\omega(\sA)$.
 Consider the compositions $\bP\rarrow\bF\rarrow F_n$ and denote them
by $\bbf_n\:\bP\rarrow F_n$.
 By assumption, any morphism $\bP\rarrow F_0$ can be lifted to
a morphism $\bP\rarrow E_0$ in $\SPro_\omega(\sA)$.
 Denote a morphism lifting~$\bbf_0$ by $\bbe_0\:P\rarrow E_0$.
 The pair of morphisms $\bbf_1$ and~$\bbe_0$ induces a morphism
$\bbg_1\:\bP\rarrow E_0\times_{F_0}F_1$.
 By assumption, the morphism~$\bbg_1$ can be lifted to a morphism
$\bbe_1\:\bP\rarrow E_1$.
 The pair of morphisms $\bbf_2$ and~$\bbe_1$ induces a morphism
$\bbg_2\:\bP\rarrow E_1\times_{F_1}F_2$.
 By assumption, the morphism~$\bbg_2$ can be lifted to a morphism
$\bbe_2\:\bP\rarrow E_2$, etc.
 Proceeding in this way, and passing to the filtered limit, we construct
a morphism $\bbe\:\bP\rarrow\bE$ in $\SPro_\omega(\sA)$ lifting the given
morphism $\bbf\:\bP\rarrow\bF$.
\end{proof}

\begin{lem} \label{projectives-closed-under-countable-products}
 For any abelian category\/ $\sA$, the class of projective objects in
the exact category\/ $\SPro_\omega(\sA)$ is closed under countable
products.
\end{lem}

\begin{proof}
 Let us check that the class of all pro-objects satisfying
the criterion of Lemma~\ref{projectivity-criterion-in-spro-omega}
is closed under countable products in $\SPro_\omega(\sA)$.
 The point is that, for any family of projective objects
$\bP_n\in\SPro_\omega(\sA)$ and any object $C\in\sA$, any morphism
$\prod_{n\in\omega}\bP_n\rarrow C$ in $\SPro_\omega(\sA)$ factorizes
through a finite subproduct $\prod_{n=0}^m\bP_n$ by
Corollary~\ref{morphism-from-countable-product-factorizes}.
 Now, given an epimorphism $B\rarrow C$ in $\sA$, it is clear
that any morphism $\prod_{n=0}^m \bP_n\rarrow C$ can be lifted to
a morphism $\prod_{n=0}^m \bP_n\rarrow B$.
\end{proof}

 Given an additive category $\sE$ with countable products and a class
of objects $\sP\subset\sE$, we denote by $\Prod_\omega(\sP)\subset\sE$
the class of all direct summands of countable products of objects
from $\sP$ in~$\sE$.
 More generally, if all infinite products exist in $\sE$, the notation
$\Prod(\sP)\subset\sE$ stands for the class of all direct summands of
arbitrary products of copies of objects from~$\sP$.

\begin{prop} \label{projectives-as-summands-of-products}
 Let\/ $\sA$ be an abelian category, and let\/ $\sP\subset
\SPro_\omega(\sA)$ be a class of (some) projective objects.
 Assume that for every object $A\in\sA\subset\SPro_\omega(\sA)$
there exists an admissible epimorphism $\bP\rarrow A$ in\/
$\SPro_\omega(\sA)$ with $\bP\in\sP$.
 Then the class of all projective objects in\/ $\SPro_\omega(\sA)$
coincides with\/ $\Prod_\omega(\sP)\subset\SPro_\omega(\sA)$, and
there are enough projective objects in\/ $\SPro_\omega(\sA)$.
\end{prop}

\begin{proof}
 In view of Lemma~\ref{projectives-closed-under-countable-products},
it suffices to show that every object of $\SPro_\omega(\sA)$ is
the codomain of an admissible epimorphism from an object belonging
to $\Prod_\omega(\sP)$.

 Let $\bD=\plim_{n\in\omega}D_n$ be an object of $\SPro_\omega(\sA)$,
where $(D_n)_{n\in\omega}$ is an $\omega^\sop$\+indexed diagram of
epimorphisms in~$\sA$.
 By assumption, there exists an object $\bP_0\in\sP\subset
\SPro_\omega(\sA)$ together with an admissible epimorphism
$\bP_0\rarrow D_0$ in $\SPro_\omega(\sA)$.
 The pro-object $\bP_0$ can be represented by an $\omega^\sop$\+indexed
diagram $(P_{0,m})_{m\in\omega}$ of epimorphisms in~$\sA$.
 Since $D_0\in\sA$, there exists an integer $m_0\in\omega$ such
that the morphism $\bP_0\rarrow D_0$ factorizes as $\bP_0\rarrow
P_{0,m_0}\rarrow D_0$.
 Passing to a cofinal subdiagram in $(P_{0,m})_{m\in\omega}$ and
renumbering as necessary, we can assume without loss of generality
that $m_0=0$.

 Furthermore, the additive category $\SPro_\omega(\sA)$ is
idempotent-complete (in the sense of~\cite[Section~6]{Bueh}), since
it is closed under quotients in the abelian category $\Pro_\omega(\sA)$.
 In particular, $\SPro_\omega(\sA)$ is weakly idempotent-complete
in the sense of~\cite[Section~7]{Bueh}.
 Therefore, the strong form of ``obscure axiom''
\cite[Proposition~7.6(ii)]{Bueh} holds in the exact category
$\SPro_\omega(\sA)$.
 Since the morphism $\bP_0\rarrow D_0$ is an admissible epimorphism,
it follows that $P_{0,0}\rarrow D_0$ is an admissible epimorphism
in $\SPro_\omega(\sA)$.
 For a morphism in $\sA$, being an admissible epimorphism in
$\SPro_\omega(\sA)$ clearly means being an epimorphism in~$\sA$
(in fact, any epimorphism in $\SPro_\omega(\sA)$ between objects
of $\sA$ is an epimorphism in~$\sA$).
 Thus $P_{0,0}\rarrow D_0$ is an epimorphism in~$\sA$.

 In addition, $\bD\rarrow D_0$ is an admissible epimorphism in
$\SPro_\omega(\sA)$ (since the kernels of the morphisms
$D_n\rarrow D_0$, \,$n\in\omega$, form an $\omega^\sop$\+indexed
diagram of epimorphisms in~$\sA$).
 Therefore, the admissible epimorphism $\bP_0\rarrow D_0$ lifts to
a morphism $\bP_0\rarrow\bD$ in $\SPro_\omega(\sA)$.
 Passing to a cofinal subdiagram in $(P_{0,m})_{m\in\omega}$ and
renumbering as necessary again, but keeping the object $P_{0,0}$
and the morphism $P_{0,0}\rarrow D_0$ unchanged, we can assume that
the morphism $\bP_0\rarrow\bD$ comes from a morphism of diagrams
$(P_{0,m})_{m\in\omega}\rarrow(D_m)_{m\in\omega}$ via
the functor $\plim_\omega$.

 Now, for every $n\in\omega$, denote by $K_{n+1}$ the kernel of
the epimorphism $D_{n+1}\rarrow D_n$ in~$\sA$.
 By assumption, there exists an object $\bP_{n+1}\in\sP\subset
\SPro_\omega(\sA)$ together with an admissible epimorphism
$\bP_{n+1}\rarrow K_{n+1}$ in $\SPro_\omega(\sA)$.
 The pro-object $\bP_{n+1}$ can be represented by
an $\omega^\sop$\+indexed diagram $(P_{n+1,m})_{m\in\omega}$ of
epimorphisms in~$\sA$.
 Arguing as above and passing if necessary to a cofinal subdiagram,
we can assume without loss of generality that the morphism
$\bP_{n+1}\rarrow K_{n+1}$ factorizes as $\bP_{n+1}\rarrow P_{n+1,0}
\rarrow K_{n+1}$.

 Furthermore, the same argument as above shows that $P_{n+1,0}
\rarrow K_{n+1}$ is an epimorphism in~$\sA$.
 In addition, $\ker(\bD\to D_n)\rarrow K_{n+1}$ is an admissible
epimorphism in $\SPro_\omega(\sA)$ (since the kernels of
the morphisms $D_{m+n+1}\to D_{n+1}$, \,$m\in\omega$, form
an $\omega^\sop$\+indexed diagram of epimorphisms in~$\sA$).
 Therefore, the admissible epimorphism $\bP_{n+1}\rarrow K_{n+1}$ lifts
to a morphism $\bP_{n+1}\rarrow\ker(\bD\to D_n)$ in $\SPro_\omega(\sA)$.
 Once again we can pass to a cofinal subdiagram in
$(P_{n+1,m})_{m\in\omega}$ (but keep the object $P_{n+1,0}$ and
the morphism $P_{n+1,0}\rarrow K_{n+1}$ unchanged) so that
the morphism $\bP_{n+1}\rarrow\ker(\bD\to D_n)$ in $\SPro_\omega(\sA)$
gets represented by a morphism of diagrams $(P_{n+1,m})_{m\in\omega}
\rarrow(\ker(D_{m+n+1}\to D_n))_{m\in\omega}$.

 Passing to the direct sum over~$m$, we obtain a morphism of
$\omega^\sop$\+indexed diagrams
$$
 p_k\:\bigoplus\nolimits_{n+m=k}P_{n,m}\lrarrow D_k, \qquad k\in\omega,
$$
where the diagram with the components $\bigoplus_{n+m=k}P_{n,m}$
was constructed in the proof of
Corollary~\ref{products-of-pro-objects}(b).

$$
 \xymatrix{
 \dotsb\ar[r] & P_{0,2}\oplus P_{1,1}\oplus P_{2,0}
 \ar[r] \ar[d]^{p_2} & P_{0,1}\oplus P_{1,0} \ar[r] \ar[d]^{p_1}
  & P_{0,0} \ar[d]^{p_0} \\
 \dotsb\ar[r] & D_2 \ar[r] & D_1 \ar[r] & D_0
 }
$$
 
 For every $k\in\omega$, the morphism $P_{0,k+1}\oplus\dotsb\oplus
P_{k,1}\rarrow P_{0,k}\oplus\dotsb\oplus P_{k,0}$ is an epimorphism
as a finite direct sum of epimorphisms $P_{n,m+1}\rarrow P_{n,m}$.
 The morphism $P_{0,0}\rarrow D_0$ is an epimorphism; and for
every $k\ge1$ the morphism $P_{k,0}\rarrow D_k$ is the composition
$P_{k,0}\rarrow K_k\rarrow D_k$, where $K_k\rarrow D_k$ is
the kernel of the epimorphism $D_k\rarrow D_{k-1}$ and $P_{k,0}
\rarrow K_k$ is an epimorphism.
 It follows that $(\ker p_k)_{k\in\omega}$ is a diagram of
epimorphisms, so $\plim_{k\in\omega}p_k$ is an admissible epimorphism
in $\SPro_\omega(\sA)$.
 On the other hand, the pro-object
$\plim_{k\in\omega}\bigoplus_{n+m=k}P_{n,m}$ is the countable product
$\prod_{n\in\omega}P_n$ in $\Pro_\omega(\sA)$ and $\SPro_\omega(\sA)$ 
by the proof of Corollary~\ref{products-of-pro-objects}(b) and
Corollary~\ref{products-preserve-strict-pro-objects}.
\end{proof}

\Section{Left Contramodules over Right Linear Topological Rings}
\label{contramodules-preliminaries-secn}

 We suggest the papers~\cite[Introduction and Sections~5--7]{PR},
\cite[Section~2]{Pcoun}, \cite[Section~2]{Pproperf},
\cite[Sections~6\+-7]{PS1}, or~\cite[Sections~7--9]{Pflcc} and
the survey paper~\cite{Prev} as background reference sources on
contramodules over topological rings.
 The paper~\cite{Pextop} can be used for background material
on topological abelian groups with linear topologies.

 Given an associative ring $R$, we denote by $R\Modl$ the abelian
category of left $R$\+modules and by $\Modr R$ the abelian category
of right $R$\+modules.

 A topological abelian group $A$ is said to have \emph{linear topology}
if open subgroups form a base of neighborhoods of zero in~$A$.
 All topological abelian groups in this paper will be presumed to have
linear topologies.

 The \emph{completion} $\fA$ of a topological abelian group $A$ is
constructed as the projective limit $\fA=\varprojlim_{U\subset A}A/U$,
where $U$ ranges over all the open subgroups of~$A$.
 The kernels of the projection maps $\fA\rarrow A/U$ form a base of
neighborhoods of zero in the natural \emph{completion topology}
on the abelian group~$\fA$.

 The natural \emph{completion map} $\lambda_A\:A\rarrow\fA$ is
a continuous homomorphism of topological abelian groups.
 The topological abelian group $A$ is said to be \emph{separated} if
the map~$\lambda_A$ is injective, and \emph{complete} if
the map~$\lambda_A$ is surjective.
 A topological abelian group $A$ is separated and complete if and only
if the completion map~$\lambda_A$ is an isomorphism of topological
abelian groups.
 For any topological abelian group $A$ with a linear topology,
the completion $\fA=\varprojlim_{U\subset A}A/U$ is a complete,
separated topological abelian group with linear topology.

 An (associative, unital) topological ring $R$ is said to have
\emph{right linear topology} if open right ideals form a base of
neighborhoods of zero in~$R$.
 The completion $\R=\varprojlim_{I\subset R}R/I$ of a right linear
topological ring~$R$ (where $I$ ranges over the open right ideals
of~$R$) has a unique ring structure such that the completion map
$\lambda_R\:R\rarrow\R$ is a ring homomorphism.
 A right linear topological ring is said to be \emph{separated} 
(respectively, \emph{complete}) if it is separated (resp., complete)
as a topological abelian group.
 The completion $\R$ of any right linear topological ring $R$ is
a complete, separated right linear topological ring.

 Given a right linear topological ring $R$, a right $R$\+module $\N$
is said to be \emph{discrete} if the annihilator of any element
$b\in\N$ is an open right ideal in~$R$.
 Equivalently, a right $R$\+module $\N$ is discrete if and only if
the action map $\N\times R\rarrow\N$ is continuous (as a function of
two variables) with respect to the given topology of $R$ and
the discrete topology on~$\N$.
 The full subcategory of discrete right $R$\+modules is denoted by
$\Discr R\subset\Modr R$.
 The full subcategory $\Discr R$ is closed under subobjects, quotients,
and infinite direct sums in the abelian category $\Modr R$.
 The category $\Discr R$ is a locally finitely generated Grothendieck
abelian category (in the sense of the condition
in~\cite[Theorem~1.70]{AR} or of~\cite[Section~V.3]{Sten}; see
Section~\ref{pro-coherent-modules-secn} below for a discussion).

 The action of a topological ring $R$ on any discrete right $R$\+module
$\N$ extends uniquely to an action of the completion $\R$ of
the topological ring $R$ on the same module/abelian group~$\N$.
 So the category $\Discr R$ is naturally equivalent (in fact,
isomorphic) to the category $\Discr\R$.

 Given an abelian group $A$ and a set $X$, we will use the notation
$A[X]=A^{(X)}$ for the direct sum of $X$ copies of~$A$.
 The elements of $A[X]$ are interpreted as finite formal linear
combinations $\sum_{x\in X}a_xx$ of elements of $X$ with
the coefficients $a_x\in A$.
 So, for every element $a\in A[X]$, the set of all $x\in X$ for which
$a_x\ne0$ must be finite.

 Let $\fA$ be a complete, separated topological abelian group with
linear topology, and let $X$ be a set.
 Then we put $\fA[[X]]=\varprojlim_{\fU\subset\fA}(\fA/\fU)[X]$, where
$\fU$ ranges over all the open subgroups of~$\fA$.
 The elements of $\fA[[X]]$ are interpreted as infinite formal linear
combinations $\sum_{x\in X}a_xx$ of elements of $X$ with
\emph{zero-converging} families of coefficients $a_x\in\fA$.
 Here the zero-convergence condition means that, for every open
subgroup $\fU\subset\fA$, the set of all indices $x\in X$ for which
$a_x\notin\fU$ must be finite.

 Any map of sets $f\:X\rarrow Y$ induces a homomorphism of abelian
groups $\fA[[f]]\:\fA[[X]]\rarrow\fA[[Y]]$ taking a formal linear
combination $\sum_{x\in X}a_xx\in\fA[[X]]$ to the formal linear
combination $\sum_{y\in Y}b_yy\in\fA[[Y]]$ with the coefficients
defined by the following rule.
 For every $y\in Y$, put $b_y=\sum_{x\in X}^{f(x)=y}a_x$.
 Here the infinite summation sign in the latter formula is understood
as the limit of finite partial sums in the topology of~$\fA$.
 One can say that $\fA[[f]]$ is the map of ``push-forward of
measures''.
 This construction makes the assignment $X\longmapsto\fA[[X]]$
a covariant functor $\fA[[{-}]]\:\Sets\rarrow\Ab$.
 However, we will mostly consider it as a covariant functor
$\fA[[{-}]]\:\Sets\rarrow\Sets$.

 Now let $\R$ be a complete, separated right linear topological ring.
 Then the functor $\R[[{-}]]\:\Sets\rarrow\Sets$ has a natural structure
of a \emph{monad} on the category of sets (in the sense
of~\cite[Chapter~VI]{MacLane}; see also~\cite[Section~6]{PS1}).
 The monad unit $\epsilon_X\:X\rarrow\R[[X]]$ is the natural 
transformation (a functorial morphism defined for all sets~$X$)
taking an element $x\in X$ to the formal linear combination
$\sum_{y\in X}r_yy$ with $r_x=1\in\R$ and $r_y=0\in\R$ for all
$y\in X$, \ $y\ne x$.
 One can say that $\epsilon_X$~is the ``point measure'' map.
 The monad multiplication $\phi_X\:\R[[\R[[X]]]]\rarrow\R[[X]]$ is
the ``opening of parentheses'' map producing a formal linear combination
from a formal linear combination of formal linear combinations.
 The construction of the map~$\phi_X$ involves taking products of
pairs of elements in $\R$ and then computing infinite sums in~$\R$,
which are interpreted as the limits of finite partial sums in
the topology of~$\R$.
 The conditions imposed on~$\R$ (the completeness and separatedness,
but particularly the right linear topology) guarantee the convergence.

 \emph{Left contramodules} over $\R$ are defined as algebras (or, in
our preferred terminology, modules) over the monad $\R[[{-}]]$ on
the category of sets.
 Explicitly, a left $\R$\+contramodule $\fP$ is a set endowed with
a map of sets $\pi_\fP\:\R[[\fP]]\rarrow\fP$ (called
the \emph{contraaction map}) such that the following
\emph{contraassociativity} and \emph{contraunitality} conditions hold.
 Firstly, the two compositions
$$
 \R[[\R[[\fP]]]]\rightrightarrows\R[[\fP]]\rarrow\fP
$$
of the ``opening of parentheses'' and ``measure push-forward'' maps
$\phi_\fP$ and $\R[[\pi_\fP]]\:\allowbreak\R[[\R[[\fP]]]]\rarrow\R[[\fP]]$
with the contraaction map $\pi_\fP\:\R[[\fP]]\rarrow\fP$ must be equal
to each other.
 Secondly, the composition
$$
 \fP\rarrow\R[[\fP]]\rarrow\fP
$$
of the ``point measure'' map $\epsilon_\fP\:\fP\rarrow\R[[\fP]]$ and
the contraaction map $\pi_\fP\:\R[[\fP]]\allowbreak\rarrow\fP$
must be equal to the identity map~$\id_\fP$.
 We denote the category of left $\R$\+contramodules by $\R\Contra$.

 In particular, for a discrete ring $R$, there is a monad structure
on the functor $R[{-}]\:\Sets\rarrow\Sets$.
 The category of algebras/modules over this monad is naturally
equivalent (in fact, isomorphic) to the category of left $R$\+modules
$R\Modl$.
 Basically, a module structure over the monad $R[{-}]$ on a set $M$
is the same thing as a left $R$\+module structure on~$M$.

 For any complete, separated right linear topological ring $\R$,
the natural inclusion $\R[X]\rarrow\R[[X]]$ of the set of all finite
formal linear combinations of elements of a set $X$ with
the coefficients in $\R$ into the set of all infinite formal linear
combinations with zero-convergent families of coefficients is
a morphism of monads $\R[{-}]\rarrow\R[[{-}]]$ on the category of sets.
 This morphism of monads induces a forgetful functor
$\R\Contra\rarrow\R\Modl$ from the category of left $\R$\+contramodules
to the category of left modules over the ring $\R$ viewed as an abstract
(nontopological) ring.
 Basically, the contramodule forgetful functor $\R\Contra\rarrow\R\Modl$
restricts the contramodule infinite summation operations encoded in
a contraaction map $\pi_\fP\:\R[[\fP]]\rarrow\fP$ to their finite
aspects described by the map $\R[\fP]\rarrow\fP$ obtained by
precomposing the map~$\pi_\fP$ with the natural inclusion
$\R[\fP]\rarrow\R[[\fP]]$.

 The category of left $\R$\+contramodules $\R\Contra$ is abelian
with exact functors of infinite direct products.
 The forgetful functor $\R\Contra\rarrow\R\Modl$ is exact and faithful,
and preserves infinite products.

 For any set $X$, the set/abelian group $\R[[X]]$ has a natural left
$\R$\+contramodule structure with the contraaction map
$\pi_{\R[[X]]}=\phi_X$.
 The $\R$\+contramodule $\R[[X]]$ is called the \emph{free} left
$\R$\+contramodule spanned by the set~$X$.
 This terminology is justified by the fact that, for any left
$\R$\+contramodule $\fQ$, the abelian group of $\R$\+contramodule
morphisms $\R[[X]]\rarrow\fQ$ is naturally isomorphic to the group of
all maps of sets $X\rarrow\fQ$,
$$
 \Hom^\R(\R[[X]],\fQ)\simeq\Hom_\Sets(X,\fQ).
$$
 Here we denote by $\Hom^\R(\fP,\fQ)$ the abelian group of morphisms
$\fP\rarrow\fQ$ in the category $\R\Contra$.

 In other words, the free contramodule functor $\R[[{-}]]\:\Sets
\rarrow\R\Contra$ is left adjoint to the forgetful functor
$\R\Contra\rarrow\Sets$.
 Therefore, the free contramodule functor preserves coproducts.
 So the coproducts of free contramodules can be computed by
the rule
$$
 \coprod\nolimits_{\upsilon\in\Upsilon}^{\R\Contra}\R[[X_\upsilon]]
 =\R\Bigl[\Bigl[
 \coprod\nolimits_{\upsilon\in\Upsilon}^\Sets X_\upsilon
 \Bigr]\Bigr],
$$
where $\Upsilon$ is an arbitrary indexing set and
$\coprod_{\upsilon\in\Upsilon}^\sC$ denotes the $\Upsilon$\+indexed
coproduct in a category~$\sC$.

 All coproducts exist in the abelian category $\R\Contra$.
 One can compute the coproduct of an arbitrary family of left
$\R$\+contramodules $(\fP_\upsilon)_{\upsilon\in\Upsilon}$ by
representing each $\fP_\upsilon$ as the cokernel of a morphism
of free $\R$\+contramodules and using the fact that coproducts
commute with cokernels in any category.
 For example, any left $\R$\+contramodule $\fP$ is the cokernel of
the difference $\R[[\R[[\fP]]]]\rarrow\R[[\fP]]$ of the pair of
maps appearing in the contraassociativity axiom; this gives
a presentation of $\fP$ as the cokernel of a morphism of free
$\R$\+contramodules.

 The abelian category $\R\Contra$ has enough projective objects.
 A left $\R$\+con\-tra\-mod\-ule is projective if and only if it is
a direct summand of a free left $\R$\+contramodule $\R[[X]]$ for
some set~$X$.
 The abelian category $\R\Contra$ is also locally presentable
(in the sense of~\cite[Definition~1.17 and Theorem~1.20]{AR}).
 More precisely, let $\lambda$~be the successor cardinal of
the cardinality of a base of neighborhoods of zero in~$\R$.
 Then the category $\R\Contra$ is locally $\lambda$\+presentable.
 The free left $\R$\+contramodule with one generator $\R=\R[[\{{*}\}]]$
is a $\lambda$\+presentable projective generator of the abelian
category $\R\Contra$.

 In this paper, we are interested in complete, separated right linear
topological rings $\R$ with a \emph{countable} base of neighborhoods
of zero.
 This is the most well-studied and well-behaved case; see
the papers~\cite{PR,Pcoun}.
 In this case, the abelian category $\R\Contra$ is locally
$\aleph_1$\+presentable, and the free contramodule $\R=\R[[\{{*}\}]]$
is its $\aleph_1$\+presentable projective generator.

 Let $\R$ be a complete, separated right linear topological ring.
 For any discrete right $\R$\+module $\N$ and any abelian group $V$,
the abelian group $\fQ=\Hom_\boZ(\N,V)$ has a natural left
$\R$\+contramodule structure.
 The contraaction map $\pi_\fQ\:\R[[\fQ]]\rarrow\fQ$ assigns
to a formal linear combination $t=\sum_{q\in\fQ}r_qq\in\R[[\fQ]]$
the abelian group map $\pi_\fQ(t)\:\N\rarrow V$ given by the rule
$$
 \pi_\fQ(t)(b)=\sum\nolimits_{q\in\fQ}q(br_q)\in V
 \qquad\text{for all $b\in\N$}.
$$
 In this formula, the sum in the right-hand side is finite, since
the family of elements $(r_q\in\R)_{q\in\fQ}$ converges to zero
in the topology of $\R$, while the annihilator of the element
$b\in\N$ is an open right ideal in~$\R$.
 
 Let $\N$ be a discrete right $\R$\+module and $\fP$ be a left
$\R$\+contramodule.
 The \emph{contratensor product} $\N\ocn_\R\fP$ is an abelian group
constructed as the cokernel of (the difference of) the natural pair
of maps
$$
 \N\ot_\boZ\R[[\fP]]\rightrightarrows\N\ot_\boZ\fP.
$$
 Here one map $\N\ot_\boZ\R[[\fP]]\rarrow\N\ot_\boZ\fP$ is induced
by the contraaction map $\pi_\fP\:\R[[\fP]]\rarrow\fP$, while
the other map is the composition $\N\ot_\boZ\R[[\fP]]\rarrow
\N[\fP]\rarrow\N\ot_\boZ\fP$ of the map $\N\ot_\boZ\R[[\fP]]\rarrow
\N[\fP]$ induced by the discrete right action of $\R$ in $\N$ with
the obvious surjective map $\N[\fP]=\N^{(\fP)}\rarrow\N\ot_\boZ\fP$.
 In this context, for any set $X$, the map $\N\ot_\boZ\R[[X]]\rarrow
\N[X]$ induced by the discrete right action of $\R$ in $\N$ is given
by the rule
$$
 b\ot\sum\nolimits_{x\in X}r_xx\longmapsto\sum\nolimits_{x\in X}(br_x)x.
$$
 The expression in the right-hand side denotes an element of $\N[X]$,
i.~e., a finite formal linear combination of elements of $X$ with
the coefficients in~$\N$.
 Once again, this formal linear combination is finite since the family
of elements $(r_x\in\R)_{x\in X}$ converges to zero in the topology
of $\R$, while the annihilator of the element $b\in\N$ is an open
right ideal in~$\R$.

 For any discrete right $\R$\+module $\N$, any left $\R$\+contramodule
$\fP$, and any abelian group $V$, there is a natural isomophism of
abelian groups
$$
 \Hom^\R(\fP,\Hom_\boZ(\N,V))\simeq\Hom_\boZ(\N\ocn_\R\fP,\>V).
$$
 It follows that the contratensor product functor $\ocn_\R$ preserves
all colimits in both of its arguments.
 For any discrete right $\R$\+module $\N$ and any set $X$, there is
a natural isomorphism of abelian groups
$$
 \N\ocn_\R\R[[X]]\simeq\N[X]=\N^{(X)}.
$$

 A left $\R$\+contramodule $\fF$ is said to be \emph{flat} if
the contratensor product with $\fF$ is an exact functor
${-}\ocn_\R\fF\:\Discr\R\rarrow\Ab$.
 The class of flat left $\R$\+contramodules is closed under coproducts
and directed colimits in $\R\Contra$.
 All projective left $\R$\+contramodules are flat.

 Let $\R$ be a complete, separated right linear topological ring and
$\fA\subset\R$ be a closed subgroup.
 Given a left $\R$\+contramodule $\fP$, we denote by $\fA\tim\fP
\subset\fP$ the image of the subgroup $\fA[[\fP]]\subset\R[[\fP]]$
under the contraaction map $\pi_\fP\:\R[[\fP]]\rarrow\fP$.
 So $\fA\tim\fP$ is a subgroup in~$\fP$.

 For any open right ideal $\fI\subset\R$, there is a natural
isomorphism of abelian groups
$$
 (\R/\fI)\ocn_\R\fP\simeq\fP/(\fI\tim\fP).
$$
 For any set $X$ and any closed right ideal $\fJ\subset\R$, one has
$$
 \fJ\tim(\R[[X]])=\fJ[[X]]\subset\R[[X]].
$$
 In particular, for an open right ideal $\fI\subset\R$, one has
$$
 (\R/\fI)\ocn_\R\R[[X]]\simeq(\R/\fI)[X]\simeq\R[[X]]/\fI[[X]].
$$

\Section{Right Linear Topological Modules}
\label{topological-modules-secn}

 Let $\R$ be a complete, separated right linear topological ring.
 A topological $\R$\+module is said to be \emph{separated} 
(respectively, \emph{complete}) if it is separated (resp., complete)
as a topological abelian group.
 A topological right $\R$\+module is said to be \emph{right linear}
if it has a base of neighborhoods of zero consising of open
$\R$\+submodules.
 Following~\cite[Section~2.2]{PSsp}, let us consider the category
$\Modrcs\R$ of complete, separated right linear topological
$\R$\+modules with continuous right $\R$\+linear maps.
 Denote by $\Modrcsom\R\subset\Modrcs\R$ the full subcategory of
all (complete, separated, right linear) topological modules with
a countable base of neighborhoods of zero.
 Obviously, the abelian category of discrete right $\R$\+modules
$\Discr\R$ is a full subcategory in $\Modrcsom\R$.

 Let $\sC$ be a category with directed limits.
 Let us denote by $\varprojlim^\sC$ the limits computed in
the category~$\sC$.
 Then there is a natural functor $\varprojlim\:\Pro(\sC)\rarrow\sC$
taking a pro-object $\bC=\plim_{\gamma\in\Gamma}C_\gamma\in\Pro(\sC)$
to the object $\varprojlim(\bC)=
\varprojlim^\sC_{\gamma\in\Gamma}C_\gamma\in\sC$; that is
$$
 \varprojlim(\plim_{\gamma\in\Gamma}C_\gamma)=
 \varprojlim\nolimits^\sC_{\gamma\in\Gamma}C_\gamma.
$$
 The functor $\varprojlim\:\Pro(\sC)\rarrow\sC$ is right adjoint to
the fully faithful inclusion functor $\sC\rarrow\Pro(\sC)$ (so,
$\varprojlim$ is the coreflector onto $\sC\subset\Pro(\sC)$).

 Let $\sA$ be an abelian category with directed limits (equivalently,
with infinite products).
 Following~\cite[Section~9]{Pextop}, we will say that a pro-object
$\bC\in\Pro(\sA)$ is \emph{limit-epimorphic} if the natural adjunction
morphism $\varprojlim(\bC)\rarrow\bC$ is an epimorphism in the abelian
category $\Pro(\sA)$.
 Equivalently, a pro-object $\bC\in\Pro(\sA)$ is limit-epimorphic if and
only if it can be represented by a downwards directed diagram
$(C_\gamma\in\sA)_{\gamma\in\Gamma}$ indexed by a some directed
poset $\Gamma$ such that the projection morphism
$\varprojlim_{\gamma\in\Gamma}^\sA C_\gamma\rarrow C_\delta$ is
an epimorphism in $\sA$ for every $\delta\in\Gamma$.

 Notice that any limit-epimorphic pro-object in $\sA$ is obviously
strict.
 But the converse is \emph{not} true in general, even when
$\sA=k\Vect$ is the category of vector spaces over a field~$k$
\,\cite[Section~3]{HS}.
 On the other hand, a countably indexed pro-abelian group
$\bC\in\Pro_\omega(\Ab)$ is strict if and only if it is
limit-epimorphic (cf.\ Lemma~\ref{strict-omega-pro-unambigous}).

\begin{prop} \label{spro-discr-described-as-topological-modules}
 Let\/ $\R$ be a complete, separated right linear topological ring.
 In this setting: \par
\textup{(a)} There is a natural fully faithful functor\/
$\Modrcs\R\rarrow\SPro(\Discr\R)$, whose essential image consists
precisely of all the pro-objects in\/ $\Discr\R$ that are
limit-epimorphic when viewed \emph{as pro-objects in $\Modr\R$}. \par
\textup{(b)} There is a natural equivalence of categories\/
$\Modrcsom\R\simeq\SPro_\omega(\Discr\R)$.
\end{prop}

\begin{proof}
 Part~(a): let $\rN$ be a complete, separated right linear topological
right $\R$\+module.
 Then the corresponding strict pro-object in $\Discr\R$ is constructed
as $\plim_{\rU\subset\rN}\rN/\rU$, where $\rU$ ranges over all the open
$\R$\+submodules of~$\rN$.
 By the definition of separatedness and completeness, we have
$\rN=\varprojlim_{\rU\subset\rN}^\Ab\rN/\rU=
\varprojlim_{\rU\subset\rN}^{\Modr\R}\rN/\rU$, and the maps
$\rN\rarrow\rN/\rU$ are epimorphisms in $\Ab$ and $\Modr\R$.
 So $\plim_{\rU\subset\rN}\rN/\rU$ is a limit-epimorphic pro-object
in $\Ab$ and $\Modr\R$.

 Conversely, let $\bcM=\plim_{\gamma\in\Gamma}\M_\gamma$ be
a pro-object in $\Discr\R$ that is limit-epimorphic in $\Modr\R$.
 Put $N=\varprojlim^{\Modr\R}_{\gamma\in\Gamma}\M_\gamma\in\Modr\R$.
 Then we have a natural morphism $\bbp\:N\rarrow\bcM$ in
$\Pro(\Modr\R)$, and it is an epimorphism by assumption.
 Following the explicit proof of
Lemma~\ref{morphism-in-pro-objects-arise-from}(a) that can be
found in~\cite[Section~9]{Pextop}, there is a directed poset
$\Upsilon$ together with a cofinal map of directed posets
$\gamma\:\Upsilon\rarrow\Gamma$ such that the morphism
$\bbp\:N\rarrow\bcM$ is represented by a morphism of
$\Upsilon^\sop$\+indexed diagrams
$(p_\upsilon\:N\to\M_{\gamma(\upsilon)})_{\upsilon\in\Upsilon}$.
 For every $\upsilon\in\Upsilon$, the map $p_\upsilon\:N\rarrow
\M_{\gamma(\upsilon)}$ is an $\R$\+module morphism
from a right $\R$\+module $N$ to a discrete
right $\R$\+module~$\M_{\gamma(\upsilon)}$.
 Put $\N_\upsilon=\im(p_\upsilon)$; then $\N_\upsilon$ is
a discrete right $\R$\+module.

 By Lemma~\ref{lim-in-quotes-preserves-finite-co-limits}(a\+b),
the functor $\plim_{\upsilon\in\Upsilon}\:(\Modr\R)^{\Upsilon^\sop}
\rarrow\Pro(\Modr\R)$ preserves kernels and cokernels, hence also
images, of all morphisms.
 Since $\bbp\:N\rarrow\bcM$ is an epimorphism in $\Pro(\Modr\R)$, it
follows that the termwise injective map of $\Upsilon^\sop$\+indexed
diagrams $(\N_\upsilon)_{\upsilon\in\Upsilon}\rarrow
(\M_{\gamma(\upsilon)})_{\upsilon\in\Upsilon}$ induces an isomorphism
in $\Pro(\Modr\R$).
 By Lemma~\ref{pro-objects-full-subcategory}, it follows that
the same termwise embedding of diagrams induces also an isomorphism
in $\Pro(\Discr\R)$.

 Obviously, the transition morphisms in the diagram
$(\N_\upsilon)_{\upsilon\in\Upsilon}$ are surjective; so
the pro-object $\bcM\simeq\plim_{\gamma\in\Gamma}\N_\gamma$ belongs to
$\SPro(\Discr\R)\subset\Pro(\Discr\R)$.
 It remains to denote by $\rN$ the right $\R$\+module $N$ endowed with
the topology in which the kernels $\rU_\upsilon$ of the surjective
$\R$\+module maps $N\rarrow\N_\upsilon$ form a base of neighborhoods
of zero.
 Since the $\R$\+modules $\N_\upsilon$ are discrete, the action of
$\R$ in $\rN$ is continuous.
 The $\R$\+submodules $\rU_\upsilon\subset\rN$ form a cofinal subdiagram
in the downwards directed diagram of all open $\R$\+submodules in~$\rN$,
and we have $\rN=\varprojlim_{\upsilon\in\Upsilon}\rN/\rU_\upsilon$.
 Hence $\rN$ is a complete, separated right linear topological
right $\R$\+module.

 It is clear from the constructions that they provide two mutually
inverse equivalences between $\Modrcs\R$ and the full subcategory of
$\Pro(\Discr\R)$ consisting of all the pro-objects that are
limit-epimorphic as pro-objects in $\Modr\R$.
 We have also shown that all such pro-objects are strict as pro-objects
in $\Discr\R$.

 The proof of part~(b) is similar, with the main difference that every
pro-object belonging to $\SPro_\omega(\Modr\R)$ is limit-epimorphic.
\end{proof}

 It is easy to see that all limits (and in particular, infinite
products) exist in the category $\Modrcs\R$.
 Given a small category $\Gamma$ and a functor $\rN\:\Gamma^\sop
\rarrow\Modrcs\R$, the projective limit $\varprojlim_{\gamma\in\Gamma}
\rN(\gamma)$ is constructed as the limit of the same diagram computed
in the category of abstract $\R$\+modules $\Modr\R$, endowed with
the projective limit topology.
 So finite intersections of the preimages of open submodules in
$\rN(\gamma)$, \ $\gamma\in\Gamma$, form a base of neighborhoods
of zero in $\varprojlim_{\gamma\in\Gamma}\rN(\gamma)$.

 Similarly, all countable limits (and in particular, countable
products) exist in the category $\Modrcsom\R$.
 The fully faithful inclusion functor $\Modrcsom\R\rarrow\Modr\R$
preserves countable limits.

\begin{lem} \label{products-in-linear-topol-modules-and-pro-objects}
\textup{(a)} The full subcategory\/ $\Pro(\Discr\R)$ is closed
under infinite products in\/ $\Pro(\Modr\R)$.
 Consequently, the full subcategory\/ $\SPro(\Discr\R)$ is closed under
infinite products in\/ $\SPro(\Modr\R)$. \par
\textup{(b)} For any ring $R$, the full subcategory of limit-epimorphic
objects is closed under infinite products in\/ $\Pro(\Modr R)$
and\/ $\SPro(\Modr R)$. \par
\textup{(c)} The full subcategory\/ $\Modrcs\R$ is closed under
infinite products in\/ $\Pro(\Discr\R)$ and\/ $\SPro(\Discr\R)$.
\end{lem}

\begin{proof}
 Part~(a): more generally, for any two categories $\sC$ and $\sD$ with
finite products, and any functor $F\:\sC\rarrow\sD$ preserving finite
products, the induced functor $\Pro(F)\:\Pro(\sC)\rarrow\Pro(\sD)$
preserves infinite products.
 This is clear from the explicit proof of
Corollary~\ref{products-of-pro-objects}(a).
 The second assertion of part~(a) follows from the first one in view of
Corollary~\ref{products-preserve-strict-pro-objects}.

 Part~(b): one observes that the functor $\varprojlim\:\Pro(\Modr R)
\rarrow\Modr R$ preserves all limits (being a right adjoint) and
the products of epimorphisms in $\Pro(\Modr R)$ are epimorphisms again
(by Lemma~\ref{products-exact-in-pro-objects}).
 The key point is that, for any family of objects $C_\upsilon$,
\,$\upsilon\in\Upsilon$ in an abelian category $\sA$ with infinite
products, the natural morphism
$\prod_{\upsilon\in\Upsilon}^\sA C_\upsilon\rarrow
\prod_{\upsilon\in\Upsilon}^{\Pro(\sA)}C_\upsilon$ from the product
of $C_\upsilon$ taken in the category $\sA$ (but viewed as an object
of $\Pro(\sA)$) to the same product taken in the category $\Pro(\sA)$
is an epimorphism in $\Pro(\sA)$.
 This is clear from the explicit proof of
Corollary~\ref{products-of-pro-objects}(a).

 Part~(c) follows from parts~(a) and~(b) in view of
Proposition~\ref{spro-discr-described-as-topological-modules}(a).
 Alternatively, one can see directly that the construction of
infinite products in $\Modrcs\R$ as per the paragraphs preceding
this lemma agrees with the construction of infinite products in
$\Pro(\Discr\R)$ as per the explicit proof of
Corollary~\ref{products-of-pro-objects}(a).
\end{proof}

 Lemma~\ref{products-in-linear-topol-modules-and-pro-objects}(c)
says that one can equivalently compute infinite products in
$\Modrcs\R$ (as described in the preceding paragraps) or the same
infinite products in $\SPro(\Discr\R)$ (as described in
Corollaries~\ref{products-of-pro-objects}(a)
and~\ref{products-preserve-strict-pro-objects}).
 Similarly, it is clear from
Proposition~\ref{spro-discr-described-as-topological-modules}(b)
that one can equivalently compute countable products in
$\Modrcsom\R$ (as described in the preceding paragraphs) or in
$\SPro_\omega(\Discr\R)$ (as described in
Corollaries~\ref{products-of-pro-objects}(b)
and~\ref{products-preserve-strict-pro-objects}).

\medskip

 Before we finish this section, let us say a few words about
the exact category structures on the categories of topological
right $\R$\+modules.
 The additive category $\SPro_\omega(\Discr\R)$ is quasi-abelian,
and its quasi-abelian exact structure coincides with the exact structure
inherited from the abelian exact structure of $\Pro_\omega(\Discr\R)$,
by Proposition~\ref{spro-omega-quasi-abelian-characterization}\,%
(1)\,$\Rightarrow$\,(3).
 According to
Proposition~\ref{spro-discr-described-as-topological-modules}(b), it
follows that the additive category $\Modrcsom\R\simeq
\SPro_\omega(\Discr\R)$ is quasi-abelian, and its quasi-abelian
exact structure coincides with the exact structure inherited from
$\Pro_\omega(\Discr\R)$.
 We will consider $\Modrcsom\R$ as an exact category with this exact
structure.

 The admissible short exact sequences in $\Modrcsom\R$ are precisely
all the short sequences $0\rarrow\rL\rarrow\rM\rarrow\rN\rarrow0$
in $\Modrcsom\R$ satisfying the following conditions: the sequence
must be exact in $\Modr\R$, the injective morphism $\rL\rarrow\rM$ must
be closed, and the surjective morphism $\rM\rarrow\rN$ must be open.
 The key observation here is that the quotient topology on the quotient
group of a topological abelian group by a closed subgroup is
(separated and) complete whenever the subgroup has a countable base of
neighborhoods of zero~\cite[Proposition~1.4]{Pextop}
(see also~\cite[Proposition~11.6]{Pextop}).
 We leave the straightforward details to the reader.

 The additive category $\SPro(\Discr\R)$ is quasi-abelian as well,
and its quasi-abelian exact structure coincides with the exact structure
inherited from the abelian exact structure of $\Pro(\Discr\R)$,
by Proposition~\ref{spro-quasi-abelian-characterization}\,%
(1)\,$\Rightarrow$\,(3).
 However, the additive category $\Modrcs\R$ is \emph{not} quasi-abelian
already in the case when $\R=k$ is a discrete
field~\cite[Corollary~8.6]{Pextop}.
 Furthermore, in the case of a discrete field $\R=k$ already,
the full subcategory $\Modrcs\R$ (embedded as per
Proposition~\ref{spro-discr-described-as-topological-modules}(a)) does
\emph{not} inherit an exact category structure from $\Pro(\Discr\R)$
or $\SPro(\Discr\R)$ \,\cite[Proposition~9.5(b)]{Pextop}.
 Accordingly, we do \emph{not} endow $\Modrcs\R$ with any exact
category structure in this paper.
 We refer to~\cite[Corollary~8.8(b), Conclusion~8.9, and
Sections~10--11 and~13]{Pextop} for further discussion.

\Section{Pontryagin Duality and Pro-Contratensor Product}

 Let $\R$ be a complete, separated right linear topological ring.
 For any left $\R$\+contramodule $\fP$, consider the natural map of
abelian groups
$$
 \lambda_{\R,\fP}\:\fP\lrarrow
 \varprojlim\nolimits_{\fI\subset\R}\fP/(\fI\tim\fP),
$$
where the projective limit in the category of abelian groups is taken
over all the open right ideals $\fI\subset\R$.
 The $\R$\+contramodule $\fP$ is called \emph{separated} if
the map~$\lambda_{\R,\fP}$ is injective, and \emph{complete} if
$\lambda_{\R,\fP}$~is surjective.

 All free left $\R$\+contramodules $\R[[X]]$ are complete and separated,
since
$$
 \R[[X]]=\varprojlim\nolimits_{\fI\subset\R}(\R/\fI)[X]=
 \varprojlim\nolimits_{\fI\subset\R}\R[[X]]/\fI[[X]]=
 \varprojlim\nolimits_{\fI\subset\R}\R[[X]]/(\fI\tim\R[[X]]).
$$
 Hence all projective left $\R$\+contramodules are complete and
separated, too.

 Let $\fP$ and $\fQ$ be two left $\R$\+contramodules.
 Consider the abelian group $\Hom^\R(\fP,\fQ)$ of all morphisms
$\fP\rarrow\fQ$ in the category $\R\Contra$, and endow it with
the following topology.
 For every finite subset $F\subset\fP$ and every open right ideal
$\fI\subset\R$, denote by $\fV_{F,\fI}\subset\Hom^\R(\fP,\fQ)$
the subgroup consisting of all the $\R$\+contramodule morphisms
$f\:\fP\rarrow\fQ$ such that $f(F)\subset\fI\tim\fQ$.
 By definition, the subgroups $\fV_{F,\fI}$ form a base of
neighborhoods of zero in the topology of $\Hom^\R(\fP,\fQ)$.

\begin{lem} \label{contramodule-hom-group-complete-and-separated}
 Assume that the left\/ $\R$\+contramodule\/ $\fQ$ is complete and
separated.
 Then the topological abelian group\/ $\Hom^\R(\fP,\fQ)$ is complete
and separated, too.
\end{lem}

\begin{proof}
 This a slightly more general version of~\cite[Corollary~7.7]{PS1}.
 To check that $\Hom^\R(\fP,\fQ)$ is separated, suppose given
a nonzero $\R$\+con\-tra\-mod\-ule morphism $f\:\fP\rarrow\fQ$.
 Let $p\in\fP$ be an element such that $q=f(p)\ne0$ in $\fQ$, and
let $\fI\subset\R$ be an open right ideal such that
$q\notin\fI\tim\fQ$.
 Consider the singleton subset $F=\{p\}\subset\fP$.
 Then one has $f\notin\fV_{F,\fI}\subset\Hom^\R(\fP,\fQ)$.

 To check that $\Hom^\R(\fP,\fQ)$ is complete, notice that
the completion of the topological abelian group
$\Hom^\R(\fP,\fQ)$ can be computed as the projective limit
$\varprojlim_{F,\fI}\Hom^\R(\fP,\fQ)/\fV_{F,\fI}$ taken over
the directed poset of all pairs $(F,\fI)$ with the partial order
$(F',\fI')\le(F'',\fI'')$ if $F'\subset F''$ and $\fI'\supset\fI''$.
 Suppose given a compatible family of elements in the quotient groups
$h_{F,\fI}\in\Hom^\R(\fP,\fQ)/\fV_{F,\fI}$ defined for all finite
subsets $F\subset\fP$ and all open right ideals $\fI\subset\R$.

 Given an element $p\in\fP$, put $F=\{p\}$, and consider the natural
evaluation maps $\ev_{p,\fI}\:\Hom^\R(\fP,\fQ)/\fV_{F,\fI}\rarrow
\fQ/(\fI\tim\fQ)$ taking a coset $h+\fV_{F,\fI}$ to the coset
$h(p)+\fI\tim\fQ$.
 Clearly, the abelian group map $\ev_{p,\fI}$ is well-defined by
this rule.
 Now the compatible collection of cosets $\ev_{p,\fI}(h_{F,\fI})
\in\fQ/(\fI\tim\fQ)$ defines an element of the projective limit
$\varprojlim_{\fI\subset\R}\fQ/(\fI\tim\fQ)$, which by assumption
corresponds to a unique element $q\in\fQ$.
 Set $f(p)=q$.

 The key step is to check that $f\:\fP\rarrow\fQ$ is
an $\R$\+contramodule morphism.
 Given a zero-convergent formal linear combination
$t=\sum_{p\in\fP}r_pp\in\R[[\fP]]$, we need to show that the equation
$$
 f(\pi_\fP(t))=\pi_\fQ\bigl(\R[[f]](t)\bigr)
$$
holds in~$\fQ$.
 Here $\R[[f]]\:\R[[\fP]]\rarrow\R[[\fQ]]$ is the map defined in
Section~\ref{contramodules-preliminaries-secn}.

 For this purpose, it suffices to check that the desired equation
holds modulo the subgroup $\fI\tim\fQ\subset\fQ$ for every open right
ideal $\fI\subset\R$.
 Now there is a finite subset $G\subset\fP$ such that $r_p\in\fI$
for all $p\in\fP\setminus G$.
 Put $p_t=\pi_\fP(t)\in\fP$ and $F=G\cup\{p_t\}\subset\fP$.
 Furthermore, there exists an open right ideal $\fJ\subset\R$
such that $\fJ\subset\fI$ and $r_p\fJ\subset\fI$ for all $p\in G$.
 Consider the coset $h_{F,\fJ}\in\Hom^\R(\fP,\fQ)/\fV_{F,\fJ}$,
and let $g\in\Hom^\R(\fP,\fQ)$ be one of its representatives.
 By the assumption of compatibility imposed on our collection of
cosets, we have $f(p)\equiv g(p)$ modulo $\fJ\tim\fQ$ for all $p\in F$.

 Put $t'=\sum_{p\in G}r_pp$ and $t''=\sum_{p\notin G}r_pp$; so
$t=t'+t''$.
 It follows that $\R[[f]](t)=\R[[f]](t')+\R[[f]](t'')$ and
$\R[[g]](t)=\R[[g]](t')+\R[[g]](t'')$, because $\R[[f]]$ and $\R[[g]]$
are abelian group homomorphisms.
 The contraaction map $\pi_\fQ\:\R[[\fQ]]\rarrow\fQ$ is a contramodule
morphism (from the free $\R$\+contramodule $\R[[\fQ]]$); so it is
an abelian group map, too.
 Now we can compute that
\begin{multline*}
 f(\pi_\fP(t)) \equiv g(\pi_\fP(t))=\pi_\fQ\bigl(\R[[g]](t)\bigr) =
 \pi_\fQ\bigl(\R[[g]](t'+t'')\bigr) \\
 = \pi_\fQ\bigl(\R[[g]](t')\bigr)+\pi_\fQ\bigl(\R[[g]](t'')\bigr)
 \equiv \pi_\fQ\bigl(\R[[g]](t')\bigr) \\ \equiv
 \pi_\fQ\bigl(\R[[f]](t')\bigr) \equiv
 \pi_\fQ\bigl(\R[[f]](t')\bigr)+\pi_\fQ\bigl(\R[[f]](t'')\bigr) \\
 = \pi_\fQ\bigl(\R[[f]](t'+t'')\bigr)=\pi_\fQ\bigl(\R[[f]](t)\bigr)
\end{multline*}
modulo $\fI\tim\fQ$.
 Indeed, we have $\pi_\fP(t)\in F$, \  $g\in\Hom^\R(\fP,\fQ)$, \
$\R[[g]](t'')\in\fI[[\fQ]]$, \ $\R[[f]](t')-\R[[g]](t')\in
\sum_{p\in G}r_p(\fJ\tim\fQ)\subset\R[[\fQ]]$, and
$\R[[f]](t'')\in\fI[[\fQ]]$.

 Finally, it follows from the constructions involved that
$\Hom^\R(\fP,\fQ)/\fV_{F,\fI}\subset\Hom_\Sets(F,\>\fQ/(\fI\tim\fQ))
\simeq\prod_{p\in F}\fQ/(\fI\tim\fQ)$.
 Therefore, we have $h_{F,\fI}=f+\fV_{F,\fI}\in
\Hom^\R(\fP,\fQ)/\fV_{F,\fI}$ for all finite subsets $F\subset\fP$
and open right ideals $\fI\subset\R$.
 In other words, the image of the element $f\in\Hom^\R(\fP,\fQ)$
under the completion map $\Hom^\R(\fP,\fQ)\rarrow
\varprojlim_{F,\fI}(\Hom^\R(\fP,\fQ)/\fV_{F,\fI})$ is equal to our
original element $(h_{F,\fI})_{F,\fI}\in
\varprojlim_{F,\fI}(\Hom^\R(\fP,\fQ)/\fV_{F,\fI})$, as desired.
\end{proof}

\begin{rem}
There is also another, less explicit, but more elegant way to prove Lemma~\ref{contramodule-hom-group-complete-and-separated}. Namely, $\Hom^\R(\fP,\fQ)$ is a subset of $\Hom_{\Sets}(\fP,\fQ)=\fQ^\fP$ and the latter is a topological abelian groups with the product topology (when viewed as a direct product of copies of $\fQ$). If $\fQ$ is complete and separated, so is $\Hom_{\Sets}(\fP,\fQ)$. It remains to observe that $\Hom^\R(\fP,\fQ)$ is a closed subgroup with respect to this topology. We leave working out the details to the reader here.
\end{rem}

 Let $\R$ be a complete, separated right linear topological ring.
 Then the free right $\R$\+module $\R$ is a complete, separated
right linear topological right $\R$\+module.
 So $\R$ itself is an object of the category $\Modrcs\R$.
 When the topological ring $\R$ has a countable base of neighborhoods
of zero, the free right $\R$\+module $\R$ belongs to the full
subcategory $\Modrcsom\R\subset\Modrcs\R$.

 Accordingly, for any complete, separated right linear topological
ring $\R$, we can consider the full subcategory $\Prod(\R)\subset
\Modrcs\R$ (in the notation introduced in the paragraph preceding
Proposition~\ref{projectives-as-summands-of-products}).
 When $\R$ has a countable base of neighborhoods of zero, we also
have a full subcategory $\Prod_\omega(\R)\subset\Modrcsom\R$.
 Notice that the notation $\Prod_\omega(\R)$ is unambiguous here
and $\Prod_\omega(\R)\subset\Prod(\R)$, since the full subcategory
$\Modrcsom\R$ is closed under countable products in $\Modrcs\R$.

 Given an abelian or exact category $\sE$, we denote by $\sE_\proj
\subset\sE$ the full subcategory of projective objects in~$\sE$.
 In particular, $\R\Contra_\proj$ denotes the full subcategory of
projective contramodules in $\R\Contra$.

 For any complete, separated topological right $\R$\+module $\rM$,
the set/abelian group/left $\R$\+module $\Hom_\R^\cont(\rM,\R)$ of
continuous right $\R$\+module homomorphisms $\rM\rarrow\R$ is
endowed with a left $\R$\+contramodule structure as explained
in~\cite[the paragraph preceding Theorem~3.1]{PSsp}.

\begin{thm} \label{pontryagin-duality-theorem}
 Let\/ $\R$ be a complete, separated right linear topological ring.
 Then there is a natural anti-equivalence between the additive
category of projective left\/ $\R$\+contramodules\/ $\R\Contra_\proj$
and the full subcategory in\/ $\Modrcs\R$ formed by the direct
summands of infinite products of copies of the topological right\/
$\R$\+module\/ $\R$,
$$
 \R\Contra_\proj\simeq(\Prod\R)^\sop.
$$
\end{thm}

\begin{proof}
 This is~\cite[Theorem~3.1]{PSsp}.
 The functor $\R\Contra_\proj\rarrow(\Prod\R)^\sop$ assigns to
a projective left $\R$\+contramodule $\fP$ the topological abelian
group $\Hom^\R(\fP,\R)$ from the paragraph preceding
Lemma~\ref{contramodule-hom-group-complete-and-separated}.
 The right action of $\R$ in itself induces the right $\R$\+module
structure on $\Hom^\R(\fP,\R)$.
 The functor $(\Prod\R)^\sop\rarrow\R\Contra_\proj$ assigns to
a topological right $\R$\+module $\rP$ belonging to $\Prod(\R)$
the left $\R$\+con\-tra\-mod\-ule $\Hom_\R^\cont(\rP,\R)$.
\end{proof}

 Let $\R$ be a complete, separated topological ring with a countable
base of neighborhoods of zero.
 A left $\R$\+contramodule $\fP$ is said to be \emph{countably
generated} if it is a quotient contramodule of a free left
$\R$\+contramodule $\R[[X]]$ spanned by a countable set~$X$.
 Let us denote by $\R\Contra_\proj^\omega\subset\R\Contra_\proj$
the full subcategory of countably generated projective left
$\R$\+contramodules.

\begin{cor} \label{countable-pontryagin-duality}
 Let\/ $\R$ be a complete, separated topological ring with a countable
base of neighborhoods of zero.
 Then the anti-equivalence of categories from
Theorem~\ref{pontryagin-duality-theorem} restricts to a natural
anti-equivalence
$$
 \R\Contra_\proj^\omega\simeq(\Prod_\omega\R)^\sop
$$
between the additive category of countably generated projective
left\/ $\R$\+contramodules and the full subcategory in\/
$\Modrcsom\R$ formed by the direct summands of countable products
of copies of the topological right\/ $\R$\+module~$\R$.
\end{cor}

\begin{proof}
 This can be seen from~\cite[proof of Theorem~3.1]{PSsp}, or deduced
from Theorem~\ref{pontryagin-duality-theorem} using the facts that
the category equivalence from the latter theorem takes the object $\R=
\R[[\{{*}\}]]\in\R\Contra_\proj$ to the object $\R\in\Prod(\R)^\sop$,
and that any category equivalence preserves coproducts.
\end{proof}

 Let $\R$ be a complete, separated right linear topological ring.
 The functor of contratensor product
$$
 \ocn_\R\:\Discr\R\times\R\Contra\lrarrow\Ab
$$
was constructed in Section~\ref{contramodules-preliminaries-secn}.
 Passing to the pro-objects, we obtain the functor
$$
 \ocn_\R^\pro\:\Pro(\Discr\R)\times\R\Contra\lrarrow\Pro(\Ab)
$$
defined by the rule
$$
 (\plim_{\gamma\in\Gamma}\N_\gamma)\ocn_\R^\pro\fP
 = \plim_{\gamma\in\Gamma}(\N_\gamma\ocn_\R\fP)
$$
for all pro-objects $\bcN=\plim_{\gamma\in\Gamma}\N_\gamma\in
\Pro(\Discr\R)$ and all contramodules $\fP\in\R\Contra$.
 The functor ${-}\ocn_\R\fP$ takes epimorphisms in $\Discr\R$ to
epimorphisms in $\Ab$, so the functor~$\ocn_\R^\pro$ restricts
to a functor
$$
 \ocn_\R^\pro\:\SPro(\Discr\R)\times\R\Contra\lrarrow\SPro(\Ab).
$$

 Restricting the functor $\ocn_\R^\pro$ further to the full subcategory
$\Modrcs\R\subset\Pro(\Discr\R)$ (as per
Proposition~\ref{spro-discr-described-as-topological-modules}(a))
and postcomposing it with the functor $\varprojlim\:\Pro(\Ab)
\rarrow\Ab$, we obtain the functor of \emph{pro-contratensor product}
{\hbadness=1900
$$
 \proocn_\R\:\Modrcs\R\times\R\Contra\rarrow\Ab
$$
given} by the formula
$$
 \rN\proocn_\R\fP=
 \varprojlim\nolimits_{\rU\subset\rN}^\Ab(\rN/\rU\ocn_\R\fP).
$$
 Here the projective limit in the category of abelian groups $\Ab$
is taken over all the open $\R$\+submodules $\rU\subset\rN$.

 In particular, the functor $\R\proocn_\R{-}\,\:\R\Contra\rarrow\Ab$
takes any left $\R$\+con\-tra\-mod\-ule $\fP$ to the abelian group
$$
 \R\proocn_\R\fP=\varprojlim\nolimits_{\fI\subset\R}
 (\fP/\fI\tim\fP)=\lambda_{\R,\fI}(\fP).
$$

\begin{lem} \label{proocn-preserves-products}
 For any left\/ $\R$\+contramodule\/ $\fP$, the functor\/
${-}\proocn_\R\fP\:\Modrcs\R\rarrow\Ab$ preserves infinite products.
\end{lem}

\begin{proof}
 The functor ${-}\ocn_\R^\pro\fP\:\Pro(\Discr\R)\rarrow\Pro(\Ab)$
preserves infinite products for the reason explained in the proof of
Lemma~\ref{products-in-linear-topol-modules-and-pro-objects}(a).
 Furthermore, the inclusion functor $\Modrcs\R\rarrow\Pro(\Discr\R)$
preserves products by
Lemma~\ref{products-in-linear-topol-modules-and-pro-objects}(c),
while the functor $\varprojlim\:\Pro(\Ab)\rarrow\Ab$ preserves
products since it is a right adjoint functor.
\end{proof}

\begin{lem} \label{proocn-with-flat-exact-for-countable}
 Let\/ $\fF$ be a \emph{flat} left\/ $\R$\+contramodule.
 Then the functor of pro-contratensor product\/ ${-}\proocn_\R\fF$
restricted to the full subcategory\/ $\Modrcsom\R\subset\Modrcs\R$
is an exact functor\/ ${-}\proocn_\R\fF\:\Modrcsom\R\rarrow\Ab$ from
the exact category\/ $\Modrcsom\R$ to the abelian category\/~$\Ab$.
 In other words, the functor\/ ${-}\proocn_\R\fF$ takes admissible
short exact sequences in\/ $\Modrcsom\R$ to short exact sequences
in~$\Ab$.
\end{lem}

\begin{proof}
 First of all, we recall that the (quasi-abelian) exact category
structure on $\Modrcsom\R$ is inherited from the abelian exact category
structure of $\Pro_\omega(\Discr\R)$ via the fully faithful functor
$\Modrcsom\R\simeq\SPro_\omega(\Discr\R)\rarrow\Pro_\omega(\Discr\R)$.

 Most generally, the functor ${-}\ocn_\R^\pro\fF\:\Pro(\Discr\R)\rarrow
\Pro(\Ab)$ is exact, as one can see from the description of short
exact sequences in $\Pro(\Discr\R)$ provided by
Corollary~\ref{pro-objects-abelian}(a).
 In particular, the functor ${-}\ocn_\R^\pro\fF\:\SPro(\Discr\R)
\rarrow\SPro(\Ab)$ takes admissible short exact sequences in
$\SPro(\Discr\R)$ to admissible short exact sequences in $\SPro(\Ab)$
(cf.\ Proposition~\ref{strict-pro-objects-closed-under-extensions}(a)).

 Similarly, in the context of countably indexed diagrams,
the functor ${-}\ocn_\R^\pro\nobreak\fF\:\allowbreak
\Pro_\omega(\Discr\R)\rarrow\Pro_\omega(\Ab)$ is exact, as one can see
from the description of short exact sequences in $\Pro_\omega(\Discr\R)$
provided by Corollary~\ref{pro-objects-abelian}(b).
 In particular, the functor ${-}\ocn_\R^\pro\fF\:\SPro_\omega(\Discr\R)
\rarrow\SPro_\omega(\Ab)$ takes admissible short exact sequences in
$\SPro_\omega(\Discr\R)$ to admissible short exact sequences in
$\SPro_\omega(\Ab)$.

 It remains to point out that the functor $\varprojlim\:\Pro(\Ab)
\rarrow\Ab$ is exact on the exact category $\SPro_\omega(\Ab)
\subset\Pro(\Ab)$.
 Indeed, it is clear from
Proposition~\ref{strict-pro-objects-closed-under-extensions}(b)
that the functor $\varprojlim\:\SPro_\omega(\Ab)\rarrow\Ab$ takes
admissible short exact sequences in $\SPro_\omega(\Ab)$ to short
exact sequences in~$\Ab$.
\end{proof}

\begin{ex}
 The assertion of Lemma~\ref{proocn-with-flat-exact-for-countable}
and its proof concern the preservation of admissible short exact
sequences in the quasi-abelian exact category
$\Modrcsom\R\simeq\SPro_\omega(\Discr\R)$ by the functors
${-}\proocn_\R\fF$ and ${-}\ocn_\R^\pro\fF$.
 The following counterexample shows, however, that the functors
${-}\ocn_\R^\pro\fF\:\SPro_\omega(\Discr\R)\rarrow\SPro_\omega(\Ab)$
and ${-}\proocn_\R\nobreak\fF\:\Modrcsom\R\rarrow\Ab$ do \emph{not}
preserve kernels of morphisms, for a flat $\R$\+con\-tra\-mod\-ule
$\fF$ in general.

 It suffices to consider the case of a discrete ring $\R=\boZ$,
the ring of integers.
 Then we have $\Discr\R=\Ab=\R\Contra$.
 Choose a prime number~$p$.
 Let $(E_n)_{n\in\omega}$ denote the $\omega^\sop$\+indexed diagram
of abelian groups $E_n=\boZ/p^n\boZ$, with the obvious surjective
triansition maps $E_{n+1}\rarrow E_n$.
 Let $(D_n)_{n\in\omega}$ denote the constant $\omega^\sop$\+indexed
diagram $D_n=\boZ$, with the transition maps $\id_\boZ\:D_{n+1}
\rarrow D_n$.
 Then there is an obvious termwise surjective morphism of diagrams
$(g_n)_{n\in\omega}\:(E_n)_{n\in\omega}\rarrow(D_n)_{n\in\omega}$.
 The kernel of the morphism $(g_n)_{n\in\omega}$ in the category
$\Ab^{\omega^\sop}$ is the diagram $(B_n)_{n\in\omega}$ with
$B_n=p^n\boZ$ and the identity inclusions $p^{n+1}\boZ\rarrow p^n\boZ$
as the transition maps.

 Put $\bD=\plim_{n\in\omega}D_n$, \ $\bE=\plim_{n\in\omega}E_n$,
$\bbg=\plim_{n\in\omega}g_n$, and $\bB=\plim_{n\in\omega}B_n$.
 Then $\bB$ is the kernel of the morphism~$\bbg$ in
the category $\Pro_\omega(\Ab)$.
 Following the proof of
Proposition~\ref{spro-omega-quasi-abelian-characterization},
the kernel of the morphism~$\bbg$ in $\SPro_\omega(\Ab)$ is computed as
the coreflection of the object $\bB$ to the full subcategory
$\SPro_\omega(\Ab)\subset\Pro_\omega(\Ab)$.
 One can immediately see that this coreflection $\bC$ is the zero
object, $\bC=\plim_{n\in\omega}C_n=0$, where $C_n=0$ for all
$n\in\omega$.
 So $\ker(\bbg)=0$ in $\SPro_\omega(\Ab)$.

 Equivalently, denote by $\rD=\varprojlim_{n\in\omega}D_n$ and
$\rE=\varprojlim_{n\in\omega}E_n$ the objects of the category
$\Modrcsom\R$ (i.~e., the topological abelian groups) corresponding
to the pro-objects $\bD$ and $\bE$ under the equivalence of
categories $\SPro_\omega(\Discr\R)\simeq\Modrcsom\R$ from
Proposition~\ref{spro-discr-described-as-topological-modules}(b).
 Then $\rD$ is a discrete abelian group $\rD=\boZ$, while
$\rE=\boZ_p$ is the topological abelian group of $p$\+adic integers.
 So the continuous morphism $g=\varprojlim_{n\in\omega}g_n\:
\rD\rarrow\rE$ is injective.
 Once again, it is clear from the discussion of limits in the category
$\Modrcsom\R$ in Section~\ref{topological-modules-secn} that
$\ker(g)=0$ in $\Modrcsom\R$.

 On the other hand, the contratensor product functor $\ocn_\R$ is
just the tensor product of abelian groups $\ot_\boZ\:\Ab\times\Ab
\rarrow\Ab$ in our case, and the flat $\R$\+contramodules are simply
the flat/torsion-free abelian groups.
 Consider the torsion-free abelian group of rational numbers $\fF=\boQ$.
 Then we have $\bE\ocn_\R^\pro\fF=0$ in $\SPro_\omega(\Ab)$,
since $\boZ/p^n\boZ\ot_\boZ\boQ=0$ for all $n\in\omega$.
 The pro-object $\bD\ocn_\R^\pro\fF$ is just the constant pro-object
$\boQ\in\Ab\subset\SPro_\omega(\Ab)$.
 So the morphism $\bbg\ocn_\R^\pro\fF$ vanishes, and its kernel in
$\SPro_\omega(\Ab)$ is $\boQ\in\Ab\subset\SPro_\omega(\Ab)$.
 Thus $\ker(\bbg)\ocn_\R^\pro\fF=0\ne\boQ=\ker(\bbg\ocn_\R^\pro\fF)$
in $\SPro_\omega(\Discr\R)$ and $\SPro_\omega(\Ab)$.

 Similarly, the pro-contratensor product functor is computed in
the situation at hand as (counterintuitively!)
$\rE\proocn_\R\fF=0\in\Ab$ and $\rD\proocn_\R\fF=\boQ\in\Ab$.
 So the morphism $g\proocn_\R\fF$ vanishes, and its kernel
in $\Ab$ is the abelian group~$\boQ$.
 Thus $\ker(g)\proocn_\R\fF=0\ne\boQ=\ker(g\proocn_\R\fF)$ in~$\Ab$.

 In other words, Lemma~\ref{proocn-with-flat-exact-for-countable}
and its proof tell us (in particular) that the functors
${-}\ocn_\R^\pro\fF\:\SPro_\omega(\Discr\R)\rarrow\SPro_\omega(\Ab)$
and ${-}\proocn_\R\nobreak\fF\:\Modrcsom\R\rarrow\Ab$ take admissible
monomorphisms to (admissible) monomorphisms.
 The counterexample above shows that the same functors do \emph{not}
take monomorphisms to monomorphisms.
\end{ex}

 The following proposition is the main result of this section.

\begin{prop} \label{contramodule-hom-and-pro-contratensor}
 Let\/ $\R$ be a complete, separated topological ring.
 Let\/ $\fP$ be a projective left\/ $\R$\+contramodule.
 Then, for any complete, separated left\/ $\R$\+contramodule\/ $\fQ$,
there is a natural isomorphism of abelian groups
$$
 \Hom^\R(\fP,\fQ)\simeq\Hom^\R(\fP,\R)\proocn_\R\fQ,
$$
where the complete, separated right linear topology on the right\/
$\R$\+module\/ $\Hom^\R(\fP,\R)$ was constructed in the paragraph
before Lemma~\ref{contramodule-hom-group-complete-and-separated}.
\hfuzz=3.5pt
\end{prop}

\begin{proof}
 Let us first construct a natural map of abelian groups
$$
 \psi=\psi_{\fP,\fQ}\:\Hom^\R(\fP,\R)\proocn_\R\fQ
 \lrarrow\Hom^\R(\fP,\fQ)
$$
for any left $\R$\+contramodule $\fP$ and complete, separated
left $\R$\+contramodule~$\fQ$.
 The left-hand side is the projective limit of abelian groups
$$
 \Hom^\R(\fP,\R)\proocn_\R\fQ=
 \varprojlim\nolimits_{F,\fI}
 ((\Hom^\R(\fP,\R)/\fV_{F,\fI})\ocn_\R\fQ)
$$
taken over all finite subsets $F\subset\fP$ and all open right
ideals $\fI\subset\R$.
 The quotient $\Hom^\R(\fP,\R)/\fV_{F,\fI}$ is a discrete right
$\R$\+module, because it is a submodule of the finite direct sum
$(\R/\fI)^F$ of $F$ copies of~$\R/\fI$.

 Suppose given an element $l\in\Hom^\R(\fP,\R)\proocn_\R\fQ$.
 The projective limit element $l$~is a compatible family of elements
$l_{F,\fI}\in(\Hom^\R(\fP,\R)/\fV_{F,\fI})\ocn_\R\fQ$.
 For every element $p\in\fP$, we need to construct an element
$\psi(l)(p)\in\fQ$.
 By assumption, we have $\fQ=\varprojlim_\fI\fQ/(\fI\tim\fQ)$;
so it suffices to construct a compatible family of cosets in
$\fQ/(\fI\tim\fQ)$ for all open right ideals $\fI\subset\R$.

 Put $F=\{p\}$.
 Then we have a natural injective evaluation morphism of discrete
right $\R$\+modules $\ev_{p,\fI}\:\Hom^\R(\fP,\R)/\fV_{F,\fI}
\rarrow(\R/\fI)^F=\R/\fI$.
 Consider the induced map of the contratensor products
$\ev_{p,\fI}\ocn_\R\fQ\:(\Hom^\R(\fP,\R)/\fV_{F,\fI})\ocn_\R\fQ
\rarrow(\R/\fI)\ocn_\R\fQ\simeq\fQ/(\fI\tim\fQ)$.
 Applying the map $\ev_{p,\fI}\ocn_\R\fQ$ to the element~$l_{F,\fI}$,
we obtain the desired element
$$
 \psi(l)(p)+\fI\tim\fQ=(\ev_{p,\fI}\ocn_\R\fQ)(l_{F,\fI})
 \in\fQ/(\fI\tim\fQ).
$$
 The compatibility with respect to the transition maps in
the diagram $(\fQ/(\fI\tim\fQ))_{\fI\subset\R}$ is obvious; so
we have constructed an element $\psi(l)(p)\in\fQ$.

 The key step is to prove that $f=\psi(l)\:\fP\rarrow\fQ$ is
a left $\R$\+contramodule morphism. 
 Given a zero-convergent formal linear combination
$t=\sum_{p\in\fP}r_pp\in\R[[\fP]]$, we need to show that the equation
$$
 f(\pi_\fP(t))=\pi_\fQ\bigl(\R[[f]](t)\bigr)
$$
holds in~$\fQ$.
 It suffices to check that the desired equation holds modulo
$\fI\tim\fQ$ for every open right ideal $\fI\subset\R$.

 Arguing somewhat similarly to the proof of
Lemma~\ref{contramodule-hom-group-complete-and-separated}, we choose
a finite subset $G\subset\fP$ such that $r_p\in\fI$ for all
$p\in\fP\setminus G$, and an open right ideal $\fJ\subset\R$
such that $\fJ\subset\fI$ and $r_p\fJ\subset\fI$ for all $p\in G$.
 Put $p_t=\pi_\fP(t)\in\fP$ and $F=G\cup\{p_t\}\subset\fP$.

 Then the element $l_{F,\fJ}\in(\Hom^\R(\fP,\R)/\fV_{F,\fJ})\ocn_\R\fQ$
can be presented as a sum of decomposable tensors
$$
 l_{F,\fJ}=\sum\nolimits_{i=1}^n(h_i+\fV_{F,\fJ})\ot q_i
$$
where $h_i\in\Hom^\R(\fP,\R)$ and $q_i\in\fQ$.
 Denote by $g\:\fP\rarrow\fQ$ the composition of left
$\R$\+contramodule morphisms
$$
 \fP\xrightarrow{(h_i)_{i=1}^n}\R^n
 \xrightarrow{(q_i)_{i=1}^n}\fQ.
$$
 It follows from the constructions that for all $p\in F$ one has
$f(p)-g(p)\in\fJ\tim\fQ\subset\fQ$.
 The rest of the computation proving that $f$~is an $\R$\+contramodule
morphism proceeds exactly as in the proof of
Lemma~\ref{contramodule-hom-group-complete-and-separated}.

 In order to prove that $\psi_{\fP,\fQ}$ is an isomorphism when $\fP$ is
projective, fix a separated and complete left $\R$\+contramodule $\fQ$,
and let a projective left $\R$\+contramodule $\fP$ vary.
 The full subcategory $\R\Contra_\proj$ is closed under coproducts
in $\R\Contra$, and the free $\R$\+contramodule $\R[[X]]$ is
the coproduct of $X$ copies of the free $\R$\+contramodule~$\R$.
 The functor $\fP\longmapsto\Hom^\R(\fP,\fQ)$ takes coproducts
in $\R\Contra$ (in particular, in $\R\Contra_\proj$) to products
of abelian groups.

 The functor $\fP\longmapsto\Hom^\R(\fP,\R)$ takes coproducts
in $\R\Contra$ to products in $\Modrcs\R$ (which were described in
Section~\ref{topological-modules-secn}).
 The functor ${-}\proocn_\R\fQ$ takes products in $\Modrcs\R$ to
products in $\Ab$ by Lemma~\ref{proocn-preserves-products}.
 So both the functors in the left-hand side and in the right-hand side
of the morphism~$\psi$ take coproducts (in the first argument~$\fP$)
to products of abelian groups.

 These considerations reduce the question to the case of the free
left $\R$\+contramodule with one generator $\fP=\R=\R[[\{{*}\}]]$,
when one has
\[
 \Hom^\R(\R,\fQ)\simeq\fQ\simeq
 \varprojlim\nolimits_{\fI\subset\R}(\fQ/\fI\tim\fQ) \simeq
 \R\proocn_\R\fQ\simeq\Hom^\R(\R,\R)\proocn_\R\fQ.
 \qedhere
\]
\end{proof}

\Section{Pro-Coherent Right Linear Topological Modules}
\label{pro-coherent-modules-secn}

 In this paper, we are interested in \emph{topologically coherent}
topological rings~$\R$.
 This concept goes back to Roos' paper~\cite{Roo}; a relevant recent
reference is~\cite[Section~13]{PS3}.
 Let us briefly spell out the definitions.

 We refer to~\cite[Definition~1.1]{AR} for the definition of
a \emph{finitely presentable object} in a category~$\sC$.
 The definition of a \emph{locally finitely presentable} category
can be found in~\cite[Definition~1.9 and Theorem~1.11]{AR}.
 Every locally finitely presentable abelian category is
Grothendieck by~\cite[Proposition~1.59]{AR}.

 An object of a category $\sC$ is said to be \emph{finitely generated}
if it satisfies the condition of~\cite[Definition~1.67]{AR} with
$\lambda=\aleph_0$, i.~e., if the functor $\Hom_\sC(S,{-})\:\sC
\rarrow\Sets$ preserves the colimits of directed diagrams of
monomorphisms in~$\sC$.
 A discussion of finitely generated and finitely presentable objects
in Grothendieck categories can be found in~\cite[Section~V.3]{Sten}.
 An abelian category $\sA$ with directed colimits is said to be
\emph{locally finitely generated} if it has a set of finitely generated
generators, or equivalently, every object of $\sA$ is the union of its
finitely generated subobjects (cf.~\cite[Theorem~1.70]{AR}).
 Every locally finitely generated abelian category is
Grothendieck~\cite[Corollary~9.6]{PS1}.

 Let $\sA$ be a locally finitely generated abelian category.
 A finitely generated object $S\in\sA$ is called \emph{coherent} if
every finitely generated subobject of $S$ is finitely presentable,
or equivalently, the kernel of any morphism into $S$ from a finitely
generated object is finitely generated.
 Any coherent object in $\sA$ is finitely presentable.
 The category $\sA$ is said to be \emph{locally coherent} if it has
a generating set consisting of coherent objects, or equivalently,
the kernel of any (epi)morphism from a finitely presentable object to
a finitely presentable object in $\sA$ is finitely presentable.

 Any locally coherent abelian category is locally finitely presentable.
 In a locally coherent abelian category $\sA$, the classes of coherent
and finitely presentable objects coincide, and the full subcategory of
finitely presentable objects $\sA_\fp\subset\sA$ is closed under
kernels, cokernels, and extensions in~$\sA$.
 So the category $\sA_\fp$ is abelian if $\sA$ is locally coherent,
and the inclusion functor $\sA_\fp\rarrow\sA$ is exact.
 We refer to~\cite[Section~2]{Roo}, \cite[Section~13]{PS3},
or~\cite[Section~8.2]{PS5} for further details.

 A (complete, separated) right linear topological ring $\R$ is said to
be \emph{topogically right coherent} if the Grothendieck abelian
category $\Discr\R$ is locally coherent~\cite[Section~4]{Roo},
\cite[Section~13]{PS3}.
 If this is the case, we will denote by $\coh\R=(\Discr\R)_\fp\subset
\Discr\R$ the full subcategory of finitely presentable/coherent
objects in $\Discr\R$.
 So the category $\coh\R$ is abelian.
 The following Proposition~\ref{topological-coherence-criterion}
provides a criterion; and then there is a counterexample in
Example~\ref{not-fin-pres-as-abstract-module-counterex}.

\begin{prop} \label{topological-coherence-criterion}
 A (complete, separated) right linear topological ring\/ $\R$ is
topologically right coherent if and only if it admits a base of
neighborhoods of zero $B$ consisting of open right ideals\/
$\fI\subset\R$ and satisfying the following condition.
 For any two open right ideals\/ $\fI$, $\fJ\in B$, any integer $n\ge0$,
and any right\/ $\R$\+module morphism $f\:(\R/\fI)^n\rarrow\R/\fJ$,
the kernel of~$f$ is a finitely generated right\/ $\R$\+module.
 If this is the case, then the discrete right\/ $\R$\+module\/
$\R/\fI$ is coherent for every\/ $\fI\in\R$.
\end{prop}

\begin{proof}
 This is~\cite[Remark~2 in Section~4]{Roo} or~\cite[Lemma~13.1]{PS3}.
\end{proof}

\begin{ex} \label{not-fin-pres-as-abstract-module-counterex}
 For any right linear topological ring $\R$, it is clear that
a discrete right $\R$\+module is finitely generated as an object
of $\Discr\R$ if and only if it is finitely generated as an object
of $\Modr\R$.
 However, the following simple counterexample shows that, for
a topologically right coherent topological ring $\R$, a finitely
presentable/coherent object of $\Discr\R$ \emph{need not} be
finitely presentable in $\Modr\R$ (i.~e., finitely presented
as an abstract $\R$\+module).

 Let $k$~be a field.
 Consider the sequence of polynomial rings $R_n=k[x_1,\dotsc,x_n]$,
\,$n\ge1$, and surjective transition maps $R_{n+1}\rarrow R_n$ that
are $k$\+algebra homomorphisms taking~$x_m$ to~$x_m$ for $1\le m\le n$
and $x_{n+1}$ to~$0$.
 Then the ring $\R=\varprojlim_{n\ge1}R_n$, endowed with the topology
of projective limit of the discrete rings $R_n$, has a base of
neighborhoods of zero consisting of the kernel ideals
$\fI_n=\ker(\R\to R_n)$ of the natural projections $\R\rarrow R_n$.
 The set $B$ of all the ideals $\fI_n\subset\R$, \,$n\ge1$
(as well as the set of all open right ideals in~$\R$) satisfies
the condition of Proposition~\ref{topological-coherence-criterion},
so the topological ring $\R$ is topologically coherent.
 By the same proposition, the discrete $\R$\+modules $R_n$ are
coherent as objects of $\Discr\R$.
 But they are not finitely presented as abstract $\R$\+modules,
since the ideals $\fI_n\subset\R$ are not finitely generated.
 Indeed, for any fixed $n\ge1$, one cannot find an integer~$m$
such that the kernel ideal $\fI_n/\fI_N$ of the transition map
$R_N\rarrow R_n$ would be generated by $m$~elements uniformly
for all integers $N>n$.
\end{ex}

 Given a topologically right coherent topological ring $\R$, we put
$$
 \Cohpro\R=\SPro_\omega(\coh\R).
$$
 So, by Proposition~\ref{strict-pro-objects-closed-under-extensions}(b)
and Lemma~\ref{pro-objects-in-full-subcategory-extension-closed}(b),
$\Cohpro\R$ is a full subcategory closed under extensions in
the exact category $\SPro_\omega(\Discr\R)$, which is a full subcategory
closed under extensions in the abelian category $\Pro_\omega(\Discr\R)$.
 Similarly, $\Cohpro\R$ is a full subcategory closed under extensions
in the abelian category $\Pro_\omega(\coh\R)$.

 The exact category structure on $\SPro_\omega(\Discr\R)$ is inherited
from the abelian exact structure of $\Pro_\omega(\Discr\R)$.
 The exact category $\SPro_\omega(\Discr\R)$ is quasi-abelian by
Proposition~\ref{spro-omega-quasi-abelian-characterization}\,%
(1)\,$\Rightarrow$\,(3).
 The exact category structure on $\Cohpro\R$ is inherited from
the abelian exact structure of $\Pro_\omega(\coh\R)$, and also from
the quasi-abelian exact structure of
$\SPro_\omega(\Discr\R)\simeq\Modrcsom\R$.

\begin{rem}
 Let us \emph{warn} the reader that the exact category $\Cohpro\R=
\SPro_\omega(\coh\R)$ is \emph{not} quasi-abelian for a topologically
right coherent topological ring $\R$ in general.
 In fact, the additive category $\Cohpro\R$ need not even have kernels.
 For a counterexample, consider the case of a discrete coherent
ring $\R=R$.

 Specifically, let $k$~be a field and $R=k\langle x,y\rangle$ be
the free associative $k$\+algebra with two generators $x$ and~$y$
(then the ring $R$ is both left and right coherent).
 The category $\coh\R=\coh R$ is simply the category of coherent right
$R$\+modules in this case.
 In view of
Proposition~\ref{spro-omega-quasi-abelian-characterization}\,%
(5)\,$\Rightarrow$\,(7), in order to show that $\Cohpro R$ does not
have kernels, it suffices to demonstrate an example of a descending
chain of finitely presented $R$\+modules whose intersection does
not exist in $\coh R$.

 Indeed, let us construct an $\omega$\+indexed descending chain of
finitely generated right ideals in $R$ with an infinitely generated
intersection.
 Put $I_0=R$ and
$$
 I_1=yR+xyR\subset R.
$$
 In order to construct the right ideal $I_2\subset I_1$, keep
the summand ``$yR$'' as in~$I_1$, but replace the factor ``$R$''
in ``$xyR$'' by $yR+xyR$, so
$$
 I_2=yR+xyyR+xyxyR.
$$
 To construct the right ideal $I_3\subset I_2$, keep the summands
``$yR$'' and ``$xyyR$'' as in $I_2$, but replace the factor ``$R$''
in ``$xyxyR$'' by $yR+xyR$, so
$$
 I_3=yR+xyyR+xyxyyR+xyxyxyR,
$$
etc.
 So we have
$$
 I_n=\left(\sum\nolimits_{m=0}^{n-1} (xy)^m yR\right)+(xy)^n R,
 \qquad n\ge0.
$$
 Then the right ideal
$$
 J=\bigcap\nolimits_{n\in\omega}I_n=
 \sum\nolimits_{m\in\omega}(xy)^m yR \,\subset\, R
$$
is not finitely generated.
\end{rem}

\begin{cor} \label{Cohpro-R-described-as-topological-modules}
 Let\/ $\R$ be topologically right coherent topological ring.
 Then the additive category\/ $\Cohpro\R$ is equivalent to the full
subcategory in\/ $\Modrcsom\R$ consisting of all the complete,
separated right linear topological right\/ $\R$\+modules $\rN$
that admit a countable base of neighborhoods of zero
$\rN=\rU_0\supset\rU_1\supset\rU_2\supset\dotsb$ such that
the discrete right\/ $\R$\+module $\rN/\rU_n$ is finitely
presentable/coherent for every $n\in\omega$.
\end{cor}

\begin{proof}
 By Lemma~\ref{pro-objects-full-subcategory}, the additive category
$\Cohpro\R=\SPro_\omega(\coh\R)$ is a full subcategory in
$\SPro_\omega(\Discr\R)$.
 It remains to use the equivalence of additive categories
$\SPro_\omega(\Discr\R)\simeq\Modrcsom\R$ provided by
Proposition~\ref{spro-discr-described-as-topological-modules}(b).
\end{proof}

 Let $\R$ be a topologically right coherent topological ring.
 A complete, separated right linear topological right $\R$\+module
is said to be \emph{pro-coherent} if it satisfies the condition of
Corollary~\ref{Cohpro-R-described-as-topological-modules}.

\begin{cor} \label{topological-free-module-pro-coherent}
\textup{(a)} Let\/ $\R$ be a complete, separated topologically
right coherent topological ring.
 Then the object\/ $\R\in\Modrcs\R$, viewed as an object of
the ambient exact category\/ $\SPro(\Discr\R)\supset\Modrcs\R$ as per
Proposition~\ref{spro-discr-described-as-topological-modules}(a),
belongs to the full subcategory\/ $\SPro(\coh\R)\subset\SPro(\Discr\R)$.
\par
\textup{(b)} Let\/ $\R$ be a complete, separated topologically
right coherent topological ring with a countable base of neighborhoods
of zero.
 Then the object\/ $\R\in\Modrcsom\R$ belongs to the full subcategory\/
$\Cohpro\R\subset\Modrcsom\R$ (with\/ $\Cohpro\R$ embedded into\/
$\Modrcsom\R$ as per
Corollary~\ref{Cohpro-R-described-as-topological-modules}).
\end{cor}

\begin{proof}
 Follows from Proposition~\ref{topological-coherence-criterion}.
\end{proof}

 In the rest of this paper, we will be interested in complete,
separated right linear topological rings $\R$ having the two
properties simultaneously:
\begin{itemize}
\item $\R$ has a countable base of neighborhoods of zero;
\item $\R$ is topologically right coherent.
\end{itemize}

\begin{prop} \label{projectives-in-cohpro}
 Let\/ $\R$ be a complete, separated topologically right coherent
topological ring with a countable base of neighborhoods of zero.
 Then the exact category\/ $\Cohpro\R$ has enough projective objects.
 The full subcategory of projective objects in\/ $\Cohpro\R$ coincides
with the full subcategory\/ $\Prod_\omega(\R)\subset\Cohpro\R$
consisting of all the direct summands of countable products of
the object\/ $\R\in\Cohpro\R$.
\end{prop}

\begin{proof}
 We have $\R\in\Cohpro\R$ by
Corollary~\ref{topological-free-module-pro-coherent}(b).
 Countable products exist in the additive category $\Cohpro\R$ by
Corollaries~\ref{products-of-pro-objects}(b)
and~\ref{products-preserve-strict-pro-objects}, and the class of
projective objects in $\Cohpro\R$ is closed under countable products
by Lemma~\ref{projectives-closed-under-countable-products}.
 Moreover, the countable products in $\Cohpro\R$ agree with the ones
in $\SPro_\omega(\Discr\R)$, as explained in the proof of
Lemma~\ref{products-in-linear-topol-modules-and-pro-objects}(a).
 So our notation is unambigous.
 
 To check that $\R$ is a projective object of $\Cohpro\R$, one can
use Lemma~\ref{projectivity-criterion-in-spro-omega}.
 It is easy to compute that the functor $\Hom_\R^\cont(\R,{-})=
\Hom_{\Modrcs\R}(\R,{-})$, restricted to the full subcategory
$\Discr\R\subset\Modrcs\R$, is isomorphic to the forgetful functor
$\Discr\R\rarrow\Ab$.
 This functor is exact on $\Discr\R$; hence it is also exact on
$\coh\R$.
 In fact, the whole functor $\Hom_{\Modrcs\R}(\R,{-})\:\Modrcs\R
\rarrow\Ab$ is isomorphic to the forgetful functor $\Modrcs\R
\rarrow\Ab$; hence its restriction to $\Modrcsom\R$ takes admissible
short exact sequences to short exact sequences.

 Denote by $\sP\subset\Cohpro\R$ the set of all finite direct sums
of copies of the object~$\R$.
 Let $\N\in\coh\R$ be a coherent discrete right $\R$\+module.
 Since $\coh\R$ is a full subcategory in $\Cohpro\R$, we can (and will)
view $\N$ as an object of $\Cohpro\R$.
 This entails viewing a discrete right $\R$\+module as a complete,
separated topological right $\R$\+module with a right linear topology,
or viewing an object of $\coh\R$ as an object of
$\SPro_\omega(\coh\R)=\Cohpro\R$.
 According to Proposition~\ref{projectives-as-summands-of-products},
in order to prove the desired assertion we only need to show that
there exists an object $\rP\in\sP$ together with an admissible
epimorphism $\rP\rarrow\N$ in $\Cohpro\R$.

 For any $\omega^\sop$\+indexed diagram of epimorphisms
$(P_n)_{n\in\omega}$ in an abelian category $\sA$, the projection
morphism $\plim_{n\in\omega}P_n\rarrow P_m$ is an admissible
epimorphism in $\SPro_\omega(\sA)$ for every $m\in\omega$.
 All epimorphisms in $\sA\subset\SPro_\omega(\sA)$ are also
admissible epimorphisms in $\SPro_\omega(\sA)$.
 It remains to choose an open right ideal $\fI\subset\R$ such
that the discrete right $\R$\+module $\R/\fI$ is coherent and
the discrete right $\R$\+module $\N$ is a quotient module of
$(\R/\fI)^k$ for some integer $k\ge0$.
 This is possible by Proposition~\ref{topological-coherence-criterion}.
 Then we have $\R^k\in\sP$, and both the morphisms
$\R^k\rarrow(\R/\fI)^k$ and $(\R/\fI)^k\rarrow\N$ are admissible
epimorphisms in $\Cohpro\R$.
\end{proof}

\Section{Construction of the Functor~$\Xi$}
\label{construction-of-Xi-secn}

 Given an additive category $\sE$, we denote by $\Hot(\sE)$
the cochain homotopy category of (unbounded) complexes in~$\sE$.
 The notation $\Hot^\st(\sE)\subset\Hot(\sE)$, where $\st=+$, $-$,
or~$\bb$, stands for the full subcategories of bounded below, bounded
above, or bounded complexes in $\Hot(\sE)$, as usual.
 So $\Hot(\sE)$ is a triangulated category and $\Hot^\st(\sE)$ are
its full triangulated subcategories.


 Let $\sE$ be an abelian or exact category.
 Then the notation $\sD^\st(\sE)$, where $\st=\varnothing$, $+$, $-$,
or~$\bb$, is used for the unbounded or bounded derived categories
of~$\sE$ (as in Corollary~\ref{fully-faithful-on-D-b-and-D-plus}).
 Let us also recall the notation $\sE_\proj\subset\sE$ for the full
subcategory of projective objects in~$\sE$.

 Let $\sB$ be an abelian category with enough projective objects.
 A complex $B^\bu$ is $\sB$ is said to be \emph{contraacyclic}
(\emph{in the sense of Becker}~\cite[Propositions~1.3.6(1)
and~1.3.8(1)]{Bec}) if, for every complex of projective objects
$P^\bu$ in $\sB$, every morphism of complexes $P^\bu\rarrow B^\bu$
is homotopic to zero.
 The thick subcategory of contraacyclic complexes is denoted by
$\Ac^\bctr(\sB)\subset\Hot(\sB)$.
 The related Verdier quotient category
$$
 \sD^\bctr(\sB)=\Hot(\sB)/\Ac^\bctr(\sB)
$$
is called the (\emph{Becker}) \emph{contraderived category} of~$\sB$.

\begin{thm} \label{becker-contraderived-theorem}
 Let\/ $\sB$ be a locally presentable abelian category with enough
projective objects.
 Then the composition of the fully faithful inclusion functor\/
$\Hot(\sB_\proj)\rarrow\Hot(\sB)$ with the triangulated Verdier
quotient functor\/ $\Hot(\sB)\rarrow\sD^\bctr(\sB)$ is a triangulated
equivalence
$$
 \Hot(\sB_\proj)\simeq\sD^\bctr(\sB).
$$
\end{thm}

\begin{proof}
 This is~\cite[Corollary~7.4]{PS4}.
 For a generalization, see~\cite[Corollary~6.14]{PS5}.
\end{proof}

 Let $\R$ be a complete, separated right linear topological ring with
a countable base of neighborhoods of zero.
 Assume that $\R$ is topologically right coherent.
 The aim of this section is to construct a fully faithful contravariant
triangulated functor
$$
 \Xi\:\sD^\bb(\coh\R)^\sop\lrarrow\sD^\bctr(\R\Contra)
$$
and study its basic properties.

 We begin with the triangulated functor
\begin{equation} \label{d-b-coh-into-d-b-cohpro}
 \sD^\bb(\coh\R)\lrarrow\sD^\bb(\Cohpro\R)
\end{equation}
induced by the inclusion of abelian/exact categories
$\coh\R\rarrow\Cohpro\R$.
 By Corollary~\ref{fully-faithful-on-D-b-and-D-plus}(b),
the triangulated functor~\eqref{d-b-coh-into-d-b-cohpro}
is fully faithful.
 Furthermore, we have the obvious fully faithful triangulated functor
$$
 \sD^\bb(\Cohpro\R)\lrarrow\sD^-(\Cohpro\R).
$$

 Proposition~\ref{projectives-in-cohpro} implies a triangulated
equivalence
$$
 \Hot^-(\Prod_\omega(\R))\simeq\sD^-(\Cohpro\R).
$$
 The anti-equivalence of additive categories from
Corollary~\ref{countable-pontryagin-duality} induces a triangulated
anti-equivalence
$$
 \Hot^-(\Prod_\omega(\R))^\sop\simeq\Hot^+(\R\Contra_\proj^\omega).
$$
 The fully faithful inclusion of additive categories
$\R\Contra_\proj^\omega\rarrow\R\Contra_\proj$ induces a fully
faithful triangulated functor
$$
 \Hot^+(\R\Contra_\proj^\omega)\lrarrow\Hot(\R\Contra_\proj).
$$
 Finally, we have the triangulated equivalence of
Theorem~\ref{becker-contraderived-theorem},
$$
 \Hot(\R\Contra_\proj)\simeq\sD^\bctr(\R\Contra).
$$

 Composing all the fully faithful triangulated functors above,
\begin{multline*}
 \sD^\bb(\coh\R)^\sop\lrarrow\sD^\bb(\Cohpro\R)^\sop
 \lrarrow\sD^-(\Cohpro\R)^\sop \\ \simeq\Hot^-(\Prod_\omega(\R))^\sop
 \simeq\Hot^+(\R\Contra_\proj^\omega) \\ \lrarrow
 \Hot(\R\Contra_\proj)\simeq\sD^\bctr(\R\Contra),
\end{multline*}
we obtain a fully faithful contravariant triangulated functor
$$
 \Xi\:\sD^\bb(\coh\R)^\sop\lrarrow\sD^\bctr(\R\Contra).
$$

 Let $\R\Contra_\flat\subset\R\Contra$ denote the full subcategory
of flat left $\R$\+con\-tra\-mod\-ules.
 The following theorem is the main result of this section.

\begin{thm} \label{Hom-from-Xi-computed}
 Let\/ $\R$ be a complete, separated, right linear, topologically right
coherent topological ring with a countable base of neighborhoods
of zero.
 Then, for any bounded complex of coherent discrete right\/
$\R$\+modules\/ $\N^\bu\in\sD^\bb(\coh\R)$ and any (unbounded) complex
of flat left\/ $\R$\+contramodules\/ $\fF^\bu\in\Hot(\R\Contra_\flat)$,
there is a natural isomorphism of abelian groups
$$
 \Hom_{\sD^\bctr(\R\Contra)}(\Xi(\N^\bu),\fF^\bu)
 \simeq H^0(\N^\bu\ocn_\R\fF^\bu).
$$
\end{thm}

 Before proving the theorem, we need to state a couple of lemmas.

\begin{lem} \label{flat-contramodules-complete-separated}
 Let\/ $\R$ be a complete, separated right linear topological ring
with a countable base of neighborhoods of zero.
 Then \par
\textup{(a)} all left\/ $\R$\+contramodules are complete (but not
necessarily separated); \par
\textup{(b)} all flat left\/ $\R$\+contramodules are separated.
\end{lem}

\begin{proof}
 Part~(a) is~\cite[Lemma~6.3(b)]{PR}.
 Part~(b) is~\cite[Corollary~6.15]{PR}.
 For a discussion of counterexamples of nonseparated contramodules,
see~\cite[Section~1.5]{Prev}.
\end{proof}

 For any bicomplex of abelian groups $C^{\bu,\bu}$, let us denote
by $\Tot^\sqcap(C^{\bu,\bu})$ the direct product totalization of
the complex $C^{\bu,\bu}$.
 So the components of the complex $\Tot^\sqcap(C^{\bu,\bu})$ are
$\Tot^\sqcap(C^{\bu,\bu})^m=\prod_{p+q=m}C^{p,q}$.

\begin{lem} \label{product-totalization-acyclic}
 Let $C^{\bu,\bu}$ be a bicomplex of abelian groups such that
$C^{n,\bu}=0$ for $n>0$ and, for every $i\in\boZ$, the complex of
abelian groups $C^{\bu,i}$,
$$
 \dotsb\lrarrow C^{-2,i}\lrarrow C^{-1,i}\lrarrow C^{0,i}\lrarrow0,
$$
is acyclic.
 Then the complex of abelian groups\/ $\Tot^\sqcap(C^{\bu,\bu})$
is acyclic.
\end{lem}

\begin{proof}
 This is~\cite[Lemma~6.2(b)]{PS7}.
\end{proof}

\begin{proof}[Proof of Theorem~\ref{Hom-from-Xi-computed}]
 Let $\rP^\bu$ be a bounded above complex of projective objects in
$\Cohpro\R$ endowed with a quasi-isomorphism $\rP^\bu\rarrow\N^\bu$
of complexes in the exact category $\Cohpro\R$.
 So the terms of the complex $\rP^\bu$ belong to $\Prod_\omega(\R)$.
 Let $\fP^\bu=\Hom_\R^\cont(\rP^\bu,\R)$ be the corresponding
complex in the additive category $\R\Contra_\proj^\omega\simeq
(\Prod_\omega\R)^\sop$.
 Then we have $\Xi(\N^\bu)=\fP^\bu\in\sD^\bctr(\R\Contra)$.

 Notice that for any complexes $\fP^\bu\in\Hot(\R\Contra_\proj)$
and $\fQ^\bu\in\Hot(\R\Contra)$ the Verdier quotient functor
$\Hot(\R\Contra)\rarrow\sD^\bctr(\R\Contra)$ induces an isomorphism
of abelian groups
$$
 \Hom_{\Hot(\R\Contra)}(\fP^\bu,\fQ^\bu)\simeq
 \Hom_{\sD^\bctr(\R\Contra)}(\fP^\bu,\fQ^\bu).
$$
 In other words, whenever $\fP^\bu$ is a complex of projective
$\R$\+contramodules, the $\Hom$ group in the contraderived category
can be computed as the degree~$0$ cohomology group of the complex
of $\Hom$ between the two complexes of $\R$\+contramodules
$$
 \Hom_{\sD^\bctr(\R\Contra)}(\fP^\bu,\fQ^\bu)\simeq
 H^0\Hom^\R(\fP^\bu,\fQ^\bu).
$$

 For any complex of complete, separated right linear topological
right $\R$\+modules $\rN^\bu$ and any complex of left
$\R$\+contramodules $\fQ^\bu$, let us denote by
$$
 \rN^\bu\proocn_\R^\sqcap\fQ^\bu=
 \Tot^\sqcap(\rN^n\proocn_\R\fQ^i)_{i,n\in\boZ}
$$
the direct product totalization of the bicomplex of pro-contratensor
products $\rN^n\proocn_\R\fQ^i$.
 Assume that the terms of the complex $\fQ^\bu$ are complete, separated
left $\R$\+contramodules.
 According to the proof of Theorem~\ref{pontryagin-duality-theorem},
in the situation at hand we have $\rP^\bu\simeq\Hom^\R(\fP^\bu,\R)$.
 Therefore, Proposition~\ref{contramodule-hom-and-pro-contratensor}
implies a natural isomorphism of complexes of abelian groups
\begin{equation} \label{Hom-R-complex-as-totalized-pro-contratensor}
 \Hom^\R(\fP^\bu,\fQ^\bu)\simeq\rP^\bu\proocn_\R^\sqcap\fQ^\bu.
\end{equation}
 In particular, in view of
Lemma~\ref{flat-contramodules-complete-separated},
the formula~\eqref{Hom-R-complex-as-totalized-pro-contratensor} is
applicable to our complex of flat contramodules $\fQ^\bu=\fF^\bu$.

 Denote by $\rM^\bu$ the cone of the quasi-isomorphism of complexes
$\rP^\bu\rarrow\N^\bu$.
 So $\rM^\bu$ is a bounded above, acyclic complex in the exact
category $\Cohpro\R$.
 Then $\rM^\bu$ is also an acyclic complex in the exact category
$\Modrcsom\R$ (as per the discussion in
Sections~\ref{topological-modules-secn}
and~\ref{pro-coherent-modules-secn}).
 By Lemma~\ref{proocn-with-flat-exact-for-countable}, for every
degree $i\in\boZ$, the complex of abelian groups
$\rM^\bu\proocn_\R\fF^i$ is acyclic.
 According to Lemma~\ref{product-totalization-acyclic}, it follows
that the complex of abelian groups $\rM^\bu\proocn_\R^\sqcap\fF^\bu$
is acyclic as well.

 Thus the quasi-isomorphism $\rP^\bu\rarrow\N^\bu$ in $\Cohpro\R$
induces a quasi-isomorphism of complexes of abelian groups
$$
 \rP^\bu\proocn_\R^\sqcap\fF^\bu\overset\simeq\lrarrow
 \N^\bu\proocn_\R^\sqcap\fF^\bu\,=\,\N^\bu\ocn_\R\fF^\bu,
$$
and we are done.
\end{proof}

 Now we recall the basics of the theory of compactly generated
triangulated categories~\cite{Neem0}.
 Let $\sT$ be a triangulated category with coproducts.
 An object $S\in\sT$ is said to be \emph{compact} if the functor
$\Hom_\sT(S,{-})\:\sT\rarrow\Ab$ preserves coproducts.

 A set of compact objects $\sS\subset\sT$ is said to \emph{generate}
the triangulated category $\sT$ if $\Hom_\sT(S,X[n])=0$ for a given
object $X\in\sT$, all $S\in\sS$, and all $n\in\boZ$ implies $X=0$.
 Equivalently, a set of compact objects $\sS$ generates $\sT$ if and
only if the minimal full triangulated subcategory of $\sT$ containing
$\sS$ and closed under coproducts coincides with the whole of~$\sT$
\,\cite[Theorem~2.1(2) or~4.1]{Neem0}.
 If this is the case, then the full subcategory of compact objects
in $\sT$ coincides with the minimal thick subcategory of $\sT$
containing~$\sS$ \,\cite[Theorem~2.1(3)]{Neem0}.

\begin{cor} \label{image-of-Xi-consists-of-compacts}
 All the objects in the image of the fully faithful triangulated
functor
$$
 \Xi\:\sD^\bb(\coh\R)^\sop\lrarrow\sD^\bctr(\R\Contra)
$$
are compact in the triangulated category\/ $\sD^\bctr(\R\Contra)$.
\end{cor}

\begin{proof}
 First of all, for the assertion of the corollary to make sense,
we need to show that all coproducts exist in the triangulated
category $\sD^\bctr(\R\Contra)$.
 This follows from Theorem~\ref{becker-contraderived-theorem}.
 Indeed, $\R\Contra_\proj$ is an additive category with coproducts,
hence $\Hot(\R\Contra_\proj)$ is a triangulated category with
coproducts.
 So the coproducts in $\sD^\bctr(\R\Contra)$ can be computed as
the termwise coproducts of complexes of projective contramodules
(cf.~\cite[the paragraph before Corollary~7.7]{PS4}).

 Now the desired assertion follows from
Theorem~\ref{Hom-from-Xi-computed}.
 It suffices to point out that, for any bounded complex $\N^\bu$
in $\coh\R$ (or in $\Discr\R$), the functor $\fP^\bu\longmapsto
\N^\bu\ocn_\R\fP^\bu$ takes termwise coproducts of complexes
of $\R$\+contramodules $\fP^\bu\in\Hot(\R\Contra)$ to termwise
coproducts of complexes of abelian groups.
 This follows from the fact that, for any $\N\in\Discr\R$,
the functor $\N\ocn_\R{-}\,\:\R\Contra\rarrow\Ab$ preserves coproducts
(and in fact, all colimits), as explained in
Section~\ref{contramodules-preliminaries-secn}.
\end{proof}

\Section{Flat and Projective Periodicity for Contramodules}

 In this section we discuss a contramodule generalization of the flat
and projective periodicity theorem of
Benson--Goodearl~\cite[Theorem~2.5]{BG} and its stronger version
due to Neeman~\cite[Theorem~8.6\,(i)\,$\Leftrightarrow$\,(3)
and Remark~2.15]{Neem}.
 This discussion was started in the preceding papers by the first-named
author~\cite[Proposition~12.1]{Pflcc}, \cite[Theorems~5.1
and~6.1]{Pbc}; we continue it in the present paper.

 The results of this section concerning the flat and projective
periodicity are partly conditional upon compact generation of
the contraderived category by the image of the functor $\Xi$ constructed
in the previous Section~\ref{construction-of-Xi-secn}.
 The compact generation will be established in the final
Section~\ref{compact-generation-secn}; this will make the results
of this section unconditional.
 In the meantime, the results of~\cite{Pflcc,Pbc} on flat and
projective periodicity can be used to prove the compact generation
easily under more restrictive assumptions.

 We start with a sequence of lemmas going back to
the paper~\cite{PR}.
 A more general discussion can be found in
the papers~\cite{Pproperf,PPT}.

\begin{lem} \label{flat-contramodules-closure-properties}
 Let\/ $\R$ be a complete, separated right linear topological ring
with a countable base of neighborhoods of zero.
 Then the full subcategory of flat left\/ $\R$\+contramodules\/
$\R\Contra_\flat$ is closed under extensions and kernels of
epimorphisms in\/ $\R\Contra$.
\end{lem}

\begin{proof}
 This is~\cite[Corollaries~6.8 and~6.13]{PR} together
with~\cite[Corollary~6.15]{PR}.
 Some more general assertions applicable to any complete, separated
right linear topological ring $\R$ can be found
in~\cite[Lemma~13.1(i,iii)]{PPT}.
\end{proof}

\begin{lem} \label{balancing-criterion-of-flatness}
 Let\/ $\R$ be a complete, separated right linear topological ring
with a countable base of neighborhoods of zero.
 Let\/ $0\rarrow\fH\rarrow\fG\rarrow\fF\rarrow0$ be a short exact
sequence of left\/ $\R$\+contramodules with a flat left\/
$\R$\+contramodule\/~$\fG$.
 In this setting: \par
\textup{(a)} The left\/ $\R$\+contramodule\/ $\fF$ is flat if and only
if, for every discrete right\/ $\R$\+module\/ $\N$, the induced map
of abelian groups\/ $\N\ocn_\R\fH\rarrow\N\ocn_\R\fG$ is injective.
 If this is the case, then the\/ $\R$\+contramodule\/ $\fH$ is
flat as well. \par
\textup{(b)} When the topological ring\/ $\R$ is topologically right
coherent, it suffices to consider coherent discrete right\/
$\R$\+modules\/ $\N$ in the criterion of part~\textup{(a)}.
\end{lem}

\begin{proof}
 Part~(a): the ``only if'' assertion is~\cite[Lemma~6.7 or~6.10]{PR}
together with~\cite[Corollary~6.15]{PR}.
 The ``if'' assertion is a general result applicable to any complete,
separated right linear topological  ring~$\R$; see~\cite[two paragraps
preceding Lemma~3.1]{Pproperf} or~\cite[Lemma~13.5]{PPT}.
 The ``if this is the case'' clause is a part of
Lemma~\ref{flat-contramodules-closure-properties};
see also~\cite[Lemma~13.5(ii)]{PPT} for a generalization.

 Part~(b) follows from part~(a), because any object of $\Discr\R$ is
a directed colimit of objects from $\coh\R$ \,\cite[Definition~1.9
and Theorem~1.11]{AR}, while the functor ${-}\ocn_\R\nobreak\fP\:
\allowbreak\Discr\R\rarrow\Ab$ preserves colimits for any left
$\R$\+contramodule~$\fP$ (as explained in
Section~\ref{contramodules-preliminaries-secn}) and the directed
colimit functors in $\Ab$ are exact.
\end{proof}

\begin{lem} \label{contramodule-nakayama}
 Let\/ $\R$ be a complete, separated right linear topological ring
with a countable base of neighborhoods of zero and\/ $\fP\ne0$ be
a left\/ $\R$\+contramodule.
 Then there exists a discrete right\/ $\R$\+module\/ $\N$ such that\/
$\N\ocn_\R\fP\ne0$.
 If the topological ring\/ $\R$ is topologically right coherent, then
one can choose\/ $\N$ to be a coherent discrete right\/ $\R$\+module.
\end{lem}

\begin{proof}
 The first assertion is one of the versions of contramodule Nakayama
lemma, see~\cite[Lemma~6.14]{PR}.
 A discussion can be found in~\cite[Lemmas~2.1 and~3.22]{Prev}.
 The second assertion follows similarly to the proof of
Lemma~\ref{balancing-criterion-of-flatness}(b).
\end{proof}

\begin{prop} \label{contratensor-criterion-of-pure-acyclicity}
 Let\/ $\R$ be a complete, separated right linear topological ring
with a countable base of neighborhoods of zero, and let\/ $\fF^\bu$
be a complex of flat left\/ $\R$\+contramodules.
 In this setting: \par
\textup{(a)} The complex\/ $\fF^\bu$ is acyclic with flat contramodules
of cocycles if and only if, for every discrete right\/ $\R$\+module\/
$\N$, the complex of abelian groups\/ $\N\ocn_\R\fF^\bu$ is acyclic.
\par
\textup{(b)} When the topological ring\/ $\R$ is topologically right
coherent, it suffices to consider coherent discrete right\/
$\R$\+modules\/ $\N$ in the criterion of part~\textup{(a)}.
\end{prop}

\begin{proof}
 The proofs of parts~(a) and~(b) are similar and based on the respective
parts of Lemmas~\ref{balancing-criterion-of-flatness}
and~\ref{contramodule-nakayama}.
 In order to prove the ``only if'' assertions, recall that the functor
$\N\ocn_\R{-}\,\:\R\Contra\rarrow\Ab$ is always right exact (as per
Section~\ref{contramodules-preliminaries-secn}).
 In view of Lemma~\ref{balancing-criterion-of-flatness}, it follows
that this functor takes short exact sequences of flat
$\R$\+contramodules to short exact sequences of abelian groups.

 To prove the ``if'', denote the contramodules of coboundaries and
cocycles of the complex $\fF^\bu$ by $\fB^i\subset\fZ^i\subset\fF^i$.
 Then there are short exact sequences of $\R$\+contramodules $0\rarrow
\fZ^i\rarrow\fF^i\rarrow\fB^{i+1}\rarrow0$ and $0\rarrow\fB^i\rarrow
\fF^i\rarrow\fF^i/\fB^i\rarrow0$ for all $i\in\boZ$.

 Given a (coherent) discrete right $\R$\+module $\N$, any element~$c$
in the kernel of the map $\N\ocn_\R\fB^{i+1}\rarrow\N\ocn_\R\fF^{i+1}$
can be lifted to an element~$c'$ in the kernel of the map
$\N\ocn_\R\fF^i\rarrow\N\ocn_\R\fF^{i+1}$ (since the map
$\N\ocn_\R\fF^i\rarrow\N\ocn_\R\fB^{i+1}$ is surjective).
 If $c'$~belongs to the image of the map $\N\ocn_\R\fF^{i-1}\rarrow
\N\ocn_\R\fF^i$, then $c=0$ (since the composition $\fF^{i-1}\rarrow
\fF^i\rarrow\fB^{i+1}$ vanishes).
 Therefore, acyclicity of the complex $\N\ocn_\R\fF^\bu$ implies
injectivity of the map $\N\ocn_\R\fB^{i+1}\rarrow\N\ocn_\R\fF^{i+1}$.

 Applying Lemma~\ref{balancing-criterion-of-flatness} to the short
exact sequence $0\rarrow\fB^{i+1}\rarrow\fF^{i+1}\rarrow
\fF^{i+1}/\fB^{i+1}\rarrow0$, we can conclude that
the $\R$\+contramodules $\fF^{i+1}/\fB^{i+1}$ and $\fB^{i+1}$ are flat.
 Applying Lemma~\ref{flat-contramodules-closure-properties} to the short
exact sequence $0\rarrow\fZ^i\rarrow\fF^i\rarrow\fB^{i+1}\rarrow0$,
we see that the $\R$\+contramodule $\fZ^i$ is flat.

 Finally, denote by $\fH^i=\fZ^i/\fB^i$ the contramodules of
cohomology of the complex $\fF^\bu$.
 Applying Lemma~\ref{flat-contramodules-closure-properties} to the short
exact sequence $0\rarrow\fH^i\rarrow\fF^i/\fB^i\rarrow\fB^{i+1}
\rarrow0$, we see that the $\R$\+contramodule $\fH^i$ is flat.

 Now the fact that the functor $\N\ocn_\R{-}$ takes short exact
sequences of flat contramodules to short exact sequences of abelian
groups (as mentioned above) implies a natural isomorphism of
abelian groups $\N\ocn_\R\fH^i\simeq H^i(\N\ocn_\R\fF^\bu)$
for every $i\in\boZ$.
 Using acyclicity of the complex $\N\ocn_\R\fF^\bu$ again, we
conclude that $\N\ocn_\R\fH^i=0$.
 It remains to refer to Lemma~\ref{contramodule-nakayama} in order to
deduce the desired vanishing assertion $\fH^i=0$.
\end{proof}

\begin{cor} \label{Hom-from-Xi-to-complex-of-flats-vanishing-criterion}
 Let\/ $\R$ be a complete, separated topologically right coherent
topological ring with a countable base of neighborhoods of zero, and
let\/ $\fF^\bu$ be a complex of flat left\/ $\R$\+contramodules.
 Then the following two conditions are equivalent:
\begin{enumerate}
\item For any bounded complex of coherent discrete right\/
$\R$\+modules $\N^\bu\in\sD^\bb(\coh\R)$, one has\/
$\Hom_{\sD^\bctr(\R\Contra)}(\Xi(\N^\bu),\fF^\bu)=0$. \hbadness=1200
\item The complex of\/ $\R$\+contramodules\/ $\fF^\bu$ is acyclic with
flat\/ $\R$\+contramodules of cocycles.
\end{enumerate}
\end{cor}

\begin{proof}
 By Theorem~\ref{Hom-from-Xi-computed}, condition~(1) means that
the complex of abelian groups $\N^\bu\ocn_\R\fF^\bu$ is acyclic for any
bounded complex of coherent discrete right $\R$\+modules~$\N^\bu$.
 Clearly, it suffices to check this condition for one-term complexes
$\N^\bu=\N\in\coh\R$.
 Hence the corollary follows from
Proposition~\ref{contratensor-criterion-of-pure-acyclicity}(a\+-b).
\end{proof}

 Now we arrive to the following theorem, which is a contramodule
version of~\cite[Theorem~8.6\,(i)\,$\Rightarrow$\,(iii)]{Neem}
(cf.~\cite[Theorem~6.1]{Pflcc}).

\begin{thm} \label{flat-contraacyclic-are-pure-acyclic}
 Let\/ $\R$ be a complete, separated, right linear, topologically right
coherent topological ring with a countable base of neighborhoods of
zero.
 Then any contraacyclic complex of flat left\/ $\R$\+contramodules is
acyclic with flat\/ $\R$\+contramodules of cocycles.
\end{thm}

\begin{proof}
 Contraacyclic complexes of left $\R$\+contramodules vanish in
$\sD^\bctr(\R\Contra)$, so the assertion follows immediately from
Corollary~\ref{Hom-from-Xi-to-complex-of-flats-vanishing-criterion}\,%
(1)\,$\Rightarrow$\,(2).
\end{proof}

 Property~(3) in the following proposition is a contramodule version
of the flat and projective periodicity theorem of
Benson--Goodearl~\cite[Theorem~2.5]{BG} and
Neeman~\cite[Remark~2.15]{Neem} (see also~\cite[Proposition~7.6]{CH}).
 Property~(4) in the same proposition is a contramodule version
of~\cite[Theorem~8.6\,(iii)\,$\Rightarrow$\,(i)]{Neem}
(cf.~\cite[Theorem~5.1]{Pflcc}).

 We will see below in Section~\ref{compact-generation-secn}
that the equivalent conditions of
Proposition~\ref{four-equivalent-conditions-prop} are always satisfied.
 We formulate it as an equivalence of four conditions here in order
to emphasize the logical connections between the properties involved.

\begin{prop} \label{four-equivalent-conditions-prop}
 Let\/ $\R$ be a complete, separated topologically right coherent
topological ring with a countable base of neighborhoods of zero.
 Then the following four conditions are equivalent:
\begin{enumerate}
\item the image of the functor\/ $\Xi\:\sD^\bb(\coh\R)^\sop\rarrow
\sD^\bctr(\R\Contra)$ is a set of compact generators for
the triangulated category\/ $\sD^\bctr(\R\Contra)$;
\item the contraderived category\/ $\sD^\bctr(\R\Contra)$ is compactly
generated, and the functor\/ $\Xi$ provides an anti-equivalence between
the full subcategory of compact objects in\/ $\sD^\bctr(\R\Contra)$ and
the bounded derived category\/ $\sD^\bb(\coh\R)$;
\item for any acyclic complex of projective left\/
$\R$\+contramodules\/ $\fP^\bu$ with flat\/ $\R$\+contramodules
of cocycles, the contramodules of cocycles are actually projective
(so the complex\/ $\fP^\bu$ is contractible);
\item for any complex of projective left\/ $\R$\+contramodules\/
$\fP^\bu$, and any acyclic complex of flat left\/ $\R$\+contramodules\/
$\fF^\bu$ with flat $\R$\+contramodules of cocycles, any morphism of
complexes\/ $\fP^\bu\rarrow\fF^\bu$ is homotopic to zero.
\end{enumerate}
\end{prop}

\begin{proof}
 The implication (2)~$\Longrightarrow$~(1) is obvious.
 
 (1)~$\Longrightarrow$~(2) Since we already know from the construction
that $\Xi$ is a fully faithful triangulated functor, it suffices to
point out that the bounded derived category $\sD^\bb(\coh\R)$ is
idempotent-complete.
 In fact, the triangulated category $\sD^\bb(\sA)$ is idempotent
complete for any abelian category $\sA$, and even for any
idempotent-complete exact category~$\sA$ \,~\cite[Theorem~2.8]{BS}.

 (1)~$\Longleftrightarrow$~(3)
 Let $\fP^\bu\in\Hot(\R\Contra_\proj)$ be a complex of projective
left $\R$\+con\-tra\-mod\-ules.
 Then, in view of Theorem~\ref{becker-contraderived-theorem},
condition~(1) means that $\fP^\bu$ is contractible whenever
$\Hom_{\sD^\bctr(\R\Contra)}(\Xi(\N^\bu),\fP^\bu)=0$ for all
bounded complexes of coherent discrete right $\R$\+modules~$\N^\bu$.
 By Corollary~\ref{Hom-from-Xi-to-complex-of-flats-vanishing-criterion},
the latter condition is equivalent to the complex $\fP^\bu$ being
acyclic with flat $\R$\+contramodules of cocycles.
 
 (4)~$\Longrightarrow$~(3)
 Take $\fF^\bu=\fP^\bu$, and consider the identity morphism
$\fP^\bu\rarrow\fF^\bu$.

 (1)~$\Longrightarrow$~(4)
 In view of Theorem~\ref{becker-contraderived-theorem} and the result
of~\cite[Theorem~2.1(2) or~4.1]{Neem} cited above, condition~(1) implies
that all complexes of projective left $\R$\+contramodules are homotopy
equivalent to complexes obtained from the complexes belonging to
the image of $\Xi$ using the operations of cone and infinite coproduct.

 Therefore, it suffices to prove assertion~(4) for complexes
of projective contramodules $\fP^\bu$ of the form $\fP^\bu=\Xi(\N^\bu)$,
where $\N^\bu\in \sD^\bb(\coh\R)$.
 In this case, it remains to refer to
Corollary~\ref{Hom-from-Xi-to-complex-of-flats-vanishing-criterion}\,%
(2)\,$\Rightarrow$\,(1) in order to show that the abelian group
$H^0\Hom^\R(\fP^\bu,\fF^\bu)\simeq
\Hom_{\sD^\bctr(\R\Contra)}(\fP^\bu,\fF^\bu)$ vanishes.
\end{proof}

 We can already deduce the following corollary describing
the compact generators of the contraderived category\/
$\sD^\bctr(\R\Contra)$ under the additional assumption of a topology
base of two-sided ideals in\/~$\R$.
 A more general result, assuming only a topology base of open right
ideals, will be obtained in Section~\ref{compact-generation-secn}.

\begin{cor}
 Let\/ $\R$ be a complete, separated, topologically right coherent
topological ring with a countable base of neighborhoods of zero
consisting of open \emph{two-sided} ideals.
 Then the triangulated category\/ $\sD^\bctr(\R\Contra)$ is compactly
generated, and the functor\/ $\Xi$ provides an anti-equivalence between
the full subcategory of compact objects in\/ $\sD^\bctr(\R\Contra)$ and
the bounded derived category\/ $\sD^\bb(\coh\R)$.
\end{cor}

\begin{proof}
 By~\cite[Proposition~12.1(b)]{Pflcc}, condition~(3) of
Proposition~\ref{four-equivalent-conditions-prop} holds for any
complete, separated topological ring $\R$ with a countable base of
neighborhoods of zero consisting of open two-sided ideals.
 Moreover, by~\cite[Theorem~5.1]{Pbc}, condition~(4) of
Proposition~\ref{four-equivalent-conditions-prop} holds for
such topological rings $\R$ as well.
 Applying the proposition, we conclude that condition~(2) holds,
as desired.
\end{proof}

\Section{Bounded Below Complexes of Countably Generated Projectives}
\label{countably-generated-projectives-secn}

 Let $\R$ be a complete, separated, right linear, topologically right
coherent topological ring with a countable base of neighborhoods of
zero.
 The aim of this section and the next one is to show that
the homotopy category of complexes of projective contramodules
$\Hot(\R\Contra_\proj)$ coincides with its minimal full triangulated
subcategory containing the image of the functor
$$
 \Xi\:\sD^\bb(\coh\R)^\sop\rarrow\Hot(\R\Contra_\proj)
$$
and closed under infinite coproducts.

 By construction, the image of the functor $\Xi$ is contained in
the full subcategory $\Hot^+(\R\Contra_\proj^\omega)\subset
\Hot(\R\Contra_\proj)$.
 Let us denote the essential image of $\Xi$ by
$$
 \sC\subset\Hot^+(\R\Contra_\proj^\omega).
$$
 In this section we will show that all objects of
$\Hot^+(\R\Contra_\proj^\omega)$ (and in fact, more generally all
objects of $\Hot(\R\Contra_\proj^\omega)$) can be obtained from objects
of $\sC$ using cones and countable coproducts.

 Let us recall the definitions of sequential homotopy colimits and
homotopy limits in triangulated categories.
 Let $\sT$ be a triangulated category with countable coproducts,
and let $X_0\overset{\phi_0}\rarrow X_1\overset{\phi_1}\rarrow X_2
\rarrow\dotsb$ be an $\omega$\+indexed diagram in~$\sT$.
 Then the \emph{homotopy colimit} $\hocolim_{n\in\omega}X_n\in\sT$
\,\cite{BN}, \cite[Sections~1.6\+-1.7]{Neem-book} is defined as
the object appearing in the distinguished triangle
$$
 \coprod\nolimits_{n\in\omega}X_n
 \xrightarrow{\id-shift}
 \coprod\nolimits_{n\in\omega}X_n\lrarrow
 \hocolim\nolimits_{n\in\omega}X_n\lrarrow
 \coprod\nolimits_{n\in\omega}X_n[1]
$$
 Here $shift\:\coprod_{n\in\omega}X_n\rarrow\coprod_{n\in\omega}X_n$
is the morphism with the components $\phi_n\:X_n\allowbreak\rarrow
X_{n+1}$, \,$n\in\omega$.
 As any cone in a triangulated category, the homotopy colimit is
defined uniquely up to a non-unique isomorphism.

 Dually, let $\sT^\sop$ be a triangulated category with countable
products, and let $Y_0\overset{\psi_0}\larrow Y_1\overset{\psi_1}
\larrow Y_2\larrow\dotsb$ be an $\omega^\sop$\+indexed diagram
in~$\sT^\sop$.
 Then the \emph{homotopy limit} $\holim_{n\in\omega}Y_n\in\sT^\sop$ is
defined as the object appearing in the distinguished triangle
$$
 \holim\nolimits_{n\in\omega}Y_n\lrarrow
 \prod\nolimits_{n\in\omega}Y_n
 \xrightarrow{\id-shift}
 \prod\nolimits_{n\in\omega}Y_n\lrarrow
 (\holim\nolimits_{n\in\omega}Y_n)[1].
$$
 Here $shift\:\prod_{n\in\omega}Y_n\rarrow\prod_{n\in\omega}Y_n$
is the morphism with the components $\psi_n\:Y_{n+1}\allowbreak\rarrow
Y_n$, \,$n\in\omega$.
 The homotopy limit of a sequence of objects and morphisms in
a triangulated category is defined uniquely up to a non-unique
isomorphism.

\begin{lem} \label{canonical-truncations-telescope-lemma}
 Let\/ $\sB$ be an additive category with countable products and
cokernels.
 Then every complex from\/ $\Hot^-(\sB)$ is a homotopy limit of
(uniformly bounded above) complexes from\/ $\Hot^\bb(\sB)$ in
the unbounded homotopy category\/ $\Hot(\sB)$.
\end{lem}

\begin{proof}
 Given a complex $B^\bu$ in $\sB$ and an integer $n\in\boZ$, denote
by $\tau_{\ge n}B^\bu$ the quotient complex of canonical truncation
$$
 \dotsb\lrarrow0\lrarrow 0\lrarrow\coker(B^{n-1}\to B^n)\lrarrow
 B^{n+1}\lrarrow B^{n+2}\lrarrow\dotsb
$$
 If the complex $B^\bu$ is bounded above, then the complex
$\tau_{\ge n}B^\bu$ is bounded (on both sides) for every $n\in\boZ$.
 Now we have an $\omega^\sop$\+indexed diagram of complexes
\begin{equation} \label{projective-system-of-canonical-truncations}
 \tau_{\ge0}B^\bu\llarrow\tau_{\ge-1}B^\bu\llarrow\tau_{\ge-2}B^\bu
 \llarrow\dotsb,
\end{equation}
which stabilizes into an eventually constant diagram $B^i$ at every
fixed cohomological degree $i\in\boZ$.
 Consider the related short sequence of complexes
\begin{equation} \label{telescope-of-canonical-truncations}
 0\lrarrow B^\bu\lrarrow\prod\nolimits_{n\in\omega}
 \tau_{\ge-n}B^\bu\xrightarrow{\id-shift}\prod\nolimits_{n\in\omega}
 \tau_{\ge-n}B^\bu\lrarrow0.
\end{equation}
 The degreewise eventual stabilization of the $\omega^\sop$\+indexed
diagram of canonical
truncations~\eqref{projective-system-of-canonical-truncations} implies
that the short sequence~\eqref{telescope-of-canonical-truncations} is
degreewise split exact.
 Therefore, we have $B^\bu=\holim_{n\in\omega}(\tau_{\ge-n}B^\bu)$
in $\Hot(\sB)$.
\end{proof}

\begin{lem} \label{plim-omega-telescope-in-one-degree-lemma}
 Let\/ $\sA$ be an abelian category and $\bC=\plim_{n\in\omega}C_n$
be an object of the abelian category\/ $\Pro_\omega(\sA)$.
 Then the short sequence
\begin{equation} \label{plim-omega-telescope}
 0\lrarrow\bC\lrarrow\prod\nolimits_{n\in\omega}^{\Pro_\omega(\sA)}C_n
 \xrightarrow{\id-shift}
 \prod\nolimits_{n\in\omega}^{\Pro_\omega(\sA)}C_n\lrarrow0
\end{equation}
is exact in the abelian category\/ $\Pro_\omega(\sA)$.
 Here the countable products of objects $C_n\in\sA\subset
\Pro_\omega(\sA)$ are taken in the category\/ $\Pro_\omega(\sA)$,
or equivalently, in the category\/ $\SPro_\omega(\sA)$.
 In particular, if $\bC\in\SPro_\omega(\sA)$, then the short
sequence~\eqref{plim-omega-telescope} is admissible exact
in\/ $\SPro_\omega(\sA)$.
\end{lem}

\begin{proof}
 For every $n\ge0$, consider the split short exact sequence
\begin{equation} \label{finite-telescope}
 0\lrarrow C_n\lrarrow\prod\nolimits_{m=0}^n C_m
 \xrightarrow{\id-shift}\prod\nolimits_{m=0}^{n-1}C_m\lrarrow0
\end{equation}
in the abelian category~$\sA$.
 Here $shift\:\prod_{m=0}^nC_m\rarrow\prod_{m=0}^{n-1}C_m$
is the morphism whose components are the transition maps
$C_m\rarrow C_{m-1}$, \,$0\le n\le m$, while
$C_n\rarrow\prod_{m=0}^nC_m$ is the morphism whose components
are the transition maps $C_n\rarrow C_m$.
 As the integer $n\ge0$ varies, the split short exact
sequences~\eqref{finite-telescope} form an $\omega^\sop$\+indexed
diagram with respect to the transition maps $C_{n+1}\rarrow C_n$
in the leftmost terms and the subproduct projections in
the middle and rightmost terms.

 Notice that $\plim_{n\in\omega}\prod_{m=0}^nC_m=
\prod_{n\in\omega}^{\Pro_\omega(\sA)}C_n=
\prod_{n\in\omega}^{\SPro_\omega(\sA)}C_n$ according to the explicit
proof of Corollary~\ref{products-of-pro-objects}(b).
 Applying the functor $\plim_{n\in\omega}\:\sA^{\omega^\sop}\lrarrow
\Pro_\omega(\sA)$ to the $\omega^\sop$\+indexed diagram of short exact
sequences~\eqref{finite-telescope} in $\sA$, we obtain the desired
short exact sequence~\eqref{plim-omega-telescope} in
$\Pro_\omega(\sA)$ (cf.\ Corollary~\ref{pro-objects-abelian}(b)).

 If $\bC\in\SPro_\omega(\sA)$, then all the terms of the short
exact sequence~\eqref{plim-omega-telescope} in $\Pro_\omega(\sA)$
belong to $\SPro_\omega(\sA)$, so it is an admissible short exact
sequence in $\SPro_\omega(\sA)$.
\end{proof}

\begin{lem} \label{plim-omega-telescope-for-complexes-lemma}
 Let\/ $\sA$ be an abelian category.
 Then any complex from\/ $\sD^\bb(\SPro_\omega(\sA))$ is a homotopy
limit of (uniformly bounded) complexes from\/ $\sD^\bb(\sA)$ in
the derived category\/ $\sD(\SPro_\omega(\sA))$.
\end{lem}

\begin{proof}
 First of all, we recall that countable products exist in
$\SPro_\omega(\sA)$ by Corollaries~\ref{products-of-pro-objects}(b)
and~\ref{products-preserve-strict-pro-objects}, and are exact
in the exact category structure of $\SPro_\omega(\sA)$ by
Corollary~\ref{products-exact-in-strict-pro-objects}.
 By~\cite[Lemma~1.5]{BN} or~\cite[Lemma~3.2.10]{Neem-book}, it
follows that the countable products exist in $\sD(\SPro_\omega(\sA))$
and the Verdier quotient functor $\Hot(\SPro_\omega(\sA))\rarrow
\sD(\SPro_\omega(\sA))$ preserves them.
 Furthermore, $\sD^\bb(\sA)$ is a full subcategory
in $\sD^\bb(\SPro_\omega(\sA))$ by
Corollary~\ref{fully-faithful-on-D-b-and-D-plus}(b).

 Let $\bC^\bu$ be a bounded complex (or more generally, a bounded
above complex) in $\Pro_\omega(\sA)$.
 Arguing as in the proof of
Lemma~\ref{filtered-limits-in-pro-objects}(b), one can show that
there exists a $\omega^\sop$\+indexed diagram $(C_n^\bu)_{n\in\omega}$
of uniformly bounded (respectively, uniformly bounded above)
complexes in $\sA$ such that $\plim_{n\in\omega}C_n^\bu=\bC^\bu$.
 Moreover, if $\bC^\bu$ is a complex in $\SPro_\omega(\sA)$,
then one can choose $(C_n^\bu)_{n\in\omega}$ to be a diagram of
degreewise epimorphisms of complexes in~$\sA$ (though we do not need
to use this observation).

 Assume that $\bC^\bu$ is a bounded complex in $\SPro_\omega(\sA)$,
and pick a related $\omega^\sop$\+indexed diagram of complexes
$(C_n^\bu)_{n\in\omega}$ as in the previous paragraph.
 By Lemma~\ref{plim-omega-telescope-in-one-degree-lemma}, the induced
short sequence of complexes
\begin{equation} \label{plim-omega-telescope-for-complexes}
 0\lrarrow\bC^\bu\lrarrow\prod
 \nolimits_{n\in\omega}^{\SPro_\omega(\sA)}C_n^\bu
 \xrightarrow{\id-shift}
 \prod\nolimits_{n\in\omega}^{\SPro_\omega(\sA)}C_n^\bu\lrarrow0
\end{equation}
is degreewise admissible exact in $\SPro_\omega(\sA)$.
 Therefore, we have $\bC^\bu=\holim_{n\in\omega}C_n^\bu$ in
$\sD(\SPro_\omega(\sA))$.
\end{proof}

\begin{cor} \label{d-minus-spro-omega-as-double-hocolims-of-d-b-a}
 Let\/ $\sA$ be an abelian category.
 Then any complex from\/ $\sD^-(\SPro_\omega(\sA))$ is a homotopy
limit of homotopy limits of (uniformly bounded above) complexes from\/
$\sD^\bb(\sA)$ in the derived category\/ $\sD(\SPro_\omega(\sA))$.
\end{cor}

\begin{proof}
 Notice that Lemma~\ref{canonical-truncations-telescope-lemma} is
applicable to $\sB=\SPro_\omega(\sA)$, because the full subcategory
$\SPro_\omega(\sA)$ is closed under quotients in the abelian
category $\Pro_\omega(\sA)$ by
Proposition~\ref{strict-pro-objects-closed-under-extensions}(b).
 Then compare the results of
Lemmas~\ref{canonical-truncations-telescope-lemma}
and~\ref{plim-omega-telescope-for-complexes-lemma}.
\end{proof}

\begin{lem} \label{silly-truncations-telescope-lemma}
 Let\/ $\sE$ be an additive category with countable coproducts.
 Then every complex from\/ $\Hot(\sE)$ is a homotopy colimit of
complexes from\/ $\Hot^+(\sE)$.
\end{lem}

\begin{proof}
 Given a complex $E^\bu$ in $\sE$ and an integer $n\in\boZ$,
denote by $\sigma_{\ge n}E^\bu$ the subcomplex of silly truncation
$$
 \dotsb\lrarrow0\lrarrow0\lrarrow E^n\lrarrow E^{n+1}\lrarrow
 E^{n+2}\lrarrow\dotsb
$$
 Now we have an $\omega$\+indexed diagram of complexes
$$
 \sigma_{\ge0}E^\bu\lrarrow\sigma_{\ge-1}E^\bu\lrarrow
 \sigma_{\ge-2}E^\bu\lrarrow\dotsb,
$$
which stabilizes into an eventually constant diagram $E^i$ at every
fixed cohomological degree $i\in\boZ$.
 The related short sequence of complexes
$$
 0\lrarrow\coprod\nolimits_{n\in\omega}\sigma_{\ge-n}E^\bu
 \xrightarrow{\id-shift}\coprod\nolimits_{n\in\omega}
 \sigma_{\ge-n}E^\bu\lrarrow E^\bu\lrarrow0
$$
is degreewise split exact.
 Consequently, $E^\bu=\hocolim_{n\in\omega}(\sigma_{\ge-n}E^\bu)$
in $\Hot(\sE)$.
\end{proof}

 Let us introduce some additional pieces of notation relevant to
uniformly bounded families of complexes.
 Let $m$~be an integer.
 For any exact category $\sB$, denote by $\sD^{\le m}(\sB)\subset
\sD^-(\sB)$ the full additive subcategory formed by all
the complexes concentrated in the cohomological degrees~$\le m$.
 For any additive category $\sE$, denote by $\Hot^{\ge-m}\subset
\Hot^+(\sB)$ the full additive subcategory formed by all
the complexes concentrated in the cohomological degrees~$\ge-m$.

 Finally, we can deduce the main result of this section.

\begin{prop} \label{cplxes-of-cntbly-generated-projs-are-generated}
 Let $\R$ be a complete, separated topologically right coherent
topological ring with a countable base of neighborhoods of zero.
 Then the minimal full triangulated subcategory of the homotopy
category\/ $\Hot(\R\Contra_\proj^\omega)$ containing the essential
image\/ $\sC$ of the functor\/ $\Xi\:\sD^\bb(\coh\R)^\sop\rarrow
\Hot(\R\Contra_\proj^\omega)$ and closed under countable coproducts
coincides with the whole category\/ $\Hot(\R\Contra_\proj^\omega)$.
\end{prop}

\begin{proof}
 We have equivalences of triangulated categories
\begin{equation} \label{d-minus-op-equiv-hot-minus-op-equiv-hot-plus}
 \sD^-(\Cohpro\R)^\sop\simeq\Hot^-(\Prod_\omega(\R))^\sop
 \simeq\Hot^+(\R\Contra_\proj^\omega),
\end{equation}
as per the discussion in Section~\ref{construction-of-Xi-secn}.
 By Corollary~\ref{d-minus-spro-omega-as-double-hocolims-of-d-b-a}
(applied to the abelian category $\sA=\coh\R$), every object of
$\sD^-(\Cohpro\R)$ can be obtained as a homotopy limit of homotopy
limits of uniformly bounded above objects from $\sD^\bb(\coh\R)$.

 Notice that the homotopy limits here are constructed in the unbounded
derived category $\sD(\Cohpro\R)$, which is \emph{not} anti-equivalent
to $\Hot(\R\Contra_\proj^\omega)$, generally speaking.
 However, all the actual countable products involved in the relevant
constructions are computed within the full subcategory of uniformly
bounded above complexes $\sD^{\le m}(\Cohpro\R)\subset\sD(\Cohpro\R)$
for some fixed $m\in\boZ$.

 The triangulated
equivalences~\eqref{d-minus-op-equiv-hot-minus-op-equiv-hot-plus}
transform such homotopy limits into the homotopy colimits in
$\Hot(\R\Contra_\proj^\omega)$, with all the countable coproducts
involved being computed within the full subcategory of uniformly
bounded below complexes $\Hot^{\ge-m}(\R\Contra_\proj^\omega)$
for some $m\in\boZ$.
 The latter assertion holds because the full subcategory
$\sD^{\le m}(\Cohpro\R)$ is closed under countable products in
$\sD(\Cohpro\R)$, the full subcategory
$\Hot^{\ge-m}(\R\Contra_\proj^\omega)$ is closed under countable
coproducts in $\Hot(\R\Contra_\proj^\omega)$, and the triangulated
anti-equivalence~\eqref{d-minus-op-equiv-hot-minus-op-equiv-hot-plus}
restricts to an anti-equivalence of additive categories
$\sD^{\le m}(\Cohpro\R)\simeq\Hot^{\ge-m}(\R\Contra_\proj^\omega)$.

 We have shown that every object of $\Hot^+(\R\Contra_\proj^\omega)$
can be obtained from the objects of $\sC$ using cones and countable
coproducts.
 In order to conclude that all the objects of
$\Hot(\R\Contra_\proj^\omega)$ can be so obtained, it remains
to refer to Lemma~\ref{silly-truncations-telescope-lemma} (for
the additive category $\sE=\R\Contra_\proj^\omega$).
\end{proof}

\Section{Compact Generation and Flat/Projective Periodicity}
\label{compact-generation-secn}

 The following very general result is an application of the theory
of well-generated triangulated categories~\cite{Neem-book,Kra2}.
 It is a part of~\cite[Proposition~4.9]{SaoSt}.
 For a basic background discussion, see~\cite[Sections~6\+-7]{PS4}.

\begin{prop} \label{well-generated-prop}
 Let $\kappa$~be a regular cardinal and\/ $\sB$ be a locally
presentable abelian category with a $\kappa$\+presentable
projective generator\/ $P\in\sB$.
 Denote by\/ $\sU\subset\Hot(\sB_\proj)$ the full subcategory of all
bounded below complexes whose terms are coproducts of less than~$\kappa$
copies of~$P$.
 Then the minimal full triangulated subcategory of\/ $\Hot(\sB_\proj)$ 
containing\/ $\sU$ and closed under coproducts coincides with the whole
homotopy category of projective objects\/ $\Hot(\sB_\proj)$.
\end{prop}

\begin{proof}
 In the notation of~\cite{SaoSt}, put $\cA=\sB$, \ $\cB=\sB_\proj$,
and $\cS=\{P\}$.
 Then $\cU=\sU$.
 The projective objects of $\sB$ are precisely all the direct summands
of the coproducts of copies of $P$ in $\sB$, so condition~(1)
of~\cite[Proposition~4.9]{SaoSt} is satisfied.
 Thus the equivalent condition~(3) of the same proposition
from~\cite{SaoSt} holds as well.
\end{proof}

 Now we can prove the main result of this paper.

\begin{thm} \label{compact-generators-of-contraderived-category}
 Let\/ $\R$ be a complete, separated, right linear, topologically
right coherent topological ring with a countable base of neighborhoods
of zero.
 Then the contraderived category\/ $\sD^\bctr(\R\Contra)$ is compactly
generated.
 The fully faithful triangulated functor\/ $\Xi\:\sD^\bb(\coh\R)^\sop
\rarrow\sD^\bctr(\R\Contra)$ from Section~\ref{construction-of-Xi-secn}
establishes a triangulated anti-equivalence between the bounded derived
category of coherent discrete right\/ $\R$\+modules $\sD^\bb(\coh\R)$
and the full subcategory of compact objects in the triangulated
category\/ $\sD^\bctr(\R\Contra)$.
\end{thm}

\begin{proof}
 It is convenient to identify $\sD^\bctr(\R\Contra)$ with
$\Hot(\R\Contra_\proj)$, as per
Theorem~\ref{becker-contraderived-theorem}.
 The contravariant triangulated functor $\Xi$ is fully faithful by
construction (see Section~\ref{construction-of-Xi-secn}), and
the objects in its image are compact in $\sD^\bctr(\R\Contra)$ by
Corollary~\ref{image-of-Xi-consists-of-compacts}.
 Let $\sC$ be the essential image of $\Xi$, as in
Section~\ref{countably-generated-projectives-secn}.
 According to
Proposition~\ref{cplxes-of-cntbly-generated-projs-are-generated},
all complexes of countably generated projective objects in $\R\Contra$
can be obtained from the objects of $\sC$ using cones and countable
coproducts, up to the homotopy equivalence.

 Put $\sB=\R\Contra_\proj$ and $P=\R[[\{{*}\}]]=\R\in\R\Contra$.
 Then $P$ is an $\aleph_1$\+presentable projective generator of
the locally $\aleph_1$\+presentable abelian category~$\sB$.
 So Proposition~\ref{well-generated-prop} is applicable for
$\kappa=\aleph_1$, and it tells us that, up to the homotopy
equivalence, all complexes of projective objects in $\R\Contra$ can be
obtained from bounded below complexes of free $\R$\+contramodules with
countable sets of generators using cones and coproducts.
 Thus the minimal full triangulated subcategory of
$\Hot(\R\Contra_\proj)$ containing $\sC$ and closed under coproducts
coincides with the whole homotopy category $\Hot(\R\Contra_\proj)$.

 We have shown that the image of $\Xi$ forms a set of compact
generators in $\sD^\bctr(\R\Contra)$.
 So condition~(1) of
Proposition~\ref{four-equivalent-conditions-prop} is satisfied.
 Hence the equivalent condition~(2) holds as well, proving
the desired description of the full subcategory of compact objects
in $\sD^\bctr(\R\Contra)$.
\end{proof}

 Finally, we can apply Proposition~\ref{four-equivalent-conditions-prop}
in order to deduce our contramodule version of the Benson--Goodearl
flat/projective periodicity theorem~\cite[Theorem~2.5]{BG} and its
extension due to
Neeman~\cite[Theorem~8.6\,(i)\,$\Leftrightarrow$\,(iii)
and Remark~2.15]{Neem}.

\begin{thm} \label{extended-flat-projective-periodicity}
 Let\/ $\R$ be a complete, separated, right linear, topologically
right coherent topological ring with a countable base of neighborhoods
of zero.
 In this context: \par
\textup{(a)} Let\/ $\fP^\bu$ be an acyclic complex of projective left\/
$\R$\+contramodules with flat\/ $\R$\+contramodules of cocycles.
 Then the\/ $\R$\+contramodules of cocycles are actually projective
(so the complex\/ $\fP^\bu$ is contractible). \par
\textup{(b)} Let\/ $\fP^\bu$ be a complex of projective left\/
$\R$\+contramodules and\/ $\fF^\bu$ be an acyclic complex of flat
left\/ $\R$\+contramodules with flat\/ $\R$\+contramodules of cocycles.
 Then any morphism of complexes\/ $\fP^\bu\rarrow\fF^\bu$ is homotopic
to zero.
 In other words, any acyclic complex of flat left\/ $\R$\+contramodules
with flat\/ $\R$\+contramodules of cocycles is contraacyclic. \par
\textup{(c)} Conversely, let\/ $\fF^\bu$ be a complex of flat left\/
$\R$\+contramodules such that, for any complex of projective left\/
$\R$\+contramodules\/ $\fP^\bu$, any morphism of complexes\/ $\fP^\bu
\rarrow\fF^\bu$ is homotopic to zero.
 Then the complex\/ $\fF^\bu$ is acyclic with flat\/ $\R$\+contramodules
of cocycles.
\end{thm}

\begin{proof}
 Part~(c) has been already proved as
Theorem~\ref{flat-contraacyclic-are-pure-acyclic}; its proof does not
use the results of
Sections~\ref{countably-generated-projectives-secn}\+-%
\ref{compact-generation-secn}.
 To prove parts~(a) and~(b), we point out that conditions~(1) and/or~(2)
of Proposition~\ref{four-equivalent-conditions-prop} are satisfied
by Theorem~\ref{compact-generators-of-contraderived-category}.
 Therefore, the equivalent conditions~(3) and~(4) hold as well.
\end{proof}

\begin{cor}
 Let\/ $\R$ be a complete, separated topologically right coherent
topological ring with a countable base of neighborhoods of zero.
 Then there are natural equivalences of triangulated categories
\begin{equation} \label{contraderived-equivalent-constructions}
 \Hot(\R\Contra_\proj)\simeq\sD(\R\Contra_\flat)\simeq
 \sD^\bctr(\R\Contra).
\end{equation}
 Here\/ $\sD(\R\Contra_\flat)$ is the derived category of the exact
category flat left\/ $\R$\+con\-tra\-mod\-ules, with the exact category
structure on\/ $\R\Contra_\flat$ inherited from the abelian exact
structure of the ambient abelian category\/ $\R\Contra$.
\end{cor}

\begin{proof}
 First of all, let us point out that the full subcategory
$\R\Contra_\flat$ is closed under extensions in $\R\Contra$ by
Lemma~\ref{flat-contramodules-closure-properties}, so it inherits
an exact category structure.
 The triangulated equivalence $\Hot(\R\Contra_\proj)\rarrow
\sD^\bctr(\R\Contra)$ is provided by
Theorem~\ref{becker-contraderived-theorem}.
 The nontrivial aspect of the latter theorem is the assertion
that, for any complex of left $\R$\+contramodules $\fM^\bu$, there
exists a complex of projective left $\R$\+contramodules $\fP^\bu$
together with a morphism of complexes $\fP^\bu\rarrow\fM^\bu$ with
a contraacyclic cone.
 Notice that any projective left $\R$\+contramodule is flat.
 In view of these observations, the assertion that the inclusions of
additive/exact/abelian categories $\R\Contra_\proj\rarrow
\R\Contra_\flat\rarrow\R\Contra$ induce the desired triangulated
equivalences~\eqref{contraderived-equivalent-constructions} follows
immediately from
Theorem~\ref{extended-flat-projective-periodicity}(b\+-c).
\end{proof}

\begin{rem}
 Let us explain how our
Theorem~\ref{extended-flat-projective-periodicity} compares to
the preceding results of~\cite[Proposition~12.1]{Pflcc}
and~\cite[Theorems~5.1 and~6.1]{Pbc}.
 Theorem~\ref{extended-flat-projective-periodicity}(a) is a version
of~\cite[Proposition~12.1]{Pflcc},
Theorem~\ref{extended-flat-projective-periodicity}(b) is a version
of~\cite[Theorem~5.1]{Pflcc}, and
Theorem~\ref{extended-flat-projective-periodicity}(c) is a version
of~\cite[Theorem~6.1]{Pflcc}.

 We would like to emphasize that the assertions of
Theorem~\ref{extended-flat-projective-periodicity} are \emph{both
more and less general} than the respective assertions
from~\cite{Pflcc,Pbc}.
 Theorem~\ref{extended-flat-projective-periodicity} is more general
in the important aspect that it only presumes a topology base of
open \emph{right} ideals in the ring $\R$, while the results
of~\cite{Pflcc,Pbc} assume a topology base of open
\emph{two-sided} ideals.
 However, Theorem~\ref{extended-flat-projective-periodicity} is
also less general in that it is based on the topological right
\emph{coherence} assumption, which was not made in~\cite{Pflcc,Pbc}.
\end{rem}

\end{document}